\documentclass{amsart}
\usepackage[margin=1in]{geometry}
\usepackage{amsthm}
\usepackage{amssymb}
\usepackage[non-compressed-cites]{amsrefs}
\usepackage{enumitem}
\usepackage{mathtools}
\usepackage{xcolor}
\usepackage{hyperref}

\allowdisplaybreaks

\newcommand{\R}{\mathbb{R}}
\newcommand{\Sph}{\mathbb{S}}
\newcommand{\del}{\partial}

\renewcommand{\div}{\operatorname{div}}
\newcommand{\Rm}{\operatorname{Rm}}
\newcommand{\tr}{\operatorname{tr}}

\newcommand{\dvol}{\mathrm{dvol}}
\newcommand\inner[2]{\left\langle #1, #2 \right\rangle}
\newcommand{\norm}[1]{\left\lVert #1 \right\rVert}
\renewcommand{\O}{\mathcal{O}}
\renewcommand{\L}{\mathcal{L}}
\renewcommand{\H}{\mathcal{H}}

\newcommand{\B}{\mathcal{B}}
\newcommand{\N}{\mathbb{N}}
\newcommand{\V}{\mathcal{V}}
\renewcommand{\P}{\mathcal{P}}
\newcommand{\U}{\mathcal{U}}
\newcommand{\Q}{\mathcal{Q}}

\newcommand{\red}[1]{\textcolor{red}{#1}}
\newcommand{\blue}[1]{\textcolor{blue}{#1}}
\newcommand{\magenta}[1]{\textcolor{magenta}{#1}}

\newtheorem{theorem}{Theorem}[section]
\newtheorem{corollary}[theorem]{Corollary}
\newtheorem{lemma}[theorem]{Lemma}
\newtheorem{definition}[theorem]{Definition}
\newtheorem{proposition}[theorem]{Proposition}
\newtheorem{assumption}[theorem]{Assumption}
\newtheorem{example}[theorem]{Example}
\newtheorem{remark}[theorem]{Remark}
\numberwithin{equation}{section}
 
\hypersetup{
    colorlinks,
    linkcolor={red!60!black},
    citecolor={blue!80!black},
}

\setlist[enumerate,1]{font=\normalfont}

\mathtoolsset{showonlyrefs}

\title[Drift-harmonic functions on asymptotic paraboloids]{Drift-harmonic functions with polynomial growth on asymptotically paraboloidal manifolds}
\author[Michael B. Law]{Michael B. Law}
\address{MIT, Department of Mathematics, 77 Massachusetts Avenue, Cambridge, MA 02139, USA.}
\email{mikelaw@mit.edu}

\begin{document}

\begin{abstract}
    We construct and classify all polynomial growth solutions to certain drift-harmonic equations on complete manifolds with paraboloidal asymptotics. These encompass the natural drift-harmonic equations on certain steady gradient Ricci solitons. Specifically, we show that all drift-harmonic functions with polynomial growth asymptotically separate variables, and compute the dimensions of spaces of drift-harmonic functions with a given polynomial growth rate. The proof uses an inductive argument that alternates between constructing and asymptotically controlling drift-harmonic functions.
\end{abstract}

\maketitle

\setcounter{tocdepth}{1}
\tableofcontents

\section{Introduction} \label{sec:intro}

\subsection{Motivation and the main result}

Given a Riemannian manifold with symmetry and a PDE that respects this symmetry, one can often separate variables to find and classify solutions. A standard example is the classification of harmonic functions in Euclidean space and Riemannian cones (see for example \cite{cm97a}*{Theorem 1.11}). If the symmetry only holds asymptotically, one expects solutions of the PDE to asymptotically resemble solutions obtained via separation of variables on the model space. Although the intuition is clear, it takes significant effort to rigorously justify. This has been done for certain classes of elliptic operators on asymptotically Euclidean \cites{lockhart-euc,mcowen,cscb}, cylindrical \cites{lockhart,lockhart-mcowen}, and hyperbolic manifolds \cite{jlee-mem}. A typical approach is to build a Fredholm theory for elliptic operators on weighted spaces adapted to the underlying geometry.

In this paper, we introduce a more hands-on approach to solving certain PDEs on manifolds asymptotic to model spaces. We consider Riemannian manifolds asymptotic to paraboloids, equipped with a \emph{drift Laplacian} $\L_f = \Delta - \nabla_{\nabla f}$ where $f$ grows like minus the distance function. As we will see, such manifolds and potentials naturally arise when studying certain Ricci solitons.

Our main theorem constructs and classifies all drift-harmonic functions ($\L_f u = 0$) with polynomial growth. To state this, we define an \emph{asymptotically paraboloidal} (AP) manifold $(M^n,g,r)$ to be a Riemannian manifold $(M^n,g)$ with a function $r: M \to [0,\infty)$ resembling the distance from the vertex of a paraboloid; see Definition \ref{def:asymp-parab}. It will be shown that $g$ is asymptotic to a paraboloidal metric $dr^2 + r g_X$, where $g_X$ is a metric on a closed manifold $\Sigma^{n-1}$. The distinct eigenvalues of $-\Delta_{g_X}$ on $\Sigma$ are denoted $0 = \lambda_1 < \lambda_2 < \cdots \to \infty$, with finite multiplicities $1=m_1,m_2,\cdots$. Let $f \in C^\infty(M)$ be a smooth function satisfying the following:
\begin{assumption} \label{assump:f}
    Outside a compact set, $f$ depends only on $r$ and we write $f = f(r)$ there. We assume
    \begin{align}
        f'(r) = -1 + \O(r^{-1}), \quad f''(r) = \O(r^{-\frac{3}{2}}), \quad f'''(r) = \O(r^{-\frac{3}{2}}), \quad \text{as } r \to \infty.
    \end{align}
\end{assumption}
Define the drift Laplacian on functions $u: M \to \R$ by
\begin{align} \label{eq:drift-laplacian}
    \L_f u = \Delta u - \inner{\nabla f}{\nabla u}, 
\end{align}
as well as the following spaces of drift-harmonic functions for each $d \in \R$:
\begin{align}
    \H &= \{u: M \to \R \mid \L_f u = 0\}, \label{eq:harm1} \\
    \H_d &= \{u: M \to \R \mid \L_f u = 0, \text{ and } |u| \leq C(r^d+1) \text{ for some } C>0\}, \label{eq:harm2} \\
    \H_d^+ &= \bigcap_{\epsilon > 0} \H_{d+\epsilon}, \label{eq:harm3} \\
    \mathring{\H}_d &= \H_d \setminus \bigcup_{\epsilon > 0} \H_{d-\epsilon}. \label{eq:harm4}
\end{align}
\noeqref{eq:harm2}
\noeqref{eq:harm3}

Our main theorem, stated next, uses the notion of \emph{asymptotically separating variables}, which measures how close a function asymptotically gets to separating variables along directions parallel and orthogonal to $\nabla r$. It also refers to \emph{asymptotic orthogonality}, which quantifies how close two functions are to being $L^2$-orthogonal on level sets of $r$ as $r \to \infty$. These notions were introduced by Colding and Minicozzi in \cite{cm97a}; we respell them out in our context in Definitions \ref{def:almost-sep} and \ref{def:almost-orthog}.

\begin{theorem}[Main theorem] \label{thm:main-expanded}
    Let $(M^n,g,r)$ be an AP manifold of dimension $n \geq 3$, and let $f \in C^\infty(M)$ satisfy Assumption \ref{assump:f}. For each $d \in \R$ and $u \in \H_d$:
    \begin{enumerate}[label=(\alph*)]
        \item There exist $C,\tau > 0$ such that $u$ $(C,\tau)$-asymptotically separates variables.
        \item There exists $\ell \in \N$ such that $u \in \mathring{\H}_{\lambda_\ell}$.
        \item The dimension of $\H_d$ is finite, with
        \begin{align}
            \dim \H_d = \sum_{\{k \in \N : \lambda_k \leq d \}} m_k.
        \end{align}
        \item There exist $C,\tau > 0$ and a basis $\B_d$ for $\H_d$ such that every distinct pair of functions $u,v \in \B_d$ is $(C,\tau)$-asymptotically orthogonal.
    \end{enumerate}
\end{theorem}

When $g = dr^2 + r g_X$ outside a compact set, Theorem \ref{thm:main-expanded} can be proved by separating variables; see Appendix \ref{sec:appA}. The general case is much subtler and its proof employs an iterative scheme; see \S\ref{subsec:elements-intro}.

Dimension counting for spaces of harmonic and drift-harmonic functions is a well-studied topic; see e.g. \cites{cm97a,cm97b,huangdimensions,huang-crelle,cm-weyl,ding,xu,wu2,wuou-crelle,mw-smms}. Theorem \ref{thm:main-expanded} is yet another contribution in this direction. Our result additionally describes the asymptotic behavior of all polynomial growth drift-harmonic functions, addressing an aspect often overlooked, except in \cites{cm97a,bernstein}. The sharp growth rate in (b) also contrasts with \cites{ding,huangharmonic,xu} where harmonic functions are only found to be in $\H_d^+$ (for appropriate $d$). Also note that our result contains a Liouville theorem: if $u \in \H$ and $u = \O(r^d)$ for some $d < \lambda_2$ as $r \to \infty$, then $u$ is constant.

Answers to many geometric problems rest on a solid understanding the behavior of solutions to PDEs on manifolds asymptotic to model spaces. Proofs of positive mass theorems in general relativity usually involve solving a Poisson equation at infinity on asymptotically Euclidean manifolds (e.g. \cites{schoen-yau-pmt,witten-pmt}). Rigidity and uniqueness results for asymptotically cylindrical and conical Ricci solitons have been established by understanding the behavior of almost-Killing vector fields, which satisfy a drift-harmonic equation up to errors \cites{brendle,brendleHigher,chodosh}. In mean curvature flow, uniqueness of asymptotically conical shrinkers follows from characterizing the asymptotic structure of almost-eigenfunctions of the drift Laplacian \cite{bernstein}. Dimension counting for harmonic and drift-harmonic functions has also been used to probe the topology of ends of Ricci solitons and manifolds with nonnegative Ricci curvature \cites{mw-smms,li-tam}.

The present work also originates from geometric considerations. Many examples of steady gradient Ricci solitons (like the Bryant soliton), together with their soliton potentials $f$, satisfy the hypotheses of Theorem \ref{thm:main-expanded} (see Example \ref{ex:bryant-AP}). The hypotheses are also met by any weighted Riemannian manifold sufficiently asymptotic to one of these. As the kernel and spectrum of $\L_f$ encode rich geometric information about weighted spaces, we envision that our result and method lay down the analytical basis for a deeper understanding of these solitons. A major obstacle which has impeded the analysis of steady gradient Ricci solitons, and which our method overcomes, is the exponential growth of the natural weighted volume $e^{-f} \, \dvol_g$ \cite{wupotential}. This severely limits the applicability of the weighted $L^2$ theory, despite the availability of integral estimates for drift-harmonic functions such as Poincar\'e and mean value inequalities \cites{mw-smms,wuwu-heat,chara}.

\subsection{Overview of the proof} \label{subsec:elements-intro}

The overarching strategy behind proving Theorem \ref{thm:main-expanded} is to iteratively apply two steps:
\begin{itemize}
    \item \textbf{Asymptotic control (A):} Given a large enough collection $\B_{\lambda_\ell} \subset \H_{\lambda_\ell}$, we establish asymptotic control on all drift-harmonic functions in $\H_{\lambda_{\ell+1}}^+$. This will also show that $\B_{\lambda_\ell}$ is a basis for $\H_{\lambda_\ell}$.
    \item \textbf{Construction (C):} We then construct a sufficiently large collection $\B_{\lambda_{\ell+1}} \subset \H_{\lambda_{\ell+1}}^+$. The asymptotic control from (A) will help show that in fact $\B_{\lambda_{\ell+1}} \subset \H_{\lambda_{\ell+1}}$.
\end{itemize}
Starting with $\B_{\lambda_1} = \{1\}$, Theorem \ref{thm:main-expanded} follows from iterating (A) and (C) and using that $\lambda_\ell \to \infty$. See \S\ref{sec:dividing} for details. Here we will just outline some key tools involved in establishing (A) and (C).

\subsubsection{Frequency functions} \label{subsubsec:freq-intro}

Frequency functions, introduced by Almgren \cite{almgren}, have proved successful for studying the growth of solutions to PDEs \cites{bernstein,cm97a,cm-sharpfreq,cm-sing}.
The frequency function $U_u(\rho)$ of a function $u$ measures the polynomial growth rate of $u$ with respect to $r$ at the scale $\{r=\rho\}$. For a drift-harmonic $u \in \H$, we begin in the likes of \cites{cm97a,cm-sharpfreq} by deriving a nonlinear ODE for $U_u$, which asymptotically reads
\begin{align} \label{eq:approximate-ODE}
    U_u' \approx Q_u - U_u,
\end{align}
where $Q_u(\rho)$ is approximately the (normalized) Rayleigh quotient of $u$ over $\{r=\rho\}$:
\begin{align}
    Q_u(\rho) \approx \frac{\rho \int_{\{r=\rho\}} |\nabla^\top u|^2}{\int_{\{r=\rho\}} u^2}.
\end{align}
Roughly, \eqref{eq:approximate-ODE} says that the polynomial growth rate of $u$ is the Rayleigh quotient of $u$ over $\{r=\rho\}$.

\subsubsection{Preservation of almost orthogonality}

We show that if two drift-harmonic functions $u,v \in \H_d$ are almost orthogonal on $\{r=\rho\}$, i.e. $\int_{\{r=\rho\}} uv \approx 0$, then this remains so on $\{r=2\rho\}$; see Corollary \ref{cor:pres-of-inner-prod}. This phenomenon was observed in \cite{cm97a} to hold for harmonic functions on manifolds with nonnegative Ricci curvature and maximal volume growth. It seems to not have been used elsewhere. We revive this idea and introduce a new iterative way to apply it.

Namely, given an adequate collection $\B_{\lambda_\ell} \subset \H_{\lambda_\ell}$, and any $u \in \H_{\lambda_{\ell+1}}^+$ outside the span of $\B_{\lambda_\ell}$, we apply preservation of almost orthogonality between $u$ and each $v \in \B_{\lambda_\ell}$, iteratively out to infinity. The errors gained in each iteration are summable due to the power-rate asymptotics in our definition of an AP manifold. This produces a lower bound on $Q_u$, and hence (by \S\ref{subsubsec:freq-intro}) on $U_u$. This is the starting point for (A).

\subsubsection{Constructing drift-harmonic functions} \label{subsubsec:construction-intro}

Our construction in (C) is based on existing constructions of harmonic functions on manifolds with nonnegative Ricci curvature by Ding \cite{ding}, Huang \cite{huangdimensions}, and Xu \cite{xu}. We start by solving a sequence of Dirichlet problems $u_i$ on increasing domains of $M$, i.e.
\begin{align} \label{eq:increasing-dirichlet}
    \begin{cases}
        \L_f u_i = 0 & \text{in } M \setminus \{r \geq 2^i\}, \\
        u_i = \Theta_i & \text{on } \{r=2^i\},
    \end{cases}
\end{align}
for appropriate boundary data $\Theta_i$.

To find a convergent subsequence $u_i \to u \in \H$, we construct barriers to obtain uniform boundary estimates for the $u_i$, and prove a so-called \emph{three circles} theorem to propagate the boundary estimates inward. In \cites{ding,xu}, these are established by exploiting the approximate dilation-invariance of the geometry as well as the equation $\Delta u = 0$. We will generalize their methods to AP manifolds and the equation $\L_f u = 0$. This requires delicate scaling arguments as AP manifolds are only approximately invariant under \emph{anisotropic} dilations (unlike manifolds with nonnegative Ricci curvature, which possess tangent cones at infinity).

Then, we construct $u \in \H_{\lambda_{\ell+1}}^+$ following \cites{ding,huangdimensions,xu}. The asymptotic control (A) improves this to $u \in \H_{\lambda_{\ell+1}}$. Repeating this construction enough times establishes (C) above.

\subsection*{Organization}

In \S\ref{sec:asymp-parab} we define AP manifolds and establish basic geometric properties. In \S\ref{sec:analytic}, we develop tools for studying drift-harmonic functions on AP manifolds, such as frequency functions and preservation of almost orthogonality. In \S\ref{sec:dividing}, we reduce Theorem \ref{thm:main-expanded} to asymptotic control (A) and construction (C) steps. These are Theorems \ref{thm:asymp-ctrl} and \ref{thm:construction} respectively, and are proved in \S\ref{sec:asymp-ctrl} and \S\ref{sec:construction}.

Appendix \ref{sec:appA} proves a model case of Theorem \ref{thm:main-expanded} which is independent from the rest of the paper. In Appendix \ref{sec:appB}, we obtain second-order control of the metric of an AP manifold. In Appendix \ref{sec:appC}, we prove estimates for drift-harmonic functions which are stated in \S\ref{subsec:pointwise-estimates}.

\subsection*{Acknowledgements}

The author thanks William Minicozzi for his interest in this work and for numerous insightful discussions. The author is supported by a Croucher Scholarship.

\section{Geometry of asymptotically paraboloidal manifolds} \label{sec:asymp-parab}

\subsection{Definition and examples of AP manifolds} \label{subsec:AP}

Our definition of an AP manifold is inspired by Bernstein's \cite{bernstein} notion of a weakly conical end. It asks for a function $r$ resembling the distance from the vertex of a paraboloid. The tensor $\eta$ in (ii) vanishes when the metric is exactly $dr^2 + rh$ for some metric $h$ on $\Sigma$.

\begin{definition} \label{def:asymp-parab}
    An \textbf{asymptotically paraboloidal (AP) manifold} $(M^n,g,r)$ is a complete, oriented, smooth Riemannian manifold $(M^n,g)$ of dimension $n \geq 3$ equipped with a smooth proper unbounded function $r: M \setminus K \to [R_0,\infty)$ defined outside a compact set $K \subset M$, such that the following hold for some $\mu > 0$.
    \begin{enumerate}[label=(\roman*)]
        \item $||\nabla r| - 1| = \O(r^{-\mu})$ as $r \to \infty$. Thus by enlarging $K$ and $R_0$ if needed, all the level sets $\{r=\rho\}$, for $\rho \geq R_0$, are smoothly diffeomorphic to a closed manifold $\Sigma$ of dimension $n-1$.
        \item The symmetric $2$-tensor $\eta := \nabla^2 r^2 - g - dr^2$ satisfies
        \begin{align}
            |\eta| &= \O(r^{-\mu}), \quad
            |\nabla\eta| = \O(r^{-1}),\quad \text{as } r \to \infty.
        \end{align}
        \item For each $\rho \geq R_0$, let $g_{\Sigma_\rho}$ be the metric on $\Sigma$ induced by the restricting $g$ to the level set $\{r=\rho\} \cong \Sigma$. Then we require the metrics $g_X(\rho) := \rho^{-1} g_{\Sigma_\rho}$ on $\Sigma$ to satisfy
        \begin{align}
            \sup_{\rho \geq R_0} \norm{g_X(\rho)}_{C^2(\Sigma)} < \infty,
        \end{align}
    where the $C^2$ norm is taken with respect to a background Riemannian metric on $\Sigma$.
    \end{enumerate}
    Here $\nabla$ is the Levi-Civita connection of $g$, and $|\cdot|$ denotes tensor norms with respect to $g$.
\end{definition}

We will frequently use two basic computations on AP manifolds. These are
\begin{align} \label{eq:hess-r}
    \nabla^2 r &= \frac{1}{2r}(\nabla^2 r^2 - 2dr^2) = \frac{1}{2r}(g-dr^2+\eta),
\end{align}
which further implies
\begin{align}
    \inner{\nabla|\nabla r|}{\nabla r} &= \frac{1}{|\nabla r|} \nabla^2 r(\nabla r, \nabla r) = \frac{1}{2r|\nabla r|}(|\nabla r|^2 - |\nabla r|^4 + \eta(\nabla r,\nabla r)) = \O(r^{-\mu-1}). \label{eq:grad|gradr|}
\end{align}

\begin{example} \label{ex:bryant-AP}
    We show that the steady Bryant soliton \cite{bryant} is AP and that its soliton potential $f$ satisfies Assumption \ref{assump:f}. A similar analysis holds for any of the complete steady Ricci solitons found by Ivey \cite{ivey} and Dancer--Wang \cite{dancerwang} which generalize the Bryant soliton. As a result, the same assertions are true for any manifolds and potential functions suitably asymptotic to these examples.
    
    The Bryant soliton $(M^n,g)$ of dimension $n \geq 3$ is a warped product
    \begin{align}
        g = dr^2 + w(r)^2 g_{\Sph^{n-1}},
    \end{align}
    where $w: (0,\infty) \to (0,\infty)$ has the following asymptotics as $r \to \infty$ (see e.g. \cite{chowsolitons}*{\S 6}):
    \begin{align} \label{eq:w'w''-expansion}
        w'(r) &= \sqrt{\frac{n-2}{2r}} + \O(r^{-3/2} \log r) = \O(r^{-1/2}), \quad w''(r) = \O(r^{-3/2}).
    \end{align}
    Integrating the expansion for $w'(r)$ gives
    \begin{align} \label{eq:w-expansion}
        w(r) &= \sqrt{2(n-2)r} + \O(1) = \O(r^{1/2}).
    \end{align}
    In particular,
    \begin{align} \label{eq:2rw'w-expansion}
        \frac{2rw'}{w} - 1 = \O(r^{-1/2}).
    \end{align}
    As $g$ is a warped product, one has $\nabla^2 r = ww' g_{\Sph^{n-1}} = \frac{w'}{w}(g-dr^2)$ (see e.g. \cite{petersen}*{\S 4.2.3}). Using \eqref{eq:w'w''-expansion}, \eqref{eq:w-expansion}, and that $|g| = n$, $|dr^2| = 1$, this gives
    \begin{align} \label{eq:nabla2r-expansion}
        |\nabla^2 r| = \O(r^{-1}).
    \end{align}
    We also compute
    \begin{align} \label{eq:eta-bryant}
        \eta &= dr^2 + 2r\nabla^2 r - g = \left( \frac{2rw'}{w} - 1 \right) (g-dr^2), \\
        \nabla \eta &= \left( \frac{2w'}{w} + \frac{2rw''}{w} - \frac{2r(w')^2}{w^2} \right) dr \otimes (g-dr^2) - \left( \frac{2rw'}{w} - 1 \right) (\nabla^2 r \otimes dr + dr \otimes \nabla^2 r).
    \end{align}
    So by \eqref{eq:w'w''-expansion}, \eqref{eq:w-expansion},  \eqref{eq:2rw'w-expansion} and \eqref{eq:nabla2r-expansion}, we get
    \begin{align} \label{eq:nablaeta-bryant}
        |\eta| = \O(r^{-1/2}), \quad |\nabla \eta| = \O(r^{-1}).
    \end{align}
    The normalized level set metrics of $g$ are $g_X(\rho) = \frac{w(\rho)^2}{\rho} g_{\Sph^{n-1}}$. By \eqref{eq:w-expansion}, $\frac{w(\rho)^2}{\rho}$ is bounded as $\rho \to \infty$, so
    \begin{align} \label{eq:condition3-bryant}
        \norm{g_X(\rho)}_{C^2(\Sph^{n-1})} < \infty.
    \end{align}
    From \eqref{eq:nablaeta-bryant}, \eqref{eq:condition3-bryant}, and the obvious fact that $|\nabla r| = 1$, it follows that $(M^n,g,r)$ is AP.

    Recall (from \cite{chowsolitons}*{\S 6} for instance) that $f$ and the scalar curvature $R$ are decreasing functions of $r$, with $R = \O(r^{-1})$ and $R' = \O(r^{-3/2})$; moreover
    \begin{align} \label{eq:bryant-eqns}
        \L_f f = \Delta f - |\nabla f|^2 = -1, \quad R + \Delta f = 0, \quad f'' = (n-1)\frac{w''}{w}.
    \end{align}
    Using the last identity with \eqref{eq:w'w''-expansion}, \eqref{eq:w-expansion}, we have $f'' = \O(r^{-2})$. Also, $(f')^2 = |\nabla f|^2 = 1 - \Delta f = 1 - R = 1 + \O(r^{-1})$. Taking roots and using that $f$ is decreasing, this implies $f' = -1 + \O(r^{-1})$. Finally, differentiating $R+\Delta f = 0$ and using that $\Delta f = f'' + \frac{(n-1)w'}{w}f'$, we get
    \begin{align} \label{eq:R'f'''form}
        0 = R' + f''' + \frac{(n-1)w''}{w}f' - \frac{(n-1)(w')^2}{w^2}f' + \frac{(n-1)w'}{w}f''.
    \end{align}
    Using the decay $R' = \O(r^{-3/2})$ with the asymptotics for $f',f'',w,w',w''$, we get $f''' = \O(r^{-3/2})$. Hence, $f$ satisfies Assumption \ref{assump:f} with respect to $r$.
\end{example}

\subsection{Conventions} \label{subsec:conventions}

Throughout the paper, $(M^n,g,r)$ will denote a fixed AP manifold. Let $R_0, \mu, \Sigma, \eta, g_{\Sigma_\rho}$ and $g_X(\rho)$ be as in Definition \ref{def:asymp-parab}. Note that the hypersurface $\{r=\rho\} \subset M$ is isometric to $(\Sigma, g_{\Sigma_\rho})$. We also adopt the following conventions:
\begin{itemize}
    \item We will assume that $R_0 = 0$, which may be achieved by adding a global constant to $r$.
    \item As $\{r > 0\}$ is diffeomorphic to $(0,\infty) \times \Sigma$, we will often denote points on $\{r > 0\}$ by $(r,\theta)$, where $\theta \in \Sigma$.
    \item $C$ denotes a positive constant that may change from expression to expression.
\end{itemize}

\subsection{The asymptotic cross-section} 

We now show that $(M^n,g,r)$ is asymptotic to a paraboloid in a sense. Our treatment is similar to \cite{bernstein}*{Appendix A}.
For each $\rho > 0$, let $\eta_{\Sigma_\rho}$ be the restriction of $\eta$ to the tangent bundle of the hypersurface $\{r=\rho\} \cong \Sigma$.

\begin{lemma} \label{lem:level-set-H}
    For each $\rho > 0$, the hypersurface $\{r=\rho\}$ has second fundamental form and mean curvature
    \begin{align}
        A_{\Sigma_\rho} &= \frac{1}{2\rho|\nabla r|(\rho,\cdot)} (g_{\Sigma_\rho} + \eta_{\Sigma_\rho}), \label{eq:Aformula1} \\
        H_{\Sigma_\rho} &= \frac{n-1}{2\rho} + \O(\rho^{-\mu-1}). \label{eq:Hformula1}
    \end{align}
\end{lemma}
\begin{proof}
    Let $Y$ and $Z$ be tangent vectors to $\{r=\rho\}$ based at the same point. Using \eqref{eq:hess-r}, we have
    \begin{align} \label{eq:Aformula}
        A_{\Sigma_\rho}(Y,Z) &= \frac{\nabla^2 r(Y,Z)}{|\nabla r|} = \frac{1}{2\rho|\nabla r|}(g-dr^2+\eta)(Y,Z) = \frac{1}{2\rho|\nabla r|} (g(Y,Z) + \eta(Y,Z)).
    \end{align}
    This proves \eqref{eq:Aformula1}. Now the AP hypotheses give $|\eta_{\Sigma_\rho}|_{g_{\Sigma_\rho}} = \O(\rho^{-\mu})$ and $1-\frac{1}{|\nabla r|(\rho,\cdot)} = \O(\rho^{-\mu})$, so tracing \eqref{eq:Aformula1} with respect to $g_{\Sigma_\rho}$ gives
    \begin{align}
        H_{\Sigma_\rho} &= \frac{n-1}{2\rho |\nabla r|(\rho,\cdot)} + \tr_{g_{\Sigma_\rho}}(\eta_{\Sigma_\rho}) = \frac{n-1}{2\rho} + \O(\rho^{-\mu-1}),
    \end{align}
    proving \eqref{eq:Hformula1}.
\end{proof}

Next, we compute the first variation of $g_X(\rho)$.

\begin{lemma} \label{lem:first-variation}
    For each $\rho > 0$, we have
    \begin{align}
        \frac{d}{d\rho} g_X(\rho) &= \frac{1}{\rho} \left( -1 + \frac{1}{|\nabla r|^2(\rho,\cdot)} \right) g_X(\rho) + \frac{1}{\rho^2 |\nabla r|^2(\rho,\cdot)} \eta_{\Sigma_\rho}
    \end{align}
    as symmetric 2-tensors on $\Sigma$.
\end{lemma}
\begin{proof}
    Let $\rho > 0$. By the first variation formula for the metric and Lemma \ref{lem:level-set-H}, at any point $\theta \in \Sigma$ we have
    \begin{align}
        \frac{d}{d\rho} g_X(\rho) &= \frac{d}{d\rho} (\rho^{-1} g_{\Sigma_\rho}) = -\frac{1}{\rho^2} g_{\Sigma_{\rho}} + \frac{2}{\rho|\nabla r|(\rho,\theta)} A_{\Sigma_{\rho}} \\
        &= -\frac{1}{\rho} g_X(\rho) + \frac{1}{\rho^2 |\nabla r|^2(\rho,\theta)} \left( g_{\Sigma_\rho} + \eta_{\Sigma_\rho} \right) \\
        &= \frac{1}{\rho} \left( -1 + \frac{1}{|\nabla r|^2(\rho,\theta)} \right) g_X(\rho) + \frac{1}{\rho^2 |\nabla r|^2(\rho,\theta)} \eta_{\Sigma_\rho}.
    \end{align}
\end{proof}

\begin{theorem} \label{thm:asymp-link}
    There is a $C^0$-Riemannian metric $g_X$ on $\Sigma$ such that $\lim_{\rho\to\infty} g_X(\rho) = g_X$ in $C^0(\Sigma)$. We call $(\Sigma,g_X)$ the \textbf{asymptotic cross-section} of $(M^n,g,r)$, Moreover, we have
    \begin{align}
        \norm{g_X(\rho) - g_X}_{C^0(\Sigma)} &= \O(\rho^{-\mu}), \label{eq:gL-convergence-1} \\
        \norm{\frac{d}{d\rho} g_{\Sigma_\rho} - g_X(\rho)}_{C^0(\Sigma)} &= \O(\rho^{-\mu}) \label{eq:gL-convergence-3}
    \end{align}
    as $\rho \to \infty$.
\end{theorem}
\begin{proof}
    The AP hypotheses give $|\eta_{\Sigma_\rho}|_{g_{\Sigma_\rho}} = \O(\rho^{-\mu})$, meaning that $-C\rho^{-\mu} g_{\Sigma_\rho} \leq \eta_{\Sigma_\rho} \leq C\rho^{-\mu} g_{\Sigma_\rho}$ as bilinear forms. Dividing this by $\rho$ gives $-C\rho^{-\mu} g_X(\rho) \leq \frac{1}{\rho} \eta_{\Sigma_\rho} \leq C\rho^{-\mu} g_X(\rho)$. Using this estimate, Lemma \ref{lem:first-variation}, and the fact that $|\nabla r| = 1 + \O(\rho^{-\mu})$, we get
    \begin{align} \label{eq:g_L'-pinched}
        -C\rho^{-\mu-1} g_X(\rho) \leq \frac{d}{d\rho} g_X(\rho) \leq C\rho^{-\mu-1} g_X(\rho).
    \end{align}
    Integrating this shows that for each $\rho_2 \geq \rho_1 > 0$,
    \begin{align} \label{eq:g_L-pinched}
        e^{-C\rho_1^{-\mu}} g_X(\rho_1) \leq g_X(\rho_2) \leq e^{C\rho_1^{-\mu}} g_X(\rho_1).
    \end{align}
    Hence the limit $g_X(\rho) \to g_X$ exists in the space of $C^0$ symmetric 2-tensors on $\Sigma$, and
    \begin{align} \label{eq:g_L-limit-pinched}
        e^{-C\rho^{-\mu}}g_X(\rho) \leq g_X \leq e^{C\rho^{-\mu}}g_X(\rho) \quad \text{for all } \rho > 0.
    \end{align}
    In particular, $g_X$ is positive definite everywhere and is thus a $C^0$-Riemannian metric. Now \eqref{eq:g_L-limit-pinched} gives
    \begin{align}
        \norm{g_X(\rho)-g_X}_{C^0(\Sigma)} \leq C\rho^{-\mu},
    \end{align}
    and proves \eqref{eq:gL-convergence-1}. Also $\frac{d}{d\rho} g_{\Sigma_\rho} = (\rho g_X(\rho))' = \rho g_X'(\rho) + g_X(\rho)$
    so by \eqref{eq:g_L'-pinched},
    \begin{align}
        -C\rho^{-\mu} g_X(\rho) \leq \frac{d}{d\rho} g_{\Sigma_\rho} - g_X(\rho) \leq C\rho^{-\mu} g_X(\rho).
    \end{align}
    In view of \eqref{eq:g_L-limit-pinched}, this proves \eqref{eq:gL-convergence-3}.
\end{proof}

\begin{remark} \label{rmk:c1a-metric}
    By condition (iii) in Definition \ref{def:asymp-parab}, and the compact embedding $C^2(\Sigma) \hookrightarrow C^{1,\alpha}(\Sigma)$, $g_X$ is actually a $C^{1,\alpha}$ metric (despite the convergence $g_X(\rho) \to g_X$ being only in $C^0(\Sigma)$).
\end{remark} 

From Theorem \ref{thm:asymp-link}, we easily deduce:
\begin{corollary} \label{cor:weakly-parab}
    For each $\tau \geq 1$, define
    \begin{align}
        g_\tau = dr^2 + \tau^{-1} g_{\Sigma_{\tau r}}.
    \end{align}
    Then $\lim_{\tau\to\infty} g_\tau = dr^2 + r g_X$ in $C^0_{\mathrm{loc}}(\{r \geq \frac{1}{2}\})$.
\end{corollary}

If $g = dr^2 + r g_X$ on $\{r \geq \frac{1}{2}\}$ to begin with, then $g_\tau = g$ on $\{r \geq \frac{1}{2}\}$ for all $\tau \geq 1$. Corollary \ref{cor:weakly-parab} therefore provides a sense in which $(M^n,g)$ is asymptotic to a paraboloid.

In Appendix \ref{sec:appB}, we show that the AP hypotheses control $g$ and $g_\tau$ up to second-order. This control will be needed to prove the results of \S\ref{subsec:pointwise-estimates} (done in Appendix \ref{sec:appC}), but is otherwise not used until \S\ref{subsec:seq-superharmonic}.

\section{Analytical machinery for drift-harmonic functions} \label{sec:analytic}

Hereafter, we fix a function $f \in C^\infty(M)$ on the AP manifold $(M^n,g,r)$ which satisfies Assumption \ref{assump:f}. The global translation of $r$ performed in \S\ref{subsec:conventions} does not affect Assumption \ref{assump:f}. We define $\L_f$ and $\H$ as in \eqref{eq:drift-laplacian} and \eqref{eq:harm1} respectively. For each $\rho \geq 0$, let 
\begin{align} \label{eq:Brho}
    B_\rho := M \setminus \{r \geq \rho\}.
\end{align}

\subsection{Frequency and related functionals} \label{subsec:freq-and-related}

For any nonzero $u \in \H$ and $\rho > 0$, define the quantities
\begin{align}
    D_u(\rho) &:= \rho^{\frac{3-n}{2}} \int_{\{r=\rho\}} u\inner{\nabla u}{\nu}, \\
    I_u(\rho) &:= \rho^{\frac{1-n}{2}} \int_{\{r=\rho\}} u^2 |\nabla r|, \\
    U_u(\rho) &:= \frac{D_u(\rho)}{I_u(\rho)} = \frac{\rho \int_{\{r=\rho\}} u\inner{\nabla u}{\nu}}{\int_{\{r=\rho\}} u^2|\nabla r|}, \\
    G_u(\rho) &:= \frac{\rho \int_{\{r=\rho\}} \inner{\nabla u}{\nu}^2 |\nabla r|^{-1}}{\int_{\{r=\rho\}} u^2 |\nabla r|}, \label{eq:Gdef} \\
    Q_u(\rho) &:= \frac{\rho \int_{\{r=\rho\}} |\nabla^\top u|^2|\nabla r|^{-1}}{\int_{\{r=\rho\}} u^2|\nabla r|}.
\end{align}
Here, $\nu = \frac{\nabla r}{|\nabla r|}$ is the outward unit normal to $\{r=\rho\}$, and $\nabla^\top u = \nabla u - \inner{\nabla u}{\nu}$ is the projection of $\nabla u$ onto $\{r=\rho\}$.
The quantity $U_u$ is called the \emph{frequency function} of $u$. When the function $u$ is clear from context, we will drop it from notation and just write $D, I, U, G, Q$.

Proceeding similarly to \cite{cm97a}, we compute derivatives of $D$, $I$ and $U$, and deduce some basic consequences.

\begin{lemma} \label{lem:D'I'}
    For any nonzero $u \in \H$ and $\rho > 0$, we have
    \begin{align}
        D'(\rho) &= \left(\frac{3-n}{2\rho}+f'(\rho)\right) D(\rho) + \rho^{\frac{3-n}{2}} \int_{\{r=\rho\}} |\nabla u|^2 |\nabla r|^{-1}, \label{eq:D'2} \\
        I'(\rho) &= \O(\rho^{-\mu-1}) I(\rho) + \frac{2D(\rho)}{\rho}, \label{eq:I'2} \\
        U'(\rho) &= \left( \frac{3-n}{2\rho} + f'(\rho) +\O(\rho^{-\mu-1}) \right)U(\rho) - \frac{2U(\rho)^2}{\rho} + G(\rho) + Q(\rho). \label{eq:U'2}
    \end{align}
\end{lemma}
\begin{proof}
    Using the first variation formula and Lemma \ref{lem:level-set-H}, we compute
    \begin{align}
        I'(\rho) &= \frac{1-n}{2\rho} I(\rho) + \rho^{\frac{1-n}{2}} \int_{\{r=\rho\}} \frac{1}{|\nabla r|}\left( u^2 |\nabla r| H_{\Sigma_\rho} + \inner{\nabla(u^2|\nabla r|)}{\nu} \right) \\
        &= \frac{1-n}{2\rho} I(\rho) + \frac{2D(\rho)}{\rho} + \rho^{\frac{1-n}{2}} \int_{\{r=\rho\}} u^2 \left( \frac{n-1}{2\rho} + \O(\rho^{-\mu-1}) \right) + \rho^{\frac{1-n}{2}} \int_{\{r=\rho\}} u^2 \inner{\frac{\nabla|\nabla r|}{|\nabla r|}}{\nu} \\
        &= \O(\rho^{-\mu-1})I(\rho) + \frac{2D(\rho)}{\rho},
    \end{align}
    where the last equality uses \eqref{eq:grad|gradr|}. This gives \eqref{eq:I'2}.
    Next, the divergence theorem and $\L_f u = 0$ give
    \begin{align}
        D(\rho) &= \rho^{\frac{3-n}{2}} e^{f(\rho)} \int_{B_\rho} \div(e^{-f}u\nabla u) = \rho^{\frac{3-n}{2}} e^{f(\rho)} \int_{B_\rho} |\nabla u|^2 e^{-f}. \label{eq:D-alternative2}
    \end{align}
    Differentiating and using the coarea formula yields \eqref{eq:D'2}. Using \eqref{eq:D'2}, \eqref{eq:I'2}, and the definition of $U$, we have
    \begin{align}
        U' &= \frac{D'}{I}-\frac{I'}{I}U = \left( \frac{3-n}{2\rho}+f'(\rho)+\O(\rho^{-\mu-1}) \right) U + \frac{\rho \int_{\{r=\rho\}} |\nabla u|^2 |\nabla r|^{-1}}{\int_{\{r=\rho\}} u^2 |\nabla r|} - \frac{2U^2}{\rho}.
    \end{align}
    Since $|\nabla u|^2 = \inner{\nabla u}{\nu}^2 + |\nabla^\top u|^2$ on $\{r=\rho\}$, this implies \eqref{eq:U'2}.
\end{proof}

\begin{remark}
    From \eqref{eq:D-alternative2} we have $U(\rho) \geq 0$ for all $\rho > 0$, with equality if and only if $u$ is constant on $B_\rho$.
\end{remark}

\begin{corollary} \label{cor:I-almost-mono}
    For any nonzero $u \in \H$ and any pair $\rho_2 > \rho_1 \geq 1$, we have
    \begin{align} \label{eq:Imono1}
        I(\rho_2) \geq e^{-C\rho_1^{-\mu}} I(\rho_1) \geq C^{-1} I(\rho_1).
    \end{align}
    Moreover, if there exist $K, d, \gamma > 0$ such that $U(\rho) \geq d-K\rho^{-\gamma}$ for all $\rho \in [\rho_1, \rho_2]$, then
    \begin{align} \label{eq:I-almost-mono}
        I(\rho_2) \geq e^{-C (\rho_1^{-\mu} + \rho_1^{-\gamma})} I(\rho_1) \left( \frac{\rho_2}{\rho_1} \right)^{2d} \geq C^{-1} I(\rho_1)\left( \frac{\rho_2}{\rho_1} \right)^{2d}
    \end{align}
    where $C = C(K,d,\gamma)$. Similarly, if $U(\rho) \leq d+K\rho^{-\gamma}$ for all $\rho \in [\rho_1,\rho_2]$, then
    \begin{align} \label{eq:Imono3}
        I(\rho_2) \leq e^{C(\rho_1^{-\mu}+\rho_1^{-\gamma})} I(\rho_1) \left( \frac{\rho_2}{\rho_1} \right)^{2d} \leq C I(\rho_1) \left( \frac{\rho_2}{\rho_1} \right)^{2d}.
    \end{align}
\end{corollary}
\begin{proof}
    Let $\rho_2 > \rho_1 \geq 1$. By Lemma \ref{eq:I'2}, we have for each $\rho \in [\rho_1, \rho_2]$
    \begin{align} \label{eq:logI'bound}
        \frac{I'(\rho)}{I(\rho)} = \O(\rho^{-\mu-1}) + \frac{2U(\rho)}{\rho} \geq -C\rho^{-\mu-1}.
    \end{align}
    Integrating this gives
    \begin{align}
        \log\left( \frac{I(\rho_2)}{I(\rho_1)} \right) &\geq -C\int_{\rho_1}^{\rho_2} \rho^{-\mu-1} \, d\rho \geq -C\rho_1^{-\mu} \geq -C.
    \end{align}
    Exponentiating this yields \eqref{eq:Imono1}. If $U(\rho) \geq d - K\rho^{-\gamma}$ for all $\rho \in [\rho_1,\rho_2]$, where $d,K,\gamma>0$, the estimate \eqref{eq:logI'bound} instead becomes
    \begin{align} \label{eq:0013308}
        \frac{I'(\rho)}{I(\rho)} \geq - C\rho^{-\mu-1} + 2d\rho^{-1} - 2K\rho^{-\gamma-1} \geq -C(\rho^{-\mu-1} + \rho^{-\gamma-1}),
    \end{align}
    where $C = C(d,K,\gamma)$. This integrates and exponentiates to \eqref{eq:I-almost-mono}. A similar argument with reversed signs proves \eqref{eq:Imono3}.
\end{proof}

\begin{corollary} \label{cor:liminf-upper}
    If $u \in \H_d^+$, then $\liminf_{\rho\to\infty} U(\rho) \leq d$.
\end{corollary}
\begin{proof}
    Otherwise, there exists $\delta > 0$ such that $U(\rho) \geq d+\delta$ for all large $\rho$. By Corollary \ref{cor:I-almost-mono}, we have $I(\rho) \geq C^{-1}\rho^{2d+2\delta}$ for all $\rho \geq 1$. However, for each $\epsilon > 0$ we have $u = \O(r^{d+\epsilon})$, so by the definition of $I$,
    \begin{align}
        I(\rho) \leq C_\epsilon \rho^{2d+2\epsilon}
    \end{align}
    for all $\rho \geq 1$, where $C_\epsilon$ is independent of $\rho$.
    Taking $\epsilon < \delta$ therefore yields a contradiction.
\end{proof}

\subsection{An ODE lemma}

By the Cauchy--Schwarz inequality, for any nonzero $u \in \H$ and $\rho > 0$, we have
\begin{align} \label{eq:simple-CS}
    G(\rho) \geq \frac{U(\rho)^2}{\rho}.
\end{align}
Using this in conjunction with \eqref{eq:U'2} gives a differential inequality for $U$. Later on, we will use this differential inequality and the next lemma to turn lower bounds for $Q$ into lower bounds for $U$.

\begin{lemma} \label{lem:important-ode-lemma}
    Let $\bar{\rho} > 0$ and let $\U: (\bar{\rho},\infty) \to [0,\infty)$ be a nonnegative $C^1$ function such that
    \begin{align}
        \U'(\rho) \geq \left(-1-\frac{C_1}{\rho}\right) \U(\rho) - \frac{\U(\rho)^2}{\rho} + \Q(\rho),
    \end{align}
    where $C_1 > 0$ and $\Q: (\bar{\rho},\infty) \to (0,\infty)$ is a continuous function satisfying
    \begin{align}
        \Q(\rho) \geq \lambda-C_2 \rho^{-\tau} \quad \text{for all } \rho > \bar{\rho}
    \end{align}
    for some $C_2, \lambda, \tau > 0$.
    Then there exists $C = C(C_1,C_2,\tau,\lambda) > 0$ such that for all $\rho > \bar{\rho}$,
    \begin{align}
        \U(\rho) \geq \max\{\lambda - C(\rho-\bar{\rho})^{-\gamma}, 0\},
    \end{align}
    where $\gamma = \min\{\tau,1\}$.
\end{lemma}
\begin{proof}
    Let $\epsilon > 0$, and suppose $\rho > \bar{\rho}$ is such that $\U(\rho) < \lambda-\frac{\epsilon}{2}$. Then
    \begin{align}
        \U'(\rho) &\geq \left(-1-\frac{C_1}{\rho}\right) \left(\lambda - \frac{\epsilon}{2}\right) - \frac{\lambda^2}{\rho} + \lambda -C_2 \rho^{-\tau}
        \geq \frac{\epsilon}{2} - C\rho^{-\gamma},
    \end{align}
    where $\gamma = \min\{\tau,1\}$ and $C = C(C_1,C_2,\tau,\lambda)$. Thus if
    \begin{align}
        \rho \geq \max\left\{ \left( \frac{4C}{\epsilon} \right)^{1/\gamma}, \bar{\rho} \right\} \quad \text{and} \quad \U(\rho) \leq \lambda - \epsilon,
    \end{align}
    then $\U'(\rho) \geq \frac{\epsilon}{4}$. Moreover, as $\U$ is nonnegative,
    it will take at most $\frac{\lambda}{\epsilon/4}$ extra distance for $\U$ to exceed $\lambda-\epsilon$, and from then on $\U$ will never go below $\lambda-\epsilon$ since otherwise $\U' \geq \frac{\epsilon}{4} > 0$, a contradiction. Hence
    \begin{align}
        \U(\rho) \geq \lambda-\epsilon \quad \text{whenever} \quad \rho \geq \max\left\{ \left( \frac{4C}{\epsilon} \right)^{1/\gamma}, \bar{\rho} \right\} + \frac{4\lambda}{\epsilon}.
    \end{align}
    Since $\frac{1}{\gamma} \geq 2$, the threshold on the right is less than $\bar{\rho}+\left(\frac{C}{\epsilon}\right)^{1/\gamma}$ for some $C = C(C_1,C_2,\tau,\lambda)$; thus
    \begin{align} \label{eq:U-conc}
        \U(\rho) \geq \lambda-\epsilon \quad \text{whenever} \quad \rho-\bar{\rho} \geq \left( \frac{C}{\epsilon} \right)^{1/\gamma}.
    \end{align}
    This implies that $\U(\rho) \geq \lambda-(C+1)(\rho-\bar{\rho})^{-\gamma}$ for all $\rho \geq \bar{\rho}$.
\end{proof}

\subsection{Another formula for $D'/D$} \label{subsec:rellich-necas}

Next, we derive an alternative formula for $D'/D$ by means of a Rellich--Ne\v{c}as type identity (Lemma \ref{lem:rellich}). Similar computations have been carried out in \cite{bernstein} and \cite{granlund}.

\begin{lemma} \label{lem:nabladivXlemma}
    Define the vector field $V = r\del_r = r\frac{\nabla r}{|\nabla r|^2}$ on $\{r \geq 0\}$. Then
    \begin{align}
        \div V = \frac{n+1}{2} + \O(r^{-\mu}),
    \end{align}
    and for any function $u$ we have
    \begin{align}
        \inner{\nabla_{\nabla u} V}{\nabla u} &= \left( \frac{1}{2} + \O(r^{-\mu}) \right) \inner{\nabla u}{\nabla r}^2 + \left( \frac{1}{2} + \O(r^{-\mu}) \right) |\nabla u|^2.
    \end{align}
\end{lemma}
\begin{proof}
    By \eqref{eq:grad|gradr|}, we have $\inner{\nabla|\nabla r|^{-2}}{\nabla r} = \O(r^{-\mu-1})$, and by \eqref{eq:hess-r},
    \begin{align}
        \Delta r = \tr_g(\nabla^2 r) = \frac{1}{2r} \tr_g(g-dr^2+\eta) = \frac{n-1}{2r}+\O(r^{-\mu-1}).
    \end{align}
    Hence,
    \begin{align}
        \div V &= \frac{r \Delta r}{|\nabla r|^2} + 1 + r\inner{\nabla|\nabla r|^{-2}}{\nabla r} = \frac{n+1}{2} + \O(r^{-\mu}).
    \end{align}
    For any function $u$, we also compute
    \begin{align}
        \inner{\nabla u}{\nabla|\nabla r|^{-2}} &= -|\nabla r|^{-4} \inner{\nabla u}{\nabla|\nabla r|^2} = -2|\nabla r|^{-4} \nabla^2 r(\nabla u,\nabla r),
    \end{align}
    and so
    \begin{align}
        \inner{\nabla_{\nabla u}V}{\nabla u} &= \frac{1}{|\nabla r|^2} \inner{\nabla u}{\nabla r}^2 + \frac{r}{|\nabla r|^2} \nabla^2 r(\nabla u,\nabla u) - 2r|\nabla r|^{-4} \inner{\nabla u}{\nabla r} \nabla^2 r(\nabla u,\nabla r).
    \end{align}
    Then by \eqref{eq:hess-r}, this becomes
    \begin{align}
        \inner{\nabla_{\nabla u}V}{\nabla u} &= \frac{1}{|\nabla r|^2} \inner{\nabla u}{\nabla r}^2 + \frac{1}{2|\nabla r|^2} \left( |\nabla u|^2 - \inner{\nabla u}{\nabla r}^2 + \eta(\nabla u,\nabla u) \right) \\
        &\quad - |\nabla r|^{-4} \inner{\nabla u}{\nabla r} \left( \inner{\nabla u}{\nabla r} - |\nabla r|^2 \inner{\nabla u}{\nabla r} + \eta(\nabla u,\nabla r) \right) \\
        &= \left( \frac{1}{2} + \O(r^{-\mu}) \right) \inner{\nabla u}{\nabla r}^2 + \left( \frac{1}{2} + \O(r^{-\mu}) \right) |\nabla u|^2.
    \end{align}
\end{proof}

\begin{lemma} \label{lem:rellich}
    If $u \in \H$, then for each $\rho > 0$ we have
    \begin{align}
        \rho e^{-f(\rho)} \int_{\{r=\rho\}} |\nabla u|^2 |\nabla r|^{-1} &= 2\rho e^{-f(\rho)} \int_{\{r=\rho\}} \inner{\nabla u}{\nu}^2 |\nabla r|^{-1} + \frac{n-1}{2}\int_{\{0<r<\rho\}} (1+\O(r^{-\mu})) |\nabla u|^2 e^{-f} \\
        &\quad - \int_{\{0<r<\rho\}} (1+\O(r^{-\mu}))\inner{\nabla u}{\nabla r}^2 e^{-f} - \int_{\{0<r<\rho\}} r|\nabla u|^2 f'(r)e^{-f}.
    \end{align}
\end{lemma}
\begin{proof}
    Let $V = r\frac{\nabla r}{|\nabla r|^2}$, which is defined on $\{r \geq 0\}$.
    By the divergence theorem and Lemma \ref{lem:nabladivXlemma}, we have
    \begin{align}
        \rho e^{-f(\rho)} \int_{\{r=\rho\}} |\nabla u|^2 |\nabla r|^{-1} &= \int_{\{0 < r < \rho\}} \div(e^{-f}|\nabla u|^2 V) \\
        &= \int_{\{0 < r < \rho\}} \left( 2\nabla^2 u(\nabla u,V) - |\nabla u|^2 \inner{\nabla f}{V} + |\nabla u|^2 \div V \right) e^{-f} \\
        &= \int_{\{0 < r < \rho\}} \left( 2\nabla^2 u(\nabla u, V) - r|\nabla u|^2 f'(r) + \frac{n+1}{2}(1+\O(r^{-\mu}))|\nabla u|^2 \right) e^{-f}. \label{eq:1060168}
    \end{align}
    Using that $\L_f u = 0$, one has
    \begin{align}
        \div(e^{-f} \inner{\nabla u}{V}\nabla u) &= e^{-f}( \inner{\nabla_{\nabla u}V}{\nabla u} + \nabla^2 u(\nabla u,V)). \label{eq:comp1}
    \end{align}
    Inserting this into \eqref{eq:1060168}, then using the divergence theorem on the $\div$ term and Lemma \ref{lem:nabladivXlemma} to handle the $\inner{\nabla_{\nabla u}V}{\nabla u}$ term, the claim follows.
\end{proof}

\begin{corollary} \label{cor:D'/D}
    For any nonzero $u \in \H$ and $\rho > 0$, we have
    \begin{align}
        \frac{D'(\rho)}{D(\rho)} &= f'(\rho) - \frac{\int_{\{0<r<\rho\}} r|\nabla u|^2 f'(r)e^{-f}}{\rho^{\frac{n-1}{2}} e^{-f(\rho)} D(\rho)} + \frac{\int_{\{0 < r < \rho\}} (1+\O(r^{-\mu}))|\nabla u|^2 e^{-f}}{\rho^{\frac{n-1}{2}} e^{-f(\rho)} D(\rho)} + \frac{2G}{U} \\
        &\quad - \frac{\int_{\{0<r<\rho\}} (1+\O(r^{-\mu})) \inner{\nabla u}{\nabla r}^2 e^{-f}}{\rho^{\frac{n-1}{2}} e^{-f(\rho)} D(\rho)} - \frac{n-3}{2} \frac{\int_{B_{0}} |\nabla u|^2 e^{-f}}{\rho^{\frac{n-1}{2}} D(\rho)}.
    \end{align}
\end{corollary}
\begin{proof}
    Using Lemma \ref{lem:D'I'}, Lemma \ref{lem:rellich} and the formula \eqref{eq:D-alternative2} for $D$, we compute
    \begin{align}
        D' &= \left( \frac{3-n}{2\rho} + f'(\rho) \right) D + \rho^{\frac{1-n}{2}}e^{f(\rho)} \left( \rho \int_{\{r=\rho\}} |\nabla u|^2 |\nabla r|^{-1} e^{-f(\rho)} \right) \\
        &= f'(\rho) D + \frac{3-n}{2} \rho^{\frac{1-n}{2}} \int_{B_\rho} |\nabla u|^2 e^{-f} + \rho^{\frac{1-n}{2}}e^{f(\rho)} \Big( 2\rho e^{-f(\rho)} \int_{\{r=\rho\}} \inner{\nabla u}{\nu}^2 |\nabla r|^{-1} \\
        &\qquad +\frac{n-1}{2}\int_{\{0<r<\rho\}} (1+\O(r^{-\mu})) |\nabla u|^2 e^{-f} - \int_{\{0<r<\rho\}} (1+\O(r^{-\mu}))\inner{\nabla u}{\nabla r}^2 e^{-f} \\
        &\qquad - \int_{\{0<r<\rho\}} r|\nabla u|^2 f'(r)e^{-f} \Big) \\
        &= f'(\rho) D + \rho^{\frac{1-n}{2}}e^{f(\rho)} \int_{\{0<r<\rho\}} (1 + \O(r^{-\mu})) |\nabla u|^2 e^{-f} + \frac{3-n}{2} \rho^{\frac{1-n}{2}} \int_{B_{0}} |\nabla u|^2 e^{-f} \\
        &\qquad + 2\rho^{\frac{3-n}{2}} \int_{\{r=\rho\}} \inner{\nabla u}{\nu}^2 |\nabla r|^{-1} - \rho^{\frac{1-n}{2}} e^{f(\rho)} \int_{\{0<r<\rho\}} (1+\O(r^{-\mu})) \inner{\nabla u}{\nabla r}^2 e^{-f} \\
        &\qquad - \rho^{\frac{1-n}{2}} e^{f(\rho)} \int_{\{0<r<\rho\}} r|\nabla u|^2 f'(r) e^{-f}.
    \end{align}
    Dividing this by $D$, and using that
    \begin{align}
        \frac{2\rho^{\frac{3-n}{2}} \int_{\{r=\rho\}} \inner{\nabla u}{\nu}^2 |\nabla r|^{-1}}{D} &= \frac{2\rho^{\frac{3-n}{2}} \int_{\{r=\rho\}} \inner{\nabla u}{\nu}^2 |\nabla r|^{-1}}{I} \frac{1}{U} = \frac{2G}{U}
    \end{align}
    where $G$ was defined in \eqref{eq:Gdef}, the corollary follows.
\end{proof}

\subsection{Almost separation of variables and preservation of almost orthogonality}

We now introduce the central notions of this paper, following Colding and Minicozzi \cite{cm97a}.

\begin{definition} \label{def:almost-sep}
    Let $\delta > 0$, $\rho_2 > \rho_1 > 0$ and $u \in \H$. We say that $u$ \textbf{$\delta$-almost separates variables} on the annulus $\{\rho_1 \leq r \leq \rho_2\}$ if
    \begin{align}
        \int_{\{\rho_1 \leq r \leq \rho_2\}} r^{-\frac{n+1}{2}} \left( r\inner{\nabla u}{\nu} - U u|\nabla r| \right)^2 \leq \delta^2 I(\rho_2).
    \end{align}
    Given $C,\tau > 0$, we say that $u$ \textbf{$(C,\tau)$-asymptotically separates variables} if for all $\rho_2 > \rho_1 > 0$, the function $u$ $C\rho_1^{-\tau}$-almost separates variables on the annulus $\{\rho_1 \leq r \leq \rho_2\}$.
\end{definition}

\begin{lemma} \label{lem:almost-sep-formula}
    Let $u \in \H$ and $\rho_2 > \rho_1 > 0$. Then
    \begin{align} \label{eq:008}
        \int_{\{\rho_1 \leq r \leq \rho_2\}} r^{-\frac{n+1}{2}} \left( r\inner{\nabla u}{\nu} - U u|\nabla r| \right)^2 = \int_{\rho_1}^{\rho_2} \left( \frac{G(\rho)}{U(\rho)} - \frac{U(\rho)}{\rho} \right) D(\rho) \, d\rho.
    \end{align}
\end{lemma}
\begin{proof}
    Computing using the coarea formula,
    \begin{align}
        &\int_{\{\rho_1 \leq r \leq \rho_2\}} r^{-\frac{n+1}{2}} \left( r\inner{\nabla u}{\nu} - U u|\nabla r| \right)^2 = \int_{\rho_1}^{\rho_2} \rho^{-\frac{n+1}{2}} \int_{\{r=\rho\}} \left( \rho \inner{\nabla u}{\nu}|\nabla r|^{-\frac{1}{2}} - U(\rho)u|\nabla r|^{\frac{1}{2}} \right)^2 \, d\rho \\
        &\quad= \int_{\rho_1}^{\rho_2} \rho^{-\frac{n+1}{2}} \int_{\{r=\rho\}} \left( \rho^2 \inner{\nabla u}{\nu}^2 |\nabla r|^{-1} - 2\rho U(\rho) u \inner{\nabla u}{\nu} + U(\rho)^2 u^2 |\nabla r| \right) \, d\rho \\
        &\quad= \int_{\rho_1}^{\rho_2} \left( \rho^{\frac{3-n}{2}} \int_{\{r=\rho\}} \inner{\nabla u}{\nu}^2 |\nabla r|^{-1} \right) d\rho - 2 \int_{\rho_1}^{\rho_2} \frac{U(\rho) D(\rho)}{\rho} \, d\rho + \int_{\rho_1}^{\rho_2} \frac{U(\rho)^2 I(\rho)}{\rho} \, d\rho \\
        &\quad= \int_{\rho_1}^{\rho_2} G(\rho) I(\rho) \, d\rho - \int_{\rho_1}^{\rho_2} \frac{U(\rho) D(\rho)}{\rho} \, d\rho.
    \end{align}
    As $I = \frac{D}{U}$, \eqref{eq:008} follows.
\end{proof}

\begin{definition} \label{def:uv-inner}
    For each $\rho > 0$, we define the normalized $L^2$-inner product $\inner{\cdot}{\cdot}_\rho$ and norm $\norm{\cdot}_\rho$ on the space of functions on $\{r=\rho\}$ by
    \begin{align}
        \inner{u}{v}_\rho &:= \rho^{\frac{1-n}{2}} \int_{\{r=\rho\}} uv|\nabla r|, \\
        \norm{u}_\rho &:= \sqrt{\inner{u}{u}_\rho}.
    \end{align}
    If $u, v: M \to \R$ are globally defined, then $\inner{u}{v}_\rho$ and $\norm{u}_\rho$ denote the above quantities computed on the restrictions $u|_{\{r=\rho\}}, v|_{\{r=\rho\}}$. Note that $I_u(\rho) = \inner{u}{u}_\rho = \norm{u}_\rho^2$; we will use these interchangeably.
\end{definition}

\begin{definition} \label{def:almost-orthog}
    Let $\delta > 0$ and $\rho > 0$. Two functions $u,v \in \H$ are \textbf{$\delta$-almost orthogonal} on $\{r=\rho\}$ if
    \begin{align}
        \frac{\left| \inner{u}{v}_\rho \right|}{\norm{u}_\rho \norm{v}_\rho} \leq \delta.
    \end{align}
    Given $C,\tau > 0$, we say that $u$ and $v$ are \textbf{$(C,\tau)$-asymptotically orthogonal} if for each $\rho > 0$, the functions $u$ and $v$ are $C\rho^{-\tau}$-almost orthogonal on $\{r=\rho\}$.
\end{definition}

The next proposition and its corollary are analogs of the results proved in \cite{cm97a}*{\S 5}. In particular, Corollary \ref{cor:pres-of-inner-prod} is known as preservation of almost orthogonality and is the main engine in the proof of Theorem \ref{thm:main-expanded}.

\begin{proposition} \label{prop:pres-of-orthogonality-1}
    There exists $C>0$ such that if
    \begin{enumerate}[label=(\roman*)]
        \item $u,v \in \H$ are nonzero,
        \item $\rho_2 > \rho_1 \geq 1$,
        \item $v$ $\delta$-almost separates variables on the annulus $\{\rho_1 \leq r \leq \rho_2\}$,
        \item $\inner{u}{v}_{\rho_2} = 0$,
    \end{enumerate}
    then
    \begin{align}
        \inner{u}{v}_{\rho_1}^2 \leq Ce^{\frac{C}{\delta^2} \rho_1^{-\mu}} \delta^2 \left( \frac{\rho_2}{\rho_1} \right)^{4d+2} I_u(\rho_2) I_v(\rho_2),
    \end{align}
    where $d = \max_{\rho \in [\rho_1,\rho_2]} U_v(\rho)$.
\end{proposition}
\begin{proof}
    The proof is an adaptation of \cite{cm97a}*{Proposition 5.1}.
    Using the divergence theorem and that $\L_f v = 0$, $\L_f u = 0$, we have
    \begin{align} \label{eq:symm-equiv}
        \int_{\{r=\rho\}} u\inner{\nabla v}{\nu} = e^{f(\rho)} \int_{B_\rho} \div(e^{-f} u \nabla v) = e^{f(\rho)} \int_{B_\rho} e^{-f} \inner{\nabla u}{\nabla v} = \int_{\{r=\rho\}} v \inner{\nabla u}{\nu}.
    \end{align}
    Define
    \begin{align}
        J(\rho) := \inner{u}{v}_\rho = \rho^{\frac{1-n}{2}} \int_{\{r=\rho\}} uv|\nabla r|.
    \end{align}
    Computing using the first variation formula and \eqref{eq:symm-equiv},
    \begin{align}
        J'(\rho) &= \frac{1-n}{2\rho} J(\rho) + \rho^{\frac{1-n}{2}} \int_{\{r=\rho\}} \left( u\inner{\nabla v}{\nu} + v\inner{\nabla u}{\nu} + uv \inner{\nabla|\nabla r|}{\nu} + uvH_{\Sigma_\rho} \right) \\
        &= 2\rho^{\frac{1-n}{2}} \int_{\{r=\rho\}} u\inner{\nabla v}{\nu} + \rho^{\frac{1-n}{2}} \int_{\{r=\rho\}} uv |\nabla r| \underbrace{\left( \frac{H_{\Sigma_\rho}}{|\nabla r|} - \frac{n-1}{2\rho} + \inner{\nabla|\nabla r|}{\frac{\nabla r}{|\nabla r|^2}} \right)}_{=: E/2}.
    \end{align}
    By Lemma \ref{lem:level-set-H}, we have $E = \O(r^{-\mu-1})$. By the Cauchy--Schwarz inequality,
    \begin{align}
        |J'(\rho)| &= \left| 2\rho^{\frac{-1-n}{2}} \left( \rho \int_{\{r=\rho\}} u\inner{\nabla v}{\nu} - \int_{\{r=\rho\}} (U_v(\rho) - \rho E) uv |\nabla r| \right) + 2\rho^{\frac{-1-n}{2}} U_v(\rho) \int_{\{r=\rho\}} uv|\nabla r| \right| \\
        &\leq 2\rho^{\frac{-1-n}{2}} \left( \int_{\{r=\rho\}} |u| \Big| \rho\inner{\nabla v}{\nu} - (U_v(\rho) - \rho E) v|\nabla r| \Big| \right) + \frac{2d}{\rho} |J(\rho)| \\
        &\leq 2\rho^{-\frac{1}{2}} \left( \rho^{\frac{1-n}{2}} \int_{\{r=\rho\}} u^2|\nabla r| \right)^{1/2} \left( \rho^{\frac{-n-1}{2}} \int_{\{r=\rho\}} \frac{1}{|\nabla r|} \Big( \rho \inner{\nabla v}{\nu} - (U_v(\rho) - \rho E)v|\nabla r| \Big)^2 \right)^{1/2} + \frac{2d}{\rho}|J(\rho)| \\
        &= 2 \sqrt{\frac{I_u(\rho)F_v(\rho)}{\rho}} + \frac{2d}{\rho} |J(\rho)|, \label{eq:602802}
    \end{align}
    where $$F_v(\rho) = \rho^{-\frac{n+1}{2}} \int_{\{r=\rho\}} \frac{1}{|\nabla r|} \Big( \rho \inner{\nabla v}{\nu} - (U_v(\rho) - \rho E)v|\nabla r| \Big)^2 \geq 0.$$
    Define $a := \sqrt{2\delta^2 I_u(\rho_2) I_v(\rho_2)}$.
    If $|J(\rho_1)| \leq a$ we are done, so assume that $|J(\rho_1)| > a$. Since $J(\rho_2) = 0$, let $\rho_3 \in (\rho_1, \rho_2)$ be the smallest $\rho$ such that $|J(\rho)| = a$. Using $-v$ in place of $v$ if necessary, we may assume $J(\rho_1) > a$, and so $J(\rho) \geq a$ for all $\rho \in [\rho_1,\rho_3]$. We compute using the absorbing inequality
    \begin{align}
        \int_{\rho_1}^{\rho_3} 2\sqrt{\frac{I_u(\rho) F_v(\rho)}{\rho}} \frac{1}{|J(\rho)|} \, d\rho &\leq \int_{\rho_1}^{\rho_3} 2 \sqrt{\frac{I_u(\rho) F_v(\rho)}{a^2 \rho}} \, d\rho \leq \int_{\rho_1}^{\rho_3} \left( \rho \frac{I_u(\rho) F_v(\rho)}{a^2 \rho} + \rho^{-1}\right) \, d\rho \\
        &\leq \frac{1}{a^2} \left( \max_{\rho \in [\rho_1,\rho_2]} I_u(\rho) \right) \int_{\rho_1}^{\rho_3} F_v(\rho) \, d\rho + \int_{\rho_1}^{\rho_3} \frac{1}{\rho} \, d\rho. \label{eq:Comp1}
    \end{align}
    Now by the coarea formula, the $\delta$-almost separation of $v$, and the fact that $E = \O(r^{-\mu-1})$,
    \begin{align}
        \int_{\rho_1}^{\rho_3} F_v(\rho) \, d\rho &= \int_{\rho_1}^{\rho_3} \left\{ \int_{\{r=\rho\}} r^{-\frac{n+1}{2}} \frac{1}{|\nabla r|} \left[ \Big( \rho \inner{\nabla v}{\nu} - U_v(\rho) v|\nabla r| \Big) + \rho Ev|\nabla r| \right]^2 \right\} \, d\rho \\
        &\leq 2\int_{\{\rho_1 \leq r \leq \rho_3\}} r^{-\frac{n+1}{2}} \Big( r \inner{\nabla v}{\nu} - U_v(r)v|\nabla r| \Big)^2 + 2 \int_{\rho_1}^{\rho_3} \left\{ \rho^{-\frac{n+1}{2}+2} \int_{\{r=\rho\}} E^2 v^2 |\nabla r| \right\} \, d\rho \\
        &\leq 2\delta^2 I_v(\rho_2) + 2C \int_{\rho_1}^{\rho_3} \rho^{-2\mu-1} I_v(\rho) \, d\rho \\
        &\leq \left( 2\delta^2 + 2C \rho_1^{-2\mu} \right) \max_{\rho\in [\rho_1, \rho_2]} I_v(\rho). \label{eq:Comp2}
    \end{align}
    Combining \eqref{eq:Comp1} and \eqref{eq:Comp2} into \eqref{eq:602802}, and also using Corollary \ref{cor:I-almost-mono}, we get
    \begin{align}
        \log \left( \frac{J(\rho_1)}{J(\rho_3)} \right) &\leq \int_{\rho_1}^{\rho_3} |(\log J)'(\rho)| \, d\rho \leq \int_{\rho_1}^{\rho_3} \left( 2\sqrt{\frac{I_u(\rho) F_v(\rho)}{\rho}} \frac{1}{|J(\rho)|} + \frac{2d}{\rho} \right) \, d\rho \\
        &= \left( \max_{\rho \in [\rho_1,\rho_2]} I_u(\rho) \right) \left( \max_{\rho \in [\rho_1,\rho_2]} I_v(\rho) \right) \left( \frac{2\delta^2 + 2C \rho_1^{-2\mu}}{a^2} \right) + (2d+1)\log\left( \frac{\rho_3}{\rho_1} \right) \\
        &\leq C \left( 1 + \frac{C}{\delta^2} \rho_1^{-2\mu} \right) + (2d+1) \log\left( \frac{\rho_3}{\rho_1} \right).
    \end{align}
    The proposition follows from exponentiating this and using that $J(\rho_3) = a$.
\end{proof}

\begin{corollary}[Preservation of almost orthogonality] \label{cor:pres-of-inner-prod}
    There exists $C>0$ such that if
    \begin{enumerate}[label=(\roman*)]
        \item $u,v \in \H$ are nonzero,
        \item $\rho_2 > \rho_1 \geq 1$,
        \item $v$ $\delta$-almost separates variables on the annulus $\{\rho_1 \leq r \leq \rho_2\}$,
    \end{enumerate}
    then
    \begin{align} \label{eq:pres-2}
        \left| \frac{\inner{u}{v}_{\rho_2}}{\norm{u}_{\rho_2} \norm{v}_{\rho_2}} - \sqrt{\frac{I_u(\rho_1)}{I_u(\rho_2)}} \sqrt{\frac{I_v(\rho_2)}{I_v(\rho_1)}} \frac{\inner{u}{v}_{\rho_1}}{\norm{u}_{\rho_1} \norm{v}_{\rho_1}} \right| \leq Ce^{\frac{C}{\delta^2}\rho_1^{-\mu}} \delta \left( \frac{\rho_2}{\rho_1} \right)^{4d+1},
    \end{align}
    where $d = \max_{\rho \in [\rho_1, \rho_2]} U_v(\rho)$.
\end{corollary}
\begin{proof}
    Write
    \begin{align} \label{eq:u-decomp}
        u = \tilde{u} + \lambda v, \quad \lambda = \frac{\inner{u}{v}_{\rho_2}}{\inner{v}{v}_{\rho_2}},
    \end{align}
    so that $\inner{\tilde{u}}{v}_{\rho_2} = 0$ and $\L_f \tilde{u} = 0$. By Proposition \ref{prop:pres-of-orthogonality-1}, and using that $I_{\tilde{u}}(\rho_2) \leq I_u(\rho_2)$ (because \eqref{eq:u-decomp} is an orthogonal decomposition with respect to $\inner{\cdot}{\cdot}_{\rho_2}$), we have
    \begin{align} \label{eq:7080}
        \inner{\tilde{u}}{v}_{\rho_1}^2 &\leq Ce^{\frac{C}{\delta^2} \rho_1^{-\mu}}\delta^2 \left( \frac{\rho_2}{\rho_1} \right)^{4d+2} I_u(\rho_2) I_v(\rho_2).
    \end{align}
    It follows that
    \begin{align}
        \left| \inner{u}{v}_{\rho_2} - \frac{I_v(\rho_2)}{I_v(\rho_1)} \inner{u}{v}_{\rho_1} \right|^2 &= \left| \inner{\lambda v}{v}_{\rho_2} - \frac{I_v(\rho_2)}{I_v(\rho_1)} \inner{\lambda v}{v}_{\rho_1} - \frac{I_v(\rho_2)}{I_v(\rho_1)} \inner{\tilde{u}}{v}_{\rho_1} \right|^2 = \left( \frac{I_v(\rho_2)}{I_v(\rho_1)} \right)^2 \inner{\tilde{u}}{v}_{\rho_1}^2 \\
        &\leq Ce^{\frac{C}{\delta^2} \rho_1^{-\mu}} \delta^2 \left( \frac{I_v(\rho_2)}{I_v(\rho_1)} \right)^2 \left( \frac{\rho_2}{\rho_1} \right)^{4d+2} I_u(\rho_2) I_v(\rho_2). \label{eq:pres-01}
    \end{align}
    Since $U_v(\rho) \leq d$ for all $\rho \in [\rho_1,\rho_2]$, it follows from Corollary \ref{cor:I-almost-mono} that $\frac{I_v(\rho_2)}{I_v(\rho_1)} \leq C \left( \frac{\rho_2}{\rho_1} \right)^{2d}$. Substituting this into \eqref{eq:pres-01}, dividing both sides by $I_u(\rho_2) I_v(\rho_2)$, then taking square roots, we arrive at \eqref{eq:pres-2}.
\end{proof}

\begin{remark} \label{rmk:pres-of-inner-prod-ball}
    If $u \in C^2(\overline{B}_{\rho_2})$ satisfies $\L_f u = 0$ in $B_{\rho_2}$, then the quantities in \S\ref{subsec:freq-and-related} are still well-defined on the interval $\rho \in [0,\rho_2]$, and Corollary \ref{cor:pres-of-inner-prod} remains valid (with the assumptions on $v$ there unchanged).
\end{remark}

\subsection{Blowdown setup and estimates for drift-harmonic functions} \label{subsec:pointwise-estimates}

In this subsection, we show how $\L_f u = 0$ can be transformed into a related parabolic equation (Lemma \ref{lem:conversion-to-parabolic}). This will substitute scaling arguments in proving estimates for drift-harmonic functions, which are also stated here. The notation and setup presented below will only reappear in \S\ref{sec:construction} and Appendix \ref{sec:appC}, so for a first reading, we suggest only acknowledging the statements of Corollary \ref{cor:deriv-growthbounds} and Theorem \ref{thm:meanValue}, then skipping ahead to \S\ref{sec:dividing} and \S\ref{sec:asymp-ctrl}.

For $t \in \R$, let $\Phi_t$ be the time-$t$ flow of the vector field $\nabla f$. Recall from \S\ref{subsec:conventions} that we have $(r,\theta)$ coordinates on $\{r > 0\}$. Since $f$ is a function of $r$ on $\{r > 0\}$, we have $\Phi_t(r,\theta) = (\phi_t(r),\theta)$ where $\phi_t$ is the solution to
\begin{align}
    \frac{\del}{\del t} \phi_t(r) = f'(\phi_t(r)), \quad \phi_0(r) = r.
\end{align}
We will often use the next basic estimate for $\phi_t$:
\begin{lemma} \label{lem:phi-bounds}
    There exists $C>0$ such that for all $r$ sufficiently large and all $t \in [0,\frac{9r}{10}]$, we have
    \begin{align} \label{eq:phi-2side-bds}
        r-t-C \leq \phi_t(r) \leq r-t+C.
    \end{align}
\end{lemma}
\begin{proof}
    Since $f'(r) = -1 + \O(r^{-1})$, for all sufficiently large $r$ we have $f'(r) \geq -1.01$ and so for all $s \in [0,\frac{9r}{10}]$,
    \begin{align}
        \phi_s(r) \geq r - 1.01s > 0.
    \end{align}
    Then for all large $r$ and $t \in [0,\frac{9r}{10}]$,
    \begin{align}
        \phi_t(r) &= r + \int_0^t f'(\phi_s(r)) \, ds \leq r + \int_0^t \left( -1 + \frac{C}{\phi_s(r)} \right) \, ds \leq r - t + \int_0^{\frac{9r}{10}} \frac{C}{r - 1.01s} \, ds \\
        &= r-t - C\log\left( r - 1.01s \right)\Big|_{s=0}^{s=\frac{9r}{10}} = r-t - C \log \left( 1 - 1.01 \times 9/10 \right) = r-t+C.
    \end{align}
    This proves the upper bound in \eqref{eq:phi-2side-bds}. The lower bound is obtained similarly, using $f'(\phi_s(r)) \geq -1 - \frac{C}{\phi_s(r)}$ instead in the estimation.
\end{proof}

Let us introduce some further setup. For each $\rho > 0$ and $t \in \R$, define the metric
\begin{align}
    \hat{g}^{(\rho)}(t) := \rho^{-1} \Phi_{\rho t}^*g.
\end{align}
Also, given any function $u: \overline{B}_\rho \to \R$, define
\begin{align}
    \hat{u}^{(\rho)}(x,t) := (\Phi_{\rho t}^*u)(x) = u(\Phi_{\rho t}(x)).
\end{align}
Then $\hat{u}^{(\rho)}$ is defined for all $(x,t) \in \overline{B}_\rho \times [0,\infty)$; however for the most part we will consider domains of the form $\overline{\Omega}^\rho \times [0,\frac{7}{8}]$ and $\Omega^\rho \times [0,\frac{7}{8}]$, where
\begin{align}
    \overline{\Omega}^\rho &:= \{\rho - 10\sqrt{\rho} \leq r \leq \rho\}, \\
    \Omega^\rho &:= \{\rho - 10\sqrt{\rho} < r < \rho\}.
\end{align}
Fix a large $\rho_0 > 0$. Then for each $\rho > 0$, define $\psi_\rho: \R \to \R$ and the diffeomorphism $\Psi_\rho: \Omega^{\rho_0} \to \Omega^\rho$ by
\begin{align}
    \psi_\rho(r) &= \rho + (r-\rho_0) \sqrt{\frac{\rho}{\rho_0}}, \\
    \Psi_\rho(r,\theta) &= (\psi_\rho(r),\theta).
\end{align}

There exists the following transformation which turns drift-harmonic functions into solutions of a heat equation with time-dependent metric. This transformation is implied in the work of Brendle \cite{brendle}.
\begin{lemma} \label{lem:conversion-to-parabolic}
    Let $\rho > 0$ and suppose $u: \overline{B}_\rho \to \R$ satisfies $\L_f u = 0$ on $B_\rho$. Then
    \begin{align} \label{eq:56018}
        (\del_t - \Delta_{\Psi_\rho^* \hat{g}^{(\rho)}(t)}) \Psi_\rho^*\hat{u}^{(\rho)} = 0 \quad \text{on } \Omega^{\rho_0} \times (0,\tfrac{7}{8}].
    \end{align}
\end{lemma}
\begin{proof}
    Unfolding definitions, we directly compute that $\del_t \hat{u}^{(\rho)}(x,t) = (\Delta_{\hat{g}^{(\rho)}(t)} \hat{u}^{(\rho)})(x,t)$ at any $(x,t) \in \Omega^{\rho_0} \times (0,\frac{7}{8}]$. The lemma follows from pulling this back by the diffeomorphism $\Psi_\rho: \Omega^{\rho_0} \to \Omega^{\rho}$.
\end{proof}

Lemma \ref{lem:unif-ctrl} shows that the coefficients of the equation \eqref{eq:56018} are uniformly bounded in $\rho$. This enables the application of standard parabolic estimates, leading to scale-invariant estimates for $u$. The remainder of this subsection will state these estimates, with proofs deferred to Appendix \ref{sec:appC}.

For each (large) $\rho > 0$ and $\tau \in (0,1/2)$, define the domains
\begin{align}
    \overline{\Omega}^{\rho}_{\tau} &:= \{ \rho - (1-\tau)\sqrt{\rho} \leq r \leq \rho - \tau\sqrt{\rho} \} \subset \Omega^{\rho}, \\
    \Omega^{\rho}_{\tau} &:= \{ \rho - (1-\tau)\sqrt{\rho} < r < \rho - \tau\sqrt{\rho} \}.
\end{align}

\begin{theorem} \label{thm:parabolicSchauder}
    For each $\alpha \in (0,1)$ and $\tau \in (0,\frac{1}{2})$, there exists $C = C(\alpha,\tau)$ such that if $\rho > 0$ and $\L_f u = 0$ on $B_\rho$, then $w := \Psi_\rho^* \hat{u}^{(\rho)}$ satisfies
    \begin{align}
        \norm{w}_{C^{2+\alpha,1+\frac{\alpha}{2}}(\overline{\Omega}^{\rho_0}_{\tau} \times [\tau,\frac{7}{8}]; \Psi_{\rho_0}^*\hat{g}^{(\rho_0)}(0))} \leq C \norm{w}_{L^\infty(\overline{\Omega}^{\rho_0} \times [0,\frac{7}{8}])}.
    \end{align}
\end{theorem}

The parabolic H\"older norm is defined in e.g. \cite{krylov}. All we need is a well-known compactness property:
\begin{theorem} \label{thm:pHolder-cpt-embedding}
    Let $\alpha \in (0,1)$ and let $K \subset \overline{\Omega}^{\rho_0} \times [0,\frac{7}{8}]$ be a compact set. Then $C^{2+\alpha,1+\frac{\alpha}{2}}(K;\Psi_{\rho_0}^*\hat{g}^{(\rho_0)}(0))$ embeds compactly in $C^{2,1}(K;\Psi_{\rho_0}^*\hat{g}^{(\rho_0)}(0))$, where the latter is the space of functions $w: K \to \R$ such that
    \begin{align}
        \norm{w}_{C^{2,1}(K;\Psi_{\rho_0}^*\hat{g}^{(\rho_0)}(0))} := \norm{w}_{C^2(K;\Psi_{\rho_0}^*\hat{g}^{(\rho_0)}(0))} + \norm{\del_t w}_{L^\infty(K)} < \infty.
    \end{align}
\end{theorem}

Using Theorems \ref{thm:parabolicSchauder} and \ref{thm:pHolder-cpt-embedding}, one can deduce:

\begin{theorem} \label{thm:scaled-back-schauder}
    For each $\tau \in (\frac{1}{2},1)$, there exists $C = C(\tau)$ such that if $\rho > 0$ and $\L_f u = 0$ on $B_\rho$, then
    \begin{align}
        \sup_{\{\frac{1}{4}\rho \leq r \leq \tau\rho\}} \left( \sqrt{r} |\nabla u| + r \left|\inner{\nabla u}{\nabla r}\right| + r|\nabla^2 u| \right) \leq C \sup_{\overline{B}_\rho} |u|.
    \end{align}
\end{theorem}

A straightforward consequence of Theorem \ref{thm:scaled-back-schauder} is the following.

\begin{corollary} \label{cor:deriv-growthbounds}
    For each $u \in \H_d$, there exists $C > 0$ such that for all $\rho > 0$,
    \begin{align}
        \sup_{B_\rho} |\nabla u| &\leq C\rho^{d-\frac{1}{2}}, \\
        \sup_{B_\rho} \left|\inner{\nabla u}{\nabla r}\right| &\leq C\rho^{d-1}, \\
        \sup_{B_\rho} |\nabla^2 u| &\leq C\rho^{d-1}.
    \end{align}
\end{corollary}

The next estimate is a mean value inequality that will help us turn $I$ bounds into pointwise bounds.

\begin{theorem} \label{thm:meanValue}
    For each $\tau \in (0,\frac{1}{2})$, there exists $C = C(\tau)$ such that if $\rho > 0$ and $\L_f u = 0$ on $B_\rho$, then
    \begin{align}
        \sup_{\{ r \leq (1-\tau)\rho \}} u^2 \leq C \rho^{-\frac{n+1}{2}} \int_{\frac{1}{32}\rho}^{\rho} s^{\frac{n-1}{2}} I_u(s) \, ds.
    \end{align}
\end{theorem}

\section{Top-level view of the proof of Theorem \ref{thm:main-expanded}} \label{sec:dividing}

Let $(\Sigma, g_X)$ be the asymptotic cross-section of the AP manifold $(M^n,g,r)$. Since $g_X$ is a $C^{1,\alpha}$ metric (see Remark \ref{rmk:c1a-metric}), its Laplacian $\Delta_{g_X}$ exists classically with $C^{0,\alpha}$ coefficients and obeys the standard spectral theory. Let $0 = \lambda_1 < \lambda_2 < \lambda_3 < \cdots \to \infty$ be the \emph{distinct} eigenvalues of $-\Delta_{g_X}$, with respective (finite) multiplicities $1=m_1, m_2, m_3, \cdots$. We also continue to take $f \in C^\infty(M)$ satisfying Assumption \ref{assump:f}, and refer to the spaces of drift-harmonic functions defined in \eqref{eq:harm1}--\eqref{eq:harm4}.

This section records the main steps leading to our central result, Theorem \ref{thm:main-expanded}. This is done in \S\ref{subsec:proof-main} after setting things up precisely in \S\ref{subsec:setup-4}.

\subsection{Setup and some definitions} \label{subsec:setup-4}


\begin{definition} \label{def:link-projections}
    The space $L^2(g_X)$ is the Hilbert space associated to the inner product
    \begin{align}
        \inner{u}{v}' &:= \int_\Sigma uv \, \dvol_{g_X},
    \end{align}
    with respect to which $\Delta_{g_X}$ is symmetric.
    If $u,v: M \to \R$ are functions and $\rho > 0$, then we define
    \begin{align}
        \inner{u}{v}_\rho' &:= \inner{u|_{\{r=\rho\}}}{v|_{\{r=\rho\}}}' = \int_\Sigma u(\rho,\cdot) v(\rho,\cdot) \, \dvol_{g_X}, \label{eq:uv'} \\
        \norm{u}_\rho' &:= \sqrt{\inner{u}{u}_\rho'},
    \end{align}
    where in \eqref{eq:uv'} we are using $(r,\theta)$ coordinates on $\rho > 0$ (see \S\ref{subsec:conventions}).
    For each $k \in \N$, we also define:
    \begin{itemize}
        \item Let $\underline{\V}_k, \V_k$ and $\overline{\V}_k$ be the direct sum of eigenspaces of $-\Delta_{g_X}$ with eigenvalues $\leq \lambda_k$, $=\lambda_k$ and $\geq \lambda_k$ respectively.
        \item For each $\phi \in C^\infty(M)$ and $\rho \geq 0$, let $\underline{\P}_{\rho,k}\phi, \P_{\rho,k}\phi$ and $\overline{\P}_{\rho,k}\phi$ be the $L^2(g_X)$-orthogonal projections of $\phi|_{\{r=\rho\}}$ (defined on $\{r=\rho\} \cong \Sigma$) onto $\underline{\V}_k, \V_k$ and $\overline{\V}_k$ respectively.
    \end{itemize}
\end{definition}

We note the following:
\begin{itemize}
    \item For each $k \in \N$, one has $\dim \V_k = m_k$ and $\dim \underline{\V}_k = m_1 + m_2 + \cdots + m_k$.
    \item If $\phi \in C^\infty(M)$, then $\phi|_{\{r=\rho\}} = \underbrace{\underline{\P}_{\rho,k-1}\phi + \P_{\rho,k}\phi}_{= \underline{\P}_{\rho,k}\phi} + \overline{\P}_{\rho,k}\phi$ is an $L^2(g_X)$-orthogonal decomposition.
    \item By Theorem \ref{thm:asymp-link}, there exists $C>0$ such that for all nonzero functions $u,v: M \to \R$,
    \begin{align}
        \left| \frac{\norm{u}_\rho'}{\norm{u}_\rho} - 1 \right| &\leq C\rho^{-\mu} \quad \text{and} \quad \left| \frac{\inner{u}{v}_\rho'}{\norm{u}_\rho' \norm{v}_\rho'} - \frac{\inner{u}{v}_\rho}{\norm{u}_\rho \norm{v}_\rho} \right| \leq C\rho^{-\mu}. \label{eq:inners-close}
    \end{align}
\end{itemize}

Next, we define a class of drift-functions which are asymptotically controlled in a precise manner.
\begin{definition} \label{def:S_k}
    For each $j \in \N$, $C > 0$ and $\tau > 0$, define $\mathring{\mathcal{S}}_{\lambda_j}(C,\tau)$ as the set of nonzero drift-harmonic functions $u \in \H$ such that for all $\rho > 0$,
    \begin{enumerate}[label=(\roman*)]
        \item $u$ $(C,\tau)$-asymptotically separates variables.
        \item $\lambda_j - C\rho^{-\tau} \leq U_u(\rho) \leq \lambda_j + C\rho^{-\tau}$.
        \item $\lambda_j - C\rho^{-\tau} \leq Q_u(\rho) \leq \lambda_j + C\rho^{-\tau}$.
        \item $\frac{\norm{\P_{\rho,j}u}_\rho'}{\norm{u}_\rho'} \geq 1-C\rho^{-\tau}$.
    \end{enumerate}
\end{definition}

The two-sided frequency bound in this definition pins down an exact growth rate for $u$:

\begin{lemma} \label{lem:S-contained-H}
    For all $j \in \N$, $C > 0$ and $\tau > 0$, we have $\mathring{\mathcal{S}}_{\lambda_j}(C,\tau) \subset \mathring{\H}_{\lambda_j}$.
\end{lemma}
\begin{proof}
    Let $u \in \mathring{\mathcal{S}}_{\lambda_j}(C,\tau)$. By definition,
    \begin{align}
        \lambda_j - C\rho^{-\tau} \leq U_u(\rho) \leq \lambda_j + C\rho^{-\tau} \quad \text{for all } \rho > 0.
    \end{align}
    Then by Corollary \ref{cor:I-almost-mono}, there exists $C>0$ such that
    \begin{align}
        I_u(\rho_2) \leq e^{C\rho_1^{-\tau}} \left( \frac{\rho_2}{\rho_1} \right)^{2\lambda_j} I_u(\rho_1) \quad \text{for all } \rho_2 > \rho_1 \geq 1.
    \end{align}
    Iterating this, we get that for each $i \in \N$,
    \begin{align}
        I_u(2^i) &\leq e^{C(2^{i-1})^{-\tau}} 2^{2\lambda_j} I_u(2^{i-1}) \leq \cdots \leq e^{C((2^{i-1})^{-\tau}+(2^{i-2})^{-\tau}+\ldots+2^{-\tau}+1)} (2^{2\lambda_j})^i I_u(1) \\
        &= e^{\frac{C}{1-2^{-\tau}}} (2^i)^{2\lambda_j} I_u(1) \leq C(2^i)^{2\lambda_j},
    \end{align}
    where the last $C$ depends on $u$ but not on $i$. By Corollary \ref{cor:I-almost-mono} again, it follows that
    \begin{align}
        I_u(\rho) \leq C\rho^{2\lambda_j} \quad \text{for all } \rho \geq 1.
    \end{align}
    Then by Theorem \ref{thm:meanValue} and the maximum principle, we have
    \begin{align}
        |u| \leq C(r^{\lambda_j}+1),
    \end{align}
    so $u \in \H_{\lambda_j}$. Meanwhile, since $\lim_{\rho\to\infty} U_u(\rho) = \lambda_j$, Corollary \ref{cor:liminf-upper} gives $u \notin \H_{\lambda_j-\epsilon}$ for all $\epsilon > 0$. Hence $u \in \mathring{\H}_{\lambda_j}$.
\end{proof}

We also need a condition addressing the existence of sufficiently many drift-harmonic functions with desirable properties.

\begin{definition} \label{def:E-condition}
    Let $\ell \in \N$. We say that $(E_\ell)$ holds if there exist $C,\tau>0$ and collections $\mathring{\B}_{\lambda_1}, \ldots, \mathring{\B}_{\lambda_\ell} \subset \H$ of global drift-harmonic functions such that for each $j \in \{1,2,\ldots,\ell\}$,
    \begin{enumerate}[label=(\roman*)]
        \item $\mathring{\B}_{\lambda_j} \subset \mathring{\mathcal{S}}_{\lambda_j}(C,\tau)$. (Hence $\mathring{\B}_{\lambda_j} \subset \mathring{\H}_{\lambda_j}$ by Lemma \ref{lem:S-contained-H}.)
        \item $|\mathring{\B}_{\lambda_j}| = m_j$.
        \item There is a point $p_0 \in M$ such that for all $j \geq 2$ and $v \in \mathring{\B}_{\lambda_j}$, we have $v(p_0) = 0$.
        \item The set $\B_{\lambda_\ell} := \bigcup_{j=1}^{\ell} \mathring{\B}_{\lambda_j}$ is linearly independent, and every distinct pair of functions $u,v \in \B_{\lambda_\ell}$ are $(C,\tau)$-asymptotically orthogonal.
    \end{enumerate}
\end{definition}

\begin{remark}
    $\B_{\lambda_\ell}$ is analogous to the basis $\B_{\lambda_\ell}(P)$ in the model situation of Proposition \ref{prop:model}.
\end{remark}

\subsection{Proof of the main theorem} \label{subsec:proof-main}

Here we present the main steps in the proof of Theorem \ref{thm:main-expanded}.

In \S\ref{sec:asymp-ctrl}, we will asymptotically control drift-harmonic functions. This is step (A) in \S\ref{subsec:elements-intro}:
\begin{theorem} \label{thm:asymp-ctrl}
    Let $\ell \in \N$ and suppose $(E_\ell)$ holds, giving collections $\mathring{\B}_{\lambda_j} \subset \H$ for $j \leq \ell$. Then for every $u \in \H_{\lambda_{\ell+1}}^+$ not in the span of $\B_{\lambda_\ell} := \bigcup_{j=1}^\ell \mathring{\B}_{\lambda_j}$, there exist $C,\tau > 0$ such that
    \begin{enumerate}[label=(\alph*)]
        \item $u \in \mathring{\mathcal{S}}_{\lambda_{\ell+1}}(C,\tau)$,
        \item For every $v \in \B_{\lambda_\ell}$, the functions $u$ and $v$ are $(C,\tau)$-asymptotically orthogonal.
    \end{enumerate}
\end{theorem}
In \S\ref{sec:construction}, we will use Theorem \ref{thm:asymp-ctrl} to construct drift-harmonic functions. This is step (C) in \S\ref{subsec:elements-intro}:
\begin{theorem} \label{thm:construction}
    Let $\ell \in \N$. If $(E_\ell)$ holds, then so does $(E_{\ell+1})$.
\end{theorem}

Using Theorems \ref{thm:asymp-ctrl} and \ref{thm:construction}, the proof of Theorem \ref{thm:main-expanded} follows easily:
\begin{proof}[Proof of Theorem \ref{thm:main-expanded}]
    Note that $(E_1)$ holds with $\B_{\lambda_1} = \mathring{\B}_{\lambda_1} = \{1\}$. By Theorem \ref{thm:construction}, $(E_\ell)$ holds for each $\ell \in \N$, giving linearly independent sets $\B_{\lambda_\ell} \subset \H_{\lambda_\ell}$. By Theorem \ref{thm:asymp-ctrl}, any $u \in \H$ outside the span of $\B_{\lambda_\ell}$ cannot belong to $\H_{\lambda_\ell}$. Thus, $\H_{\lambda_\ell}$ is spanned by $\B_{\lambda_\ell}$, and $\B_{\lambda_\ell}$ is a basis for $\H_{\lambda_\ell}$ for each $\ell \in \N$.
    
    Let $d \in \R$ and $u \in \H_d$. Let $\ell \geq 0$ be the smallest integer such that $u \in \H_{\lambda_{\ell+1}}$. If $\ell = 0$, then $u$ is constant and (a) and (b) in the theorem hold for $u$. Otherwise, $\ell \geq 1$ and $u$ is nonconstant. Then $u \notin \H_{\lambda_\ell} = \operatorname{span}\B_{\lambda_\ell}$, so by Theorem \ref{thm:asymp-ctrl} we have $u \in \mathring{\mathcal{S}}_{\lambda_{\ell+1}}(C,\tau)$ for some $C, \tau > 0$. This implies part (a) of the theorem by definition, whereas part (b) follows from Lemma \ref{lem:S-contained-H}.

    Using part (b) of the theorem, one has $\H_d = \H_{\lambda_\ell}$ where $\ell$ is the largest number such that $\lambda_\ell \leq d$.
    Then
    \begin{align}
        \dim \H_d = |\B_{\lambda_\ell}| = \sum_{\{k \in \N : \lambda_k \leq d \}} m_k,
    \end{align}
    proving part (c) of the theorem. Part (d) of the theorem follows from the fact that $(E_\ell)$ holds.
\end{proof}

\section{Asymptotic control of drift-harmonic functions: proof of Theorem \ref{thm:asymp-ctrl}} \label{sec:asymp-ctrl}

The objective of this section is to prove Theorem \ref{thm:asymp-ctrl}. As such, throughout this section we fix an $\ell \in \N$ and assume that $(E_\ell)$ holds.
So there exist $C,\tau > 0$ and collections $\mathring{\B}_{\lambda_1},\ldots,\mathring{\B}_{\lambda_\ell} \subset \H$ so that for each $j \in \{1,2,\ldots,\ell\}$, items (i)--(iv) of Definition \ref{def:E-condition} hold. We may assume $\tau < \mu/2$. We also define
\begin{itemize}
    \item For each $k \in \{1,\ldots,\ell\}$, let $\B_{\lambda_k} := \bigcup_{j=1}^k \mathring{\B}_{\lambda_j}$.
    \item $d_\ell := \max_{v \in \B_{\lambda_\ell}} \max_{\rho > 0} U_v(\rho) < \infty$.
\end{itemize}
The number $d_\ell$ is finite because each of the finitely many $v \in \B_{\lambda_\ell}$ belongs to $\mathring{\B}_{\lambda_j} \subset \mathring{\mathcal{S}}_{\lambda_j}(C,\tau)$ for some $j \in \{1,2,\ldots,\ell\}$, so $U_v(\rho)$ is bounded.


\subsection{Outline for this section} \label{subsec:outline-asymp-ctrl}


In \S\ref{subsec:almost-eigenspace}, we show that any function $\phi \in \H$ which is almost orthogonal to $\B_{\lambda_k}$ on a level set $\{r=\rho\}$ must satisfy a lower bound on $Q_\phi(\rho)$. In \S\ref{subsec:almost-orthog-implies-QUI}, we assume that $\phi$ is orthogonal to $\B_{\lambda_\ell}$ on a fixed level set $\{r=\bar{\rho}\}$. By iterating preservation of almost orthogonality (Corollary \ref{cor:pres-of-inner-prod}) outwards and using the results of \S\ref{subsec:almost-eigenspace}, we get lower bounds for $Q_\phi$. Combining this with the ODE for $U_\phi$ (Lemma \ref{lem:D'I'}), we obtain lower bounds for $U_\phi$ and $I_\phi$.

In \S\ref{subsec:lin-indep-freq-bds}, we prove similar lower bounds for a function $u \in \H_{\lambda_{\ell+1}}^+$ outside the span of $\B_{\lambda_\ell}$, as well as bounds for other quantities introduced in \S\ref{subsec:freq-and-related}. Finally, in \S\ref{subsec:almost-sep-asymp-ctrl} we bring in the results of \S\ref{subsec:rellich-necas} to prove that $U_u$ is almost monotone. This will provide the asymptotic control on $u$ claimed by Theorem \ref{thm:asymp-ctrl}.

\subsection{Projections and Rayleigh quotients over level sets} \label{subsec:almost-eigenspace}

This subsection records several relations between projections, orthogonality, and Rayleigh quotients over level sets. We will use the setup from \S\ref{subsec:setup-4}.

\begin{lemma} \label{lem:B-almost-eigenfunc-induction}
    There exist $C,\tau,R_1> 0$ such that for all $k \in \{1,\ldots,\ell\}$ and nonzero $\phi \in C^\infty(M)$, we have
    \begin{align} \label{eq:proj-upper-bd}
        \frac{\norm{\underline{\P}_{\rho,k} \phi}_\rho'}{\norm{\phi}_\rho'} \leq C\left( \max_{v \in \B_{\lambda_k}} \frac{\left| \inner{\phi}{v}_\rho \right|}{\norm{\phi}_\rho \norm{v}_\rho} + \rho^{-\tau} \right) \quad \text{for all } \rho \geq R_1.
    \end{align}
\end{lemma}
\begin{proof}
    Let $k \in \{1,\ldots,\ell\}$.
    For each $v \in \B_{\lambda_k}$, there is a unique $j \leq k$ such that $v \in \mathring{\B}_{\lambda_j} \subset \mathring{\mathcal{S}}_{\lambda_j}(C,\tau)$, so
    \begin{align} \label{eq:low-projs}
        \frac{\norm{\underline{\P}_{\rho,k} v}_\rho'}{\norm{v}_\rho'} \geq \frac{\norm{\P_{\rho,j} v}_\rho'}{\norm{v}_\rho'} \geq 1-C\rho^{-\tau} \quad \text{for all } \rho > 0.
    \end{align}
    Consider the collection
    \begin{align}
        \underline{\B}_{\rho,\lambda_k} := \{ \underline{\P}_{\rho,k} v \mid v \in \B_{\lambda_k} \} \subset \underline{\V}_k.
    \end{align}
    Combining the hypothesis $(E_\ell)$ with \eqref{eq:inners-close}, we see that every distinct pair of functions in $\B_{\lambda_k}$ is $C\rho^{-\tau}$-almost orthogonal \emph{with respect to $\inner{\cdot}{\cdot}_\rho'$} on $\{r=\rho\}$ for each $\rho \geq 1$. Using \eqref{eq:low-projs}, the last sentence remains true with $\B_{\lambda_k}$ replaced by $\underline{\B}_{\rho,\lambda_k}$. For $\rho$ sufficiently large (say $\rho \geq R_1$), this implies $\underline{\B}_{\rho,\lambda_k}$ is linearly independent.
    As $\dim \underline{\V}_k = \sum_{j=1}^k m_j = |\underline{\B}_{\rho,\lambda_k}|$, it follows that $\underline{\B}_{\rho,\lambda_k}$ is a $C\rho^{-\tau}$-almost orthogonal (with respect to $\inner{\cdot}{\cdot}_\rho'$) basis for $\underline{\V}_k$.
    Hence, for each $\rho \geq R_1$ and each function $u \in L^2(\{r=\rho\})$ which satisfies $\norm{u}_\rho' = 1$, we have
    \begin{align} \label{eq:diff-projections}
        \norm{\underline{\P}_{\rho,k} u - \sum_{v \in \underline{\B}_{\lambda_k}} \frac{\inner{u}{v}_\rho'}{\inner{v}{v}_\rho'} v}_\rho' \leq C\rho^{-\tau}.
    \end{align}
    Suppose $\phi \in C^\infty(M)$ is nonzero.
    Using \eqref{eq:inners-close}, \eqref{eq:low-projs}, and \eqref{eq:diff-projections}, it holds for all $\rho \geq R_1$ that
    \begin{align}
        \frac{\norm{\underline{\P}_{\rho,k} \phi}_\rho'}{\norm{\phi}_\rho'} &\leq \frac{1}{\norm{\phi}_\rho'} \left( \norm{\underline{\P}_{\rho,k}\phi - \sum_{v \in \underline{\B}_{\lambda_k}} \frac{\inner{\phi}{v}_\rho'}{\inner{v}{v}_\rho'} v}_\rho' + \sum_{v \in \underline{\B}_{\lambda_k}} \left| \frac{\inner{\phi}{v}_\rho'}{\inner{v}{v}_\rho'} \right| \norm{v}_\rho' \right) \\
        &\leq C\rho^{-\tau} + (1+C\rho^{-\mu}) \sum_{v \in \B_{\lambda_k}} \left| \frac{\inner{\phi}{v}_\rho'}{\inner{v}{v}_\rho'} \right| \frac{\norm{v}_\rho'}{\norm{\phi}_\rho'} \\
        &\leq C\rho^{-\tau} + C \max_{v \in \B_{\lambda_k}} \frac{\left| \inner{\phi}{v}_\rho \right|}{\norm{\phi}_\rho \norm{v}_\rho}.
    \end{align}
    Maximizing $C$ over all $k \in \{1,\ldots,\ell\}$, the lemma follows.
\end{proof}

The point of the last lemma is that if \eqref{eq:proj-upper-bd} is small, then we get Rayleigh quotient lower bounds:
\begin{lemma} \label{lem:small-proj-Q-bds}
    Given $k \in \N$, $C_0>0$ and $\tau \in (0,\mu)$, there exist $\delta = \delta(k)>0$ and $C_1 = C_1(C_0,k)>0$ such that for each $\rho > 0$ and nonzero $\phi: \{r=\rho\} \to \R$,
    \begin{enumerate}[label=(\alph*)]
        \item If $\frac{\norm{\underline{\P}_{\rho,k}\phi}_\rho'}{\norm{\phi}_\rho'} \leq C_0 \rho^{-\tau}$, then $Q_\phi(\rho) \geq \lambda_{k+1}-C_1\rho^{-\tau}$.
        \item If $\frac{\norm{\underline{\P}_{\rho,k}\phi}_\rho'}{\norm{\phi}_\rho'} \leq \delta$, then $Q_\phi(\rho) \geq \frac{1}{2}(\lambda_k+\lambda_{k+1})$.
    \end{enumerate}
\end{lemma}
\begin{proof}
    Let $k \in \N$, and let $\rho > 0$ and $\phi: \{r=\rho\} \to \R$ be nonzero. We may assume that $\norm{\phi}_\rho' = 1$. Since $\phi|_{\{r=\rho\}} = \underline{\P}_{\rho,k}\phi + \overline{\P}_{\rho,k+1}\phi$ is a $\inner{\cdot}{\cdot}_\rho'$-orthogonal decomposition on $\{r=\rho\}$ preserved by $\Delta_{g_X}$,
    \begin{align}
        (1+C\rho^{-\mu}) Q_\phi(\rho) &\geq \rho \int_{\{r=\rho\}} |\nabla^\top \phi|^2 \, \dvol_{g_X} = -\inner{\phi}{\Delta_{g_X}\phi}_\rho' \\
        &= -\inner{\underline{\P}_{\rho,k}\phi}{\Delta_{g_X} (\underline{\P}_{\rho,k}\phi)}' - \inner{\overline{\P}_{\rho,k+1}\phi}{\Delta_{g_X} (\overline{\P}_{\rho,k+1}\phi)}' \\
        &\geq \lambda_{k+1} \norm{\overline{\P}_{\rho,k+1}\phi}_\rho'^2 = \lambda_{k+1} (1-\norm{\underline{\P}_{\rho,k}\phi}_\rho'^2) \\
        &\geq \lambda_{k+1} (1-\norm{\underline{\P}_{\rho,k}\phi}_\rho').
    \end{align}
    The claims (a) and (b) follow easily from this.
\end{proof}

\begin{proposition} \label{prop:B-almost-eigenfunc}
    There exist $\tau, R_1 > 0$ such that if $\phi \in C^\infty(M)$ satisfies for some $C>0$
    \begin{enumerate}[label=(\roman*)]
        \item $Q_\phi(\rho) \leq \lambda_{\ell+1} + C\rho^{-\tau}$ for all $\rho > 0$,
        \item For each $v \in \B_{\lambda_\ell}$, $\phi$ and $v$ are $(C,\tau)$-asymptotically orthogonal,
    \end{enumerate}
    then
    \begin{align}
        \frac{\norm{\P_{\rho,\ell+1} \phi}_\rho'}{\norm{\phi}_\rho'} \geq 1-\tilde{C} \rho^{-\tau} \quad \text{for all } \rho > 0
    \end{align}
    where $\tilde{C} = \tilde{C}(C)$.
\end{proposition}
\begin{proof}
    By the assumptions (i) and (ii), it holds for all $\rho > 0$ that
    \begin{align} \label{eq:000002}
        \frac{\rho \int_{\{r=\rho\}} |\nabla^\top \phi|^2 \, \dvol_{g_X}}{\int_{\{r=\rho\}} \phi^2 \, \dvol_{g_X}} \leq (1+C\rho^{-\mu}) Q_\phi(\rho) \leq \lambda_{\ell+1} + C\rho^{-\tau}
    \end{align}
    and
    \begin{align}
        \frac{\left| \inner{\phi}{v}_\rho \right|}{\norm{\phi}_\rho \norm{v}_\rho} \leq C\rho^{-\tau} \quad \text{for all } v \in \B_{\lambda_\ell}.
    \end{align}
    The latter estimate and Lemma \ref{lem:B-almost-eigenfunc-induction} give $\tau,R_1 > 0$ and $\tilde{C} = \tilde{C}(C)>0$ such that
    \begin{align} \label{eq:000001}
        \frac{\norm{\underline{\P}_{\rho,\ell} \phi}_\rho'}{\norm{\phi}_\rho'} \leq \tilde{C} \rho^{-\tau} \quad \text{for all } \rho \geq R_1.
    \end{align} 
    Since $\phi|_{\{r=\rho\}} = \underline{\P}_{\rho,\ell} \phi + \P_{\rho,\ell+1}\phi + \overline{\P}_{\rho,\ell+2}\phi$ is an $\inner{\cdot}{\cdot}'$-orthogonal decomposition preserved by $\Delta_{g_X}$, we have for all $\rho \geq R_1$
    \begin{align}
        \frac{\rho \int_{\{r=\rho\}} |\nabla^\top \phi|^2 \, \dvol_{g_X}}{\int_{\{r=\rho\}} \phi^2 \, \dvol_{g_X}} &\geq \lambda_{\ell+1} \frac{\norm{\P_{\rho,\ell+1}\phi}_\rho'^2}{\norm{\phi}_\rho'^2} + \lambda_{\ell+2} \frac{\norm{\overline{\P}_{\rho,\ell+2}\phi}_\rho'^2}{\norm{\phi}_\rho'^2} \\
        &\geq \lambda_{\ell+1} \frac{\norm{\P_{\rho,\ell+1}\phi}_\rho'^2}{\norm{\phi}_\rho'^2} + \lambda_{\ell+2} \left( 1 - \tilde{C}^2\rho^{-2\tau} - \frac{\norm{\P_{\rho,\ell+1}\phi}_\rho'^2}{\norm{\phi}_\rho'^2} \right)
    \end{align}
    where the last inequality uses \eqref{eq:000001}. Combining this with \eqref{eq:000002}, rearranging, and using that $\lambda_{\ell+2} > \lambda_{\ell+1}$, the proposition follows.
\end{proof}

The next corollary proceeds along similar lines, though it will not be used until \S\ref{sec:construction}.
\begin{corollary} \label{cor:B-lin-combs-proj}
    There exist $C,\tau > 0$ such that for all nonzero $\phi \in \operatorname{span}(\B_{\lambda_\ell} \setminus \B_{\lambda_1})$ and $\rho \geq 1$, one has
    \begin{align}
        \frac{\norm{\P_{\rho,j}\phi}_\rho'}{\norm{\phi}_\rho'} \leq C\rho^{-\tau} \quad \text{for } j=1 \text{ and all } j \geq \ell+1.
    \end{align}
\end{corollary}
\begin{proof}
    We will prove this assuming $\phi = au + bv$ for some $a,b \in \R$ and distinct $u,v \in \B_{\lambda_\ell} \setminus \B_{\lambda_1}$. The general case is similar. Let $j = 1$ or $j \geq \ell+1$. Since $(E_\ell)$ holds by assumption, Definition \ref{def:S_k} gives
    \begin{align} \label{eq:ccc1}
        \frac{\norm{\P_{\rho,j}u}_\rho'}{\norm{u}_\rho'} \leq C\rho^{-\tau}, \quad \frac{\norm{\P_{\rho,j}v}_\rho'}{\norm{v}_\rho'} \leq C\rho^{-\tau}.
    \end{align}
    Using \eqref{eq:inners-close} and the $(C,\tau)$-asymptotic orthogonality between $u$ and $v$,
    \begin{align}
        a^2 \norm{u}_\rho'^2 + b^2 \norm{v}_\rho'^2 &= \norm{\phi}_\rho'^2 - 2ab \inner{u}{v}_\rho' \leq \norm{\phi}_\rho'^2 + 2|a||b| C\rho^{-\tau} \norm{u}_\rho' \norm{v}_\rho' \\
        &\leq \norm{\phi}_\rho'^2 + C\rho^{-\tau} (a^2 \norm{u}_\rho'^2 + b^2 \norm{v}_\rho'^2),
    \end{align}
    which implies
    \begin{align} \label{eq:ccc2}
        a^2 \norm{u}_\rho'^2 + b^2 \norm{v}_\rho'^2 \leq (1+C\rho^{-\tau}) \norm{\phi}_\rho'^2.
    \end{align}
    Using \eqref{eq:ccc1} and \eqref{eq:ccc2},
    \begin{align}
        \norm{\P_{\rho,j}\phi}_\rho'^2 &= a^2 \norm{\P_{\rho,j}u}_\rho'^2 + b^2 \norm{\P_{\rho,j}v}_\rho'^2 + 2ab \inner{\P_{\rho,j}u}{\P_{\rho,j}v}_\rho' \\
        &\leq 2a^2 \norm{\P_{\rho,j}u}_\rho'^2 + 2b^2 \norm{\P_{\rho,j}v}_\rho'^2 \leq C\rho^{-2\tau} (a^2 \norm{u}_\rho'^2 + b^2 \norm{v}_\rho'^2) \leq C\rho^{-2\tau} \norm{\phi}_\rho'^2. 
    \end{align}
\end{proof}

\subsection{Almost orthogonality to $\B_{\lambda_k}$ implies $Q,U,I$ lower bounds} \label{subsec:almost-orthog-implies-QUI}

\begin{proposition} \label{prop:small-proj-bounds}
    Given $C_1 > 0$ and $\tau \in (0,\mu)$, there exist $C > 0$ and $R_1 \geq 1$ such that if $k \in \{1,\ldots,\ell\}$ and $\phi \in \H$ is nonzero with
    \begin{align} \label{eq:small-proj-bounds-condition}
        \frac{\left|\inner{\phi}{v}_\rho\right|}{\norm{\phi}_\rho \norm{v}_\rho} \leq C_1\rho^{-\tau} \quad \text{for all } v \in \B_{\lambda_k} \text{ and } \rho \geq R_1,
    \end{align}
    then
    \begin{enumerate}[label=(\alph*)]
        \item $Q_\phi(\rho) \geq \lambda_{k+1} - C\rho^{-2\tau}$ for all $\rho > R_1$,
        \item $U_\phi(\rho) \geq \lambda_{k+1} - C(\rho-R_1)^{-2\tau}$ for all $\rho > R_1$.
        \item $\frac{I_\phi(\rho_2)}{I_\phi(\rho_1)} \geq C^{-1} \left( \frac{\rho_2}{\rho_1} \right)^{2\lambda_{k+1}}$ for all $\rho_2 > \rho_1 \geq 2R_1$.
    \end{enumerate}
\end{proposition}
\begin{proof}
    Lemma \ref{lem:B-almost-eigenfunc-induction} gives $C > 0$ and $R_1 \geq 1$ so that for any $k \in \{1,\ldots,\ell\}$ and nonzero $\phi \in \H$ satisfying \eqref{eq:small-proj-bounds-condition},
    \begin{align}
        \frac{\norm{\underline{\P}_{\rho,k}\phi}_\rho'}{\norm{\phi}_\rho'} &\leq C(C_1\rho^{-\tau} + \rho^{-\tilde{\tau}}) \leq C\rho^{-\tau} \quad \text{for all } \rho \geq R_1.
    \end{align}
    Then by Lemma \ref{lem:small-proj-Q-bds}, for all $\rho \geq R_1$ we have
    \begin{align}
        Q_\phi(\rho) \geq \lambda_{k+1} - C\rho^{-2\tau}. \label{eq:Qphi-lower}
    \end{align}
    This proves part (a) of the proposition. From Lemma \ref{lem:D'I'} and \eqref{eq:simple-CS}, we have
    \begin{align} \label{eq:U-ODE-geq}
        U_\phi'(\rho) &\geq \left( -1 - \frac{C}{\rho} \right) U_\phi(\rho) - \frac{U_\phi(\rho)^2}{\rho} + Q_\phi(\rho) \quad \text{for all } \rho \geq 1.
    \end{align}
    Using \eqref{eq:Qphi-lower}, \eqref{eq:U-ODE-geq}, and Lemma \ref{lem:important-ode-lemma}, it follows that
    \begin{align} \label{eq:U-GEQ-PHI}
        U_\phi(\rho) \geq \lambda_{k+1} - C(\rho-R_1)^{-2\tau} \quad \text{for all } \rho > R_1,
    \end{align}
    where $C = C(C_1,\tau,\lambda_{k+1})$. Maximizing this constant over $k \in \{1,\ldots,\ell\}$, part (b) of the proposition follows. For part (c), let $\rho_2 > \rho_1 \geq 2R_1$. By Lemma \ref{lem:D'I'} and \eqref{eq:U-GEQ-PHI},
    \begin{align}
        \log\left(\frac{I_\phi(\rho_2)}{I_\phi(\rho_1)}\right) &\geq -C\int_{\rho_1}^{\rho_2} \rho^{-\mu-1} \, d\rho + \int_{\rho_1}^{\rho_2} \left( \frac{2\lambda_{k+1}}{\rho} - \frac{C(\rho-R_1)^{-2\tau}}{\rho} \right) \, d\rho \\
        &\geq -C\rho_1^{-\mu} + 2\lambda_{k+1} \log\left( \frac{\rho_2}{\rho_1} \right) - C\int_{\rho_1}^{\infty} (\rho-R_1)^{-2\tau} \rho^{-1} \, ds. \label{eq:IIphi}
    \end{align}
    Since $\max_{s \in [R_1,\infty)} s^{-2\tau}(s+R_1)^{\tau} < \infty$, it follows that $(\rho-R_1)^{-2\tau} \leq C\rho^{-\tau}$ for all $\rho \geq \rho_1 \geq 2R_1$. Inserting this into \eqref{eq:IIphi} gives $\log\left( \frac{I_\phi(\rho_2)}{I_\phi(\rho_1)} \right) \geq -C\rho^{-\tau} + 2\lambda_{k+1} \log\left( \frac{\rho_2}{\rho_1} \right)$.
    Exponentiating this and using that $\rho_2 \geq 2R_1 \geq 2$ yields part (c) of the proposition.
\end{proof}

We also need the following variation on Proposition \ref{prop:small-proj-bounds}:

\begin{proposition} \label{prop:small-proj-bounds-2}
    There exist $C,\delta > 0$ and $R_1 \geq 1$ such that if $k \in \{1,\ldots,\ell\}$ and $\phi \in \H$ is nonzero with
    \begin{align} \label{eq:small-proj-bounds-condition-2}
        \frac{\left|\inner{\phi}{v}_\rho\right|}{\norm{\phi}_\rho \norm{v}_\rho} \leq \delta \quad \text{for all } v \in \B_{\lambda_k} \text{ and } \rho \geq R_1,
    \end{align}
    then
    \begin{enumerate}[label=(\alph*)]
        \item $Q_\phi(\rho) \geq \frac{1}{2}(\lambda_k+\lambda_{k+1})$ for all $\rho > R_1$.
        \item $U_\phi(\rho) \geq \frac{1}{2}(\lambda_k+\lambda_{k+1}) - C(\rho-R_1)^{-\mu}$ for all $\rho > R_1$.
        \item $\frac{I_\phi(\rho_2)}{I_\phi(\rho_1)} \geq C^{-1} \left( \frac{\rho_2}{\rho_1} \right)^{\lambda_k +\lambda_{k+1}}$ for all $\rho_2 > \rho_1 \geq 2R_1$.
    \end{enumerate}
\end{proposition}
\begin{proof}
    Lemma \ref{lem:B-almost-eigenfunc-induction} gives $C,\tau > 0$ and $R_1 \geq 1$ so that for all $k \in \{1,\ldots,\ell\}$ and nonzero $\phi \in \H$ satisfying \eqref{eq:small-proj-bounds-condition-2},
    \begin{align}
        \frac{\norm{\underline{\P}_{\rho,k}\phi}_\rho'}{\norm{\phi}_\rho'} \leq C(\delta + \rho^{-\tau}) \quad \text{for all } \rho \geq R_1.
    \end{align}
    If $R_1$ is sufficiently large and $\delta > 0$ is sufficiently small, depending on $k$, then Lemma \ref{lem:small-proj-Q-bds} gives
    \begin{align}
        Q_\phi(\rho) \geq \frac{1}{2}(\lambda_{k+1}+\lambda_k) \quad \text{for all } \rho \geq R_1. \label{eq:Qphi-lower-2}
    \end{align}
    Minimizing $\delta$ over $k \in \{1,\ldots,\ell\}$, part (a) of the proposition follows. The other two claims follow from the same argument as in the proof of Proposition \ref{prop:small-proj-bounds}, except using \eqref{eq:Qphi-lower-2} in place of \eqref{eq:Qphi-lower}.
\end{proof}

We now show that if a nontrivial drift-harmonic function $\phi$ is orthogonal to $\B_{\lambda_\ell}$ on a sufficiently far $r$-level set, then $I_\phi$ grows at a polynomial rate of at least $2\lambda_{\ell+1}$. This is accomplished by repeatedly applying preservation of almost orthogonality (Corollary \ref{cor:pres-of-inner-prod}), as well as Propositions \ref{prop:small-proj-bounds} and \ref{prop:small-proj-bounds-2}.

\begin{proposition} \label{prop:bootstrap-rayleigh}
    There exist $C,\tau>0$ and $\bar{\rho} \geq 1$ such that if $\phi \in \H$ is nonzero with $\inner{\phi}{v}_{\bar{\rho}} = 0$ for all $v \in \B_{\lambda_\ell}$, then
    \begin{align}
        \frac{\left| \inner{\phi}{v}_\rho \right|}{\norm{\phi}_\rho \norm{v}_\rho} \leq C\rho^{-\tau} \quad \text{for all } \rho \geq \bar{\rho} \text{ and } v \in \B_{\lambda_\ell},
    \end{align}
    and
    \begin{align}
        \frac{I_\phi(\rho_2)}{I_\phi(\rho_1)} \geq C^{-1} \left( \frac{\rho_2}{\rho_1} \right)^{2\lambda_{\ell+1}} \quad \text{for all } \rho_2 > \rho_1 \geq 2\bar{\rho}.
    \end{align}
\end{proposition}
\begin{proof}
    Let $\bar{\rho} \geq 1$, to be chosen large.
    Let $\phi \in \H$ be nonzero and suppose $\inner{\phi}{v}_{\bar{\rho}} = 0$ for all $v \in \B_{\lambda_\ell}$.
    
    Let $v \in \B_{\lambda_\ell}$. By definition, $v$ $(C,\tau)$-asymptotically separates variables, and we are assuming $\tau < \mu/2$. Hence, for every $\rho > \bar{\rho}$, writing $\rho \in (2^{q-1}\bar{\rho}, 2^q\bar{\rho}]$ for some $q \in \N$, we get by Corollary \ref{cor:pres-of-inner-prod}
    \begin{align}
        \frac{\left| \inner{\phi}{v}_{\rho} \right|}{\norm{\phi}_{\rho} \norm{v}_{\rho}} &\leq Ce^{C(2^{q-1}\bar{\rho})^{2\tau-\mu}} (2^{q-1}\bar{\rho})^{-\tau} 2^{4d_\ell+1} + \sqrt{\frac{I_\phi(2^{q-1}\bar{\rho})}{I_\phi(\rho)}} \sqrt{\frac{I_v(\rho)}{I_v(2^{q-1}\bar{\rho})}} \frac{\left| \inner{\phi}{v}_{2^{q-1}\bar{\rho}} \right|}{\norm{\phi}_{2^{q-1}\bar{\rho}} \norm{v}_{2^{q-1}\bar{\rho}}} \\
        &\leq C\bar{\rho}^{-\tau} (2^{-\tau})^{q-1} + \sqrt{\frac{I_\phi(2^{q-1}\bar{\rho})}{I_\phi(\rho)}} \sqrt{\frac{I_v(\rho)}{I_v(2^{q-1}\bar{\rho})}} \frac{\left| \inner{\phi}{v}_{2^{q-1}\bar{\rho}} \right|}{\norm{\phi}_{2^{q-1}\bar{\rho}} \norm{v}_{2^{q-1}\bar{\rho}}}, \label{eq:1iter}
    \end{align}
    where $d_\ell$ was defined at the start of \S\ref{sec:asymp-ctrl}, and the last inequality uses that $(2^{q-1}\bar{\rho})^{2\tau-\mu} \leq \bar{\rho}^{2\tau-\mu} \leq 1$. Notice that we can iterate this on the last factor. Doing this $q$ times and using that $\inner{\phi}{v}_{\bar{\rho}} = 0$, we get
    \begin{align}
        \frac{\left| \inner{\phi}{v}_{\rho} \right|}{\norm{\phi}_{\rho} \norm{v}_{\rho}} &\leq C\bar{\rho}^{-\tau} \Bigg\{ (2^{-\tau})^{q-1} + \sqrt{\frac{I_\phi(2^{q-1}\bar{\rho})}{I_\phi(\rho)}} \sqrt{\frac{I_v(\rho)}{I_v(2^{q-1}\bar{\rho})}} (2^{-\tau})^{q-2} + \sqrt{\frac{I_\phi(2^{q-2}\bar{\rho})}{I_\phi(\rho)}} \sqrt{\frac{I_v(\rho)}{I_v(2^{q-2}\bar{\rho})}} (2^{-\tau})^{q-3} + \cdots \\
        &\qquad\qquad + \sqrt{\frac{I_\phi(2^2\bar{\rho})}{I_\phi(\rho)}} \sqrt{\frac{I_v(\rho)}{I_v(2^2\bar{\rho})}} 2^{-\tau} + \sqrt{\frac{I_\phi(2\bar{\rho})}{I_\phi(\rho)}} \sqrt{\frac{I_v(\rho)}{I_v(2\bar{\rho})}} \Bigg\} \quad \text{for all } v \in \B_{\lambda_\ell} \text{ and } \rho > \bar{\rho}. \label{eq:iterated-pres}
    \end{align}
    Importantly, the constant $C>0$ does not depend on $\phi$ nor $q$.

    $(\star)$ Let $v \in \mathring{\B}_{\lambda_1} = \{1\}$. Equation \eqref{eq:Imono1} gives $I_\phi(\rho_2) \geq C^{-1} I_\phi(\rho_1)$ for all $\rho_2 > \rho_1 \geq 2\bar{\rho}$, whereas \eqref{eq:Imono3} gives $I_v(\rho_2) \leq CI_v(\rho_1)$ for all $\rho_2 > \rho_1 \geq 2\bar{\rho}$.
    By \eqref{eq:iterated-pres}, it follows that for all $\rho > \bar{\rho}$, writing $\rho \in (2^{q-1}\bar{\rho}, 2^q\bar{\rho}]$,
    \begin{align}
        \frac{\left| \inner{\phi}{v}_{\rho} \right|}{\norm{\phi}_{\rho} \norm{v}_{\rho}} &\leq C\bar{\rho}^{-\tau}\left\{ (2^{-\tau})^{q-1} + (2^{-\tau})^{q-2} + \cdots + 2^{-\tau} + 1 \right\} \leq C\bar{\rho}^{-\tau},
    \end{align}
    where we have bounded the geometric series to make $C$ again independent of $\phi$ and $q$. Thus
    \begin{align}
        \frac{\left| \inner{\phi}{v}_{\rho} \right|}{\norm{\phi}_{\rho} \norm{v}_{\rho}} &\leq C\bar{\rho}^{-\tau} \quad \text{for all } \rho > \bar{\rho} \text{ and } v \in \B_{\lambda_1}.
    \end{align}
    By Proposition \ref{prop:small-proj-bounds-2}, if $\bar{\rho}$ is large enough, then there exists $C>0$ such that
    \begin{align}
        \frac{I_\phi(\rho_2)}{I_\phi(\rho_1)} \geq C^{-1} \left( \frac{\rho_2}{\rho_1} \right)^{2\lambda_2} \quad \text{for all } \rho_2 > \rho_1 \geq 2\bar{\rho}.
    \end{align}
    Using this in \eqref{eq:iterated-pres}, we get for each $\rho \in (2^{q-1}\bar{\rho}, 2^q\bar{\rho}]$ (assuming $\tau \leq \lambda_2$ already)
    \begin{align}
        \frac{\left| \inner{\phi}{v}_{\rho} \right|}{\norm{\phi}_{\rho} \norm{v}_{\rho}} &\leq C\bar{\rho}^{-\tau} \Bigg\{ (2^{-\tau})^{q-1} + (2^{-\tau})^{q-2} + 2^{-\lambda_2} (2^{-\tau})^{q-3} + \cdots + \left(2^{-\lambda_2}\right)^{q-3} (2^{-\tau}) + \left(2^{-\lambda_2}\right)^{q-2} \Bigg\} \\
        &\leq C\bar{\rho}^{-\tau} q(2^{-\tau})^{q-2} = C2^{2\tau} q(2^q\bar{\rho})^{-\tau} \leq C(2^q\bar{\rho})^{-\tau/2} \leq C\rho^{-\tau/2},
    \end{align}
    where the second last inequality holds because for all large $\bar{\rho}$, one has $(2^q\bar{\rho})^{\tau} \geq q$ for all $q \in \N$. Thus, replacing $\tau/2$ by $\tau$ on the right, we have shown that
    \begin{align}
        \frac{\left| \inner{\phi}{v}_{\rho} \right|}{\norm{\phi}_{\rho} \norm{v}_{\rho}} &\leq C\rho^{-\tau} \quad \text{for all } \rho > \bar{\rho} \text{ and } v \in \B_{\lambda_1}.
    \end{align}
    Applying Proposition \ref{prop:small-proj-bounds}, we get $C>0$ and $R_1 \geq 1$ such that if $\bar{\rho} \geq R_1$, then
    \begin{align} \label{eq:05508}
        \frac{I_\phi(\rho_2)}{I_\phi(\rho_1)} \geq C^{-1} \left( \frac{\rho_2}{\rho_1} \right)^{2\lambda_2} \quad \text{for all } \rho_2 > \rho_1 \geq 2\bar{\rho}.
    \end{align}

    Using \eqref{eq:05508}, we can now repeat the above, starting from the paragraph $(\star)$, but taking $v \in \mathring{\B}_{\lambda_2}$ instead.
    The end result of this is that by enlarging $\bar{\rho}$ and shrinking $\tau$ sufficiently, there exists $C>0$ such that
    \begin{align}
        \frac{\left| \inner{\phi}{v}_\rho \right|}{\norm{\phi}_\rho \norm{v}_\rho} \leq C\rho^{-\tau} \quad \text{for all } \rho > \bar{\rho} \text{ and } v \in \B_{\lambda_2},
    \end{align}
    and
    \begin{align}
        \frac{I_\phi(\rho_2)}{I_\phi(\rho_1)} \geq C^{-1} \left( \frac{\rho_2}{\rho_1} \right)^{2\lambda_3} \quad \text{for all } \rho_2 > \rho_1 \geq 2\bar{\rho}.
    \end{align}
    Repeating the process up to and including $v \in \mathring{\B}_{\lambda_\ell}$, the proposition follows.
\end{proof}

\subsection{Linear independence from $\B_{\lambda_\ell}$ bounds frequency-related quantities} \label{subsec:lin-indep-freq-bds}

For the rest of \S\ref{sec:asymp-ctrl}, we study functions $u \in \H_{\lambda_{\ell+1}}^+$ outside the span of $\B_{\lambda_\ell}$. By the maximum principle and unique continuation \cites{garofalolin1,garofalolin2}, the restriction of $u$ to any level set $\{r=\rho\}$ is also outside the span of $\B_{\lambda_\ell}$ on that level set.

Proposition \ref{prop:bootstrap-rayleigh} gives lower bounds for $I_\phi$ when $\phi \in \H$ is orthogonal to $\B_{\lambda_\ell}$ on a far level set $\{r=\bar{\rho}\}$. In this subsection, we obtain a similar lower bound for $I_u$, alongside bounds for other quantities introduced in \S\ref{subsec:freq-and-related}. These bounds lead to the observation that $U_u$ and $Q_u$ become close at infinity.

\begin{proposition} \label{prop:u-lower-bounds}
    There exists $\tau > 0$ such that for each $u \in \H_{\lambda_{\ell+1}}^+$ outside the span of $\B_{\lambda_\ell}$, there exists $C>0$ such that for all $\rho > 0$,
    \begin{enumerate}[label=(\alph*)]
        \item $Q_u(\rho) \geq \lambda_{\ell+1} - C\rho^{-\tau}$.
        \item $U_u(\rho) \geq \lambda_{\ell+1} - C\rho^{-\tau}$.
        \item $I_u(\rho) \geq C^{-1}\rho^{2\lambda_{\ell+1}}$.
        \item $D_u(\rho) \geq C^{-1}\rho^{2\lambda_{\ell+1}}$.
    \end{enumerate}
    Moreover,
    \begin{align} \label{eq:Prop5.6s}
        \frac{\left| \inner{u}{v}_\rho \right|}{\norm{u}_\rho \norm{v}_\rho} \leq C\rho^{-\tau} \quad \text{for all } \rho > 0 \text{ and } v \in \B_{\lambda_\ell}.
    \end{align}
\end{proposition}
\begin{proof}
    For each $\rho > 0$, let $P_\rho u$ be the $\inner{\cdot}{\cdot}_\rho$-orthogonal projection of $u$ onto the span of $\B_{\lambda_\ell}$. Then $P_\rho u$ is a function on $\{r=\rho\}$ which can be expressed as a linear combination of functions in $\B_{\lambda_\ell}$. This expression allows $P_\rho u$ to be interpreted as a globally defined function, which we shall continue to do. Then we let
    \begin{align} \label{eq:u-decomp-rho}
        w_\rho := u - P_\rho u,
    \end{align}
    so that $P_\rho u \in \operatorname{span}(\B_{\lambda_\ell}) \subset \H$, $w_\rho \in \H$, and $w_\rho$ is $\inner{\cdot}{\cdot}_\rho$-orthogonal to $\operatorname{span}(\B_{\lambda_\ell})$ on $\{r=\rho\}$.
    
    Let $\bar{\rho}$ be given by Proposition \ref{prop:bootstrap-rayleigh}. Note that
    \begin{align}
        u = \left( P_\rho w_{\bar{\rho}} + P_{\bar{\rho}} u \right) + \left( w_{\bar{\rho}} - P_\rho w_{\bar{\rho}} \right)
    \end{align}
    restricts to an orthogonal decomposition on $\{r=\rho\}$: the first bracketed term is in $\operatorname{span}(\B_{\lambda_\ell})$ while the second bracketed term is in $(\operatorname{span}(\B_{\lambda_\ell}))^\perp$. Of course, another orthogonal decomposition on $\{r=\rho\}$ is
    \begin{align} \label{eq:u-decomp-2}
        u = P_\rho u + w_\rho.
    \end{align}
    By the uniqueness of orthogonal decompositions, it follows that on $\{r=\rho\}$,
    \begin{align}
        P_\rho u &= P_\rho w_{\bar{\rho}} + P_{\bar{\rho}} u, \quad w_\rho = w_{\bar{\rho}} - P_\rho w_{\bar{\rho}}.
    \end{align}
    Then
    \begin{align} \label{eq:0002}
        \frac{\norm{P_\rho u}_\rho}{\norm{w_\rho}_\rho} &\leq \frac{\norm{P_\rho w_{\bar{\rho}}}_\rho}{\norm{w_{\bar{\rho}} - P_\rho w_{\bar{\rho}}}_\rho}
        + \frac{\norm{P_{\bar{\rho}} u}_\rho}{\norm{w_{\bar{\rho}} - P_\rho w_{\bar{\rho}}}_\rho}
    \end{align}
    Since $w_{\bar{\rho}}$ is orthogonal to $\operatorname{span}(\B_{\lambda_\ell})$ on $\{r=\bar{\rho}\}$, Proposition \ref{prop:bootstrap-rayleigh} gives
    \begin{align} \label{eq:0001}
        \frac{\left| \inner{w_{\bar{\rho}}}{v}_\rho \right|}{\norm{w_{\bar{\rho}}}_\rho \norm{v}_\rho} \leq C\rho^{-\tau} \quad \text{for all } \rho \geq \bar{\rho} \text{ and } v \in \B_{\lambda_\ell},
    \end{align}
    where $C, \tau > 0$ are independent of $u$, and
    \begin{align} \label{eq:Iwrhobar-bd}
        I_{w_{\bar{\rho}}}(\rho) \geq C^{-1} \left( \frac{\rho}{2\bar{\rho}} \right)^{2\lambda_{\ell+1}} I_{w_{\bar{\rho}}}(2\bar{\rho}) \geq C^{-1}\rho^{2\lambda_{\ell+1}} \quad \text{for all } \rho \geq 2\bar{\rho}
    \end{align}
    where $C$ on the right depends on $u$. By similar reasoning to \eqref{eq:diff-projections}, the estimate \eqref{eq:0001} implies
    \begin{align} \label{eq:00010}
        \frac{\norm{P_\rho w_{\bar{\rho}}}_\rho}{\norm{w_{\bar{\rho}}}_\rho} \leq C\rho^{-\tau} \quad \text{for all } \rho \geq \bar{\rho}.
    \end{align}

    By \eqref{eq:Iwrhobar-bd} and \eqref{eq:00010}, it follows that for all $\rho \geq 2\bar{\rho}$,
    \begin{align} \label{eq:3304}
        \norm{w_{\bar{\rho}} - P_\rho w_{\bar{\rho}}}_\rho \geq (1-C\rho^{-\tau}) \norm{w_{\bar{\rho}}}_\rho \geq (1-C\rho^{-\tau}) C^{-1} \rho^{\lambda_{\ell+1}} \geq C^{-1} \rho^{\lambda_{\ell+1}}.
    \end{align}
    From the definition of $(E_\ell)$, we have $\norm{v}_\rho \leq C\rho^{\lambda_\ell}$ for each $v \in \B_{\lambda_\ell}$. Since $P_{\bar{\rho}}u$ is a fixed linear combination of such $v$'s, it follows that $\norm{P_{\bar{\rho}} u}_\rho \leq C\rho^{\lambda_\ell}$. Combining with \eqref{eq:3304}, we get
    \begin{align}
        \frac{\norm{P_{\bar{\rho}} u}_\rho}{\norm{w_{\bar{\rho}} - P_\rho w_{\bar{\rho}}}_\rho} &\leq C\rho^{-(\lambda_{\ell+1} - \lambda_\ell)} \quad \text{for all } \rho \geq 2\bar{\rho}.
    \end{align}
    Putting this and \eqref{eq:00010} back into \eqref{eq:0002}, we get (after decreasing $\tau$ so that $\tau < \lambda_{\ell+1}-\lambda_\ell$)
    \begin{align}
        \frac{\norm{P_\rho u}_\rho}{\norm{w_\rho}_\rho} &\leq C\rho^{-\tau} \quad \text{for all } \rho \geq 2\bar{\rho}
    \end{align}
    In view of \eqref{eq:u-decomp-2} being an orthogonal decomposition on $\{r=\rho\}$, it follows that
    \begin{align}
        \frac{\norm{P_\rho u}_\rho}{\norm{u}_\rho} &\leq C\rho^{-\tau} \quad \text{for all } \rho \geq 2\bar{\rho}.
    \end{align}
    Then for each $v \in \B_{\lambda_\ell}$,
    \begin{align}
        \frac{\left| \inner{u}{v}_\rho \right|}{\norm{u}_\rho \norm{v}_\rho} &= \frac{\norm{\frac{\inner{u}{v}_\rho}{\inner{v}{v}_\rho} v}_\rho}{\norm{u}_\rho} \leq \frac{\norm{P_\rho u}_\rho}{\norm{u}_\rho} \leq C\rho^{-\tau} \quad \text{for all } \rho \geq 2\bar{\rho}.
    \end{align}
    This implies \eqref{eq:Prop5.6s} up to increasing $C$ depending on the values of $u$ and each $v \in \B_{\lambda_\ell}$ on $B_{2\bar{\rho}}$. Using Proposition \ref{prop:small-proj-bounds} and the fact that $D_u = U_u I_u$, the rest of the assertions follow.
\end{proof}

\begin{lemma} \label{lem:pointwise-prelim-bounds}
    For each $u \in \H_{\lambda_{\ell+1}}^+$ outside the span of $\B_{\lambda_\ell}$, and for each $\epsilon > 0$, there exists $C_\epsilon > 0$ such that for all $\rho > 0$,
    \begin{enumerate}[label=(\alph*)]
        \item $U_u(\rho) \leq C_\epsilon\rho^\epsilon$.
        \item $Q_u(\rho) \leq C_\epsilon\rho^\epsilon$.
        \item $G_u(\rho) \leq C_\epsilon\rho^{-1+\epsilon}$.
    \end{enumerate}
\end{lemma}
\begin{proof}
    For each $\epsilon > 0$, there exists $C_\epsilon > 0$ such that $|u| \leq C_\epsilon r^{\lambda_{\ell+1}+\epsilon}$. By Corollary \ref{cor:deriv-growthbounds}, this gives
    \begin{align}
        \left|\inner{\nabla u}{\nabla r}\right| &\leq C_\epsilon r^{\lambda_{\ell+1}+\epsilon-1}, \label{eq:schauder1} \\
        |\nabla^\top u| &\leq C_\epsilon r^{\lambda_{\ell+1}+\epsilon-\frac{1}{2}}.
    \end{align}
    The lemma follows from combining these with the $I_u$ lower bound from Proposition \ref{prop:u-lower-bounds}.
\end{proof}

\begin{lemma}[$U_u$, $Q_u$ closeness] \label{lem:UQcloseness-new}
    For each nonzero $u \in \H_{\lambda_{\ell+1}}^+(M)$, there exists $C > 0$ such that $|U_u(\rho) - Q_u(\rho)| \leq C \rho^{-\frac{1}{3}}$ for all $\rho > 0$.
\end{lemma}
\begin{proof}
    Using the first variation formula, Lemma \ref{lem:level-set-H}, as well as \eqref{eq:hess-r} and \eqref{eq:grad|gradr|}, we have
    \begin{align}
        Q'(\rho) &= \frac{Q}{\rho} + \frac{\rho \int_{\{r=\rho\}} \left( \inner{\nabla(|\nabla^\top u|^2 |\nabla r|^{-1})}{\nu} + |\nabla^\top u|^2 |\nabla r|^{-1} H_{\Sigma_\rho} \right) \frac{1}{|\nabla r|}}{\int_{\{r=\rho\}} u^2|\nabla r|} \\
        &\quad - \frac{\rho \int_{\{r=\rho\}} |\nabla^\top u|^2 |\nabla r|^{-1}}{\left( \int_{\{r=\rho\}} u^2|\nabla r| \right)^2} \cdot \int_{\{r=\rho\}} \left( \inner{\nabla(u^2|\nabla r|)}{\nu} + u^2|\nabla r| H_{\Sigma_\rho} \right) \frac{1}{|\nabla r|} \\
        &= \frac{Q}{\rho} + \frac{\rho \int_{\{r=\rho\}} \left[ 
        \inner{\nabla|\nabla^\top u|^2}{\nu}|\nabla r|^{-2}
        + \red{|\nabla^\top u|^2 \inner{\nabla|\nabla r|^{-1}}{\frac{\nabla r}{|\nabla r|^2}}}
        + |\nabla^\top u|^2 |\nabla r|^{-2} \left( \blue{\frac{n-1}{2\rho}} + \red{\O(\rho^{-\mu-1})} \right) \right]}{\int_{\{r=\rho\}} u^2|\nabla r|} \\
        &\quad - \frac{\rho \int_{\{r=\rho\}} |\nabla^\top u|^2 |\nabla r|^{-1}}{\left( \int_{\{r=\rho\}} u^2|\nabla r| \right)^2} \cdot \int_{\{r=\rho\}} \left[ \magenta{2u\inner{\nabla u}{\nu}} + \red{u^2 \inner{\nabla|\nabla r|}{\frac{\nabla r}{|\nabla r|^2}}} + u^2\left( \blue{\frac{n-1}{2\rho}} + \red{\O(\rho^{-\mu-1})} \right) \right] \\
        &= \frac{Q}{\rho} 
        - \magenta{\frac{2QU}{\rho}}
        + \blue{\frac{n-1}{2\rho}(1+\O(\rho^{-\mu})) Q} 
        - \blue{\frac{n-1}{2\rho}(1+\O(\rho^{-\mu})) Q} 
        + \red{\O(\rho^{-\mu-1})Q}
        + \frac{\rho \int_{\{r=\rho\}} \inner{\nabla|\nabla^\top u|^2}{\frac{\nabla r}{|\nabla r|^3}}}{\int_{\{r=\rho\}} u^2|\nabla r|} \\
        &= \frac{Q}{\rho} - \frac{2QU}{\rho} + \O(\rho^{-\mu-1})Q + \frac{\rho \int_{\{r=\rho\}} \inner{\nabla|\nabla^\top u|^2}{\frac{\nabla r}{|\nabla r|^3}}}{\int_{\{r=\rho\}} u^2|\nabla r|}. \label{eq:Q'}
    \end{align}
    Now with $W = \frac{\nabla r}{|\nabla r|^3}$, we compute at any point on $\{r=\rho\}$
    \begin{align}
        \inner{\nabla|\nabla^\top u|^2}{W} &= 2 \inner{\nabla_W \nabla^\top u}{\nabla^\top u} = 2\inner{\nabla_W \nabla u}{\nabla^\top u} - 2\inner{\nabla_W (\inner{\nabla u}{\nu}\nu)}{\nabla^\top u} \\
        &= 2\nabla^2 u(W,\nabla^\top u) - 2\inner{\nabla u}{\nu} \inner{\nabla_W \nu}{\nabla^\top u}
    \end{align}
    and (using \eqref{eq:hess-r})
    \begin{align}
        \inner{\nabla_W \nu}{\nabla^\top u} &= \inner{\frac{\nabla_W \nabla r}{|\nabla r|}}{\nabla^\top u} = \nabla^2 r\left( \frac{\nabla r}{|\nabla r|^4}, \nabla^\top u \right) = \frac{1}{2r} \eta\left( \frac{\nabla r}{|\nabla r|^4}, \nabla^\top u \right).
    \end{align}
    We combine these two computations and Corollary \ref{cor:deriv-growthbounds} to get
    \begin{align}
        \left| \inner{\nabla|\nabla^\top u|^2}{W} \right| &\leq 2|\nabla^2 u| |W| |\nabla^\top u| + \frac{1}{r}\left| \inner{\nabla u}{\nu} \right| |\nabla^\top u| |\eta|\left| \frac{\nabla r}{|\nabla r|^4} \right| \leq C_\epsilon r^{2\lambda_{\ell+1}+2\epsilon-\frac{3}{2}}. \label{eq:03}
    \end{align}
    Going back to \eqref{eq:Q'} then using \eqref{eq:03}, Proposition \ref{prop:u-lower-bounds} and Lemma \ref{lem:pointwise-prelim-bounds}, we have
    \begin{align}
        |Q'(\rho)| &\leq
        C_\epsilon \rho^{-\frac{1}{2}+2\epsilon} \leq C' \rho^{-\frac{1}{3}}, \label{eq:Q'decay}
    \end{align}
    where we have chosen $\epsilon = \frac{1}{12}$ and written $C' := C_{1/12}$ in the last step. To conclude, we argue similarly to Lemma \ref{lem:important-ode-lemma}. By Lemma \ref{lem:D'I'}, \eqref{eq:simple-CS}, and Lemma \ref{lem:pointwise-prelim-bounds}, it holds pointwise that
    \begin{align} \label{eq:UlowerQ}
        U'(\rho) \geq \left(-1-\frac{C}{\rho}\right) U(\rho) - \frac{U(\rho)^2}{\rho} + Q(\rho)
    \end{align}
    and, for some $\hat{C} > 0$,
    \begin{align} \label{eq:UupperQ}
        U'(\rho) \leq \left( -1+ \frac{C}{\rho} \right) U(\rho) + \hat{C}\rho^{-5/6} + Q(\rho).
    \end{align}
    From here on, fix $\zeta \in (0,1)$ and suppose $\rho$ is such that $U(\rho) < Q(\rho) - \frac{\zeta}{2}$. Then by \eqref{eq:UlowerQ} and Lemma \ref{lem:pointwise-prelim-bounds},
    \begin{align} \label{eq:U'zeta}
        U'(\rho) &\geq \left(-1-\frac{C}{\rho}\right) \left(Q(\rho)-\frac{\zeta}{2}\right) - \frac{U(\rho)^2}{\rho} + Q(\rho) \geq \frac{\zeta}{2} - C'' \rho^{-5/6}
    \end{align}
    where $C'' > 0$. We may assume $C'' > C'$ from \eqref{eq:Q'decay}.
    Then define $$\rho_* := \left(\frac{8C''}{\zeta} \right)^3.$$ It follows that if
    \begin{align} \label{eq:rho-cond}
        \rho \geq \rho_* \quad \text{and} \quad U(\rho) \leq Q(\rho) - \zeta,
    \end{align}
    then $Q'(\rho) \leq \frac{\zeta}{8}$ (by \eqref{eq:Q'decay}). Also, by \eqref{eq:U'zeta}, we have $U'(\rho) \geq \frac{3\zeta}{8}$.
    We split into two cases:
    \begin{itemize}
        \item If $U(\rho_*)-Q(\rho_*) \geq -\zeta$, then we cannot have $U(\rho)-Q(\rho) < -\zeta$ at any $\rho > \rho_*$ since otherwise $(U-Q)' \geq \frac{3\zeta}{8} - \frac{\zeta}{8}  = \frac{\zeta}{4} > 0$ at the first point where this happens, a contradiction.
        \item If $U(\rho_*) - Q(\rho_*) < -\zeta$, then $U(\rho_*)-Q(\rho_*) \geq -C \sqrt{\rho_*}$ (by Lemma \ref{lem:pointwise-prelim-bounds}), and $(U-Q)' \geq \frac{3\zeta}{8} - \frac{\zeta}{8} = \frac{\zeta}{4}$ for as long as $U-Q < -\zeta$. Hence, it takes at most a distance of $\frac{C\sqrt{\rho_*}}{\zeta/4} \leq C \left( \frac{1}{\zeta} \right)^{5/2} \leq \frac{C}{\zeta^3}$ from $\rho_*$ to reach a point where $U - Q = -\zeta$, and from then onwards we can never have $U - Q < -\zeta$ since otherwise $(U-Q)' \geq \frac{3\zeta}{8} - \frac{\zeta}{8} = \frac{\zeta}{4} > 0$ at the first point where this happens, a contradiction.
    \end{itemize}
    Combining these cases, we see that
    \begin{align}
        U(\rho) \geq Q(\rho)-\zeta \quad \text{whenever} \quad \rho > \rho_* + \frac{C}{\zeta^3}.
    \end{align}
    The number on the right is exactly of the form $(C/\zeta)^3$. This implies that $U(\rho) \geq Q(\rho) - (C+1)\rho^{-\frac{1}{3}}$, proving one half of the lemma. The other half is proved similarly, using \eqref{eq:UupperQ} in place of \eqref{eq:UlowerQ}.
\end{proof}

\subsection{Almost separation of variables and asymptotic control} \label{subsec:almost-sep-asymp-ctrl}

In this subsection, we prove Theorem \ref{thm:asymp-ctrl}. The key to this is the next result which shows that $U_u$ is almost monotone. This is modelled on a related result in \cite{cm97a}*{Proposition 4.11}.

\begin{lemma} \label{lem:almost-mono-U-new}
    For each $u \in \H_{\lambda_{\ell+1}}^+$ outside the span of $\B_{\lambda_\ell}$, there exists $C > 0$ such that for all $\rho_0 \geq 1$,
    \begin{align} \label{eq:U-almost-mono-0}
        \int_{\rho_0}^\infty \left| -\frac{U_u'}{U_u} + \frac{2G_u}{U_u} - \frac{2U_u}{\rho} \right| \, d\rho \leq C\rho_0^{-\sigma}
    \end{align}
    where $\sigma = \min\{\frac{\mu}{2}, \frac{1}{4}, 2\lambda_{\ell+1}\} > 0$.
    In particular, $U_u$ is almost monotone in the sense that
    \begin{align} \label{eq:U-almost-mono-1}
        \int_{\rho_0}^\infty \min\{ (\log U_u)'(\rho),0 \} \, d\rho \geq -C\rho_0^{-\sigma}.
    \end{align}
\end{lemma}
\begin{proof}
    By \eqref{eq:simple-CS}, we have $\frac{U'}{U} \geq \frac{U'}{U} - \frac{2G}{U} + \frac{2U}{\rho} \geq -\left|-\frac{U'}{U} + \frac{2G}{U} - \frac{2U}{\rho}\right|$, so
    \begin{align}
        \min\{ (\log U)'(\rho), 0 \} \geq -\left| -\frac{U'}{U} + \frac{2G}{U} - \frac{2U}{\rho} \right|.
    \end{align}
    Therefore, \eqref{eq:U-almost-mono-0} implies \eqref{eq:U-almost-mono-1}. It remains to prove \eqref{eq:U-almost-mono-0}.

    Using Lemma \ref{lem:D'I'} and Corollary \ref{cor:D'/D}, we have
    \begin{align}
        -\frac{U'}{U} + \frac{2G}{U} - \frac{2U}{\rho} &= -\frac{D'}{D} + \frac{I'}{I} + \frac{2G}{U} - \frac{2U}{\rho} \\
        &= -f'(\rho) + \frac{\int_{\{0<r<\rho\}} rf'(r) |\nabla u|^2 e^{-f}}{\rho^{\frac{n-1}{2}} e^{-f(\rho)} D(\rho)} - \frac{\int_{\{0<r<\rho\}} (1 + \O(r^{-\mu})) |\nabla u|^2 e^{-f}}{\rho^{\frac{n-1}{2}} e^{-f(\rho)} D(\rho)} \\
        &\quad + \frac{\int_{\{0<r<\rho\}} (1+\O(r^{-\mu}))\inner{\nabla u}{\nabla r}^2 e^{-f}}{\rho^{\frac{n-1}{2}} e^{-f(\rho)} D(\rho)} + \frac{n-3}{2} \frac{\int_{B_{0}} |\nabla u|^2 e^{-f}}{\rho^{\frac{n-1}{2}} D(\rho)} + \O(\rho^{-\mu-1}). \label{eq:9030}
    \end{align}
    By Proposition \ref{prop:u-lower-bounds}, there exists $C>0$ such that
    \begin{align} \label{eq:denom-bd}
        \rho^{\frac{n-1}{2}} D(\rho) \geq C^{-1} \rho^{2\lambda_{\ell+1} + \frac{n-1}{2}} \quad \text{for all } \rho \geq 1.
    \end{align}
    Also, for each $a \in \R$, L'H\^opital's rule and Assumption \ref{assump:f} give $\lim_{\rho\to\infty} \left( \rho^{-a} e^{f(\rho)} \int_1^\rho s^a e^{-f(s)} \, ds \right) = 1$, so there exists $C(a)$ such that for all $\rho \geq 1$,
    \begin{align} \label{eq:lhopital}
        \int_1^\rho s^a e^{-f(s)} \, ds \leq C(a) \rho^a e^{-f(\rho)}.
    \end{align}
    Using the coarea formula and Corollary \ref{cor:deriv-growthbounds} with \eqref{eq:denom-bd} and \eqref{eq:lhopital}, we get for each $\rho \geq 1$,
    \begin{equation}
        \left| \frac{\int_{\{0<r<\rho\}} (1+\O(r^{-\mu}))\inner{\nabla u}{\nabla r}^2 e^{-f}}{\rho^{\frac{n-1}{2}} e^{-f(\rho)} D(\rho)} \right| \leq \frac{C \int_0^\rho \left( \int_{\{r=s\}} \inner{\nabla u}{\nu}^2 \right) e^{-f(s)} \, ds}{\rho^{2\lambda_{\ell+1}+\frac{n-1}{2}}e^{-f(\rho)}} \leq C_\epsilon \rho^{-2+2\epsilon} \leq C\rho^{-\frac{3}{2}}, \label{eq:0014}
    \end{equation}
    where we have selected $\epsilon = \frac{1}{4}$ in the last inequality. Similar manipulations bound
    \begin{align}
        \left|\frac{\int_{\{0 < r < \rho\}} \O(r^{-\mu})|\nabla u|^2 e^{-f}}{\rho^{\frac{n-1}{2}} e^{-f(\rho)} D(\rho)}\right| \leq C_\epsilon \rho^{-\mu-1+\epsilon} \leq C\rho^{-\frac{\mu}{2}-1}. \label{eq:0015}
    \end{align}
    Using the estimates \eqref{eq:denom-bd}, \eqref{eq:0014} and \eqref{eq:0015} back in \eqref{eq:9030}, then integrating over $[\rho_0,\infty)$, we get
    \begin{align}
        \int_{\rho_0}^{\infty} \left| -\frac{U'}{U} + \frac{2G}{U} - \frac{2U}{\rho} \right| \, d\rho &\leq \int_{\rho_0}^\infty \left| -f'(\rho) + \frac{\int_{\{0<r<\rho\}} rf'(r) |\nabla u|^2 e^{-f}}{\rho^{\frac{n-1}{2}} e^{-f(\rho)} D(\rho)} - \frac{\int_{\{0<r<\rho\}} |\nabla u|^2 e^{-f}}{\rho^{\frac{n-1}{2}} e^{-f(\rho)} D(\rho)} \right| \, d\rho \\
        &\quad + C \int_{\rho_0}^\infty (\rho^{-\frac{\mu}{2}-1} + \rho^{-\frac{3}{2}} + \rho^{-2\lambda_{\ell+1}-1} + \rho^{-\mu-1}) \, d\rho. \label{eq:9031}
    \end{align}
    The last line is bounded by $C\rho_0^{-\sigma}$. Thus to prove \eqref{eq:U-almost-mono-0}, it remains to establish that
    \begin{align}
        \left| -f'(\rho) + \frac{\int_{\{0<r<\rho\}} rf'(r) |\nabla u|^2 e^{-f}}{\rho^{\frac{n-1}{2}} e^{-f(\rho)} D(\rho)} - \frac{\int_{\{0<r<\rho\}} |\nabla u|^2 e^{-f}}{\rho^{\frac{n-1}{2}} e^{-f(\rho)} D(\rho)} \right| \leq C\rho^{-\frac{5}{4}}. \label{eq:target-big-bound}
    \end{align}
    Since $\rho^{\frac{n-1}{2}} e^{-f(\rho)} D(\rho) = \rho \int_{B_\rho} |\nabla u|^2 e^{-f}$, we have
    \begin{align}
        &\rho \left| -f'(\rho) + \frac{\int_{\{0<r<\rho\}} rf'(r) |\nabla u|^2 e^{-f}}{\rho^{\frac{n-1}{2}} e^{-f(\rho)} D(\rho)} - \frac{\int_{\{0<r<\rho\}} |\nabla u|^2 e^{-f}}{\rho^{\frac{n-1}{2}} e^{-f(\rho)} D(\rho)} \right| \\
        &\quad = \rho \left| \frac{\int_{\{0<r<\rho\}} (rf'(r)-\rho f'(\rho)) |\nabla u|^2 e^{-f}}{\rho^{\frac{n-1}{2}} e^{-f(\rho)} D(\rho)} - \frac{\rho f'(\rho) \int_{B_0} |\nabla u|^2 e^{-f}}{\rho^{\frac{n-1}{2}} e^{-f(\rho)} D(\rho)} - \frac{1}{\rho} + \frac{\int_{B_0} |\nabla u|^2 e^{-f}}{\rho^{\frac{n-1}{2}} e^{-f(\rho)} D(\rho)} \right| \\
        &\quad \leq \Bigg| \underbrace{ \frac{\int_0^\rho (sf'(s)-\rho f'(\rho)) \int_{\{r=s\}} |\nabla u|^2 |\nabla r|^{-1} e^{-f(s)} \, ds}{\rho^{\frac{n-3}{2}} e^{-f(\rho)} D(\rho)} - 1 }_{=: Z_1(\rho)} \Bigg| + C\rho^{-\frac{1}{4}}, \label{eq:010102}
    \end{align}
    where the last estimate comes from the coarea formula, \eqref{eq:denom-bd}, and the exponential growth of $e^{-f(\rho)}$. Integrating by parts, we compute
    \begin{align}
        |Z_1(\rho)| &= \left| \frac{\left[ (sf'(s)-\rho f'(\rho)) \int_{B_s} |\nabla u|^2 e^{-f} \right]_{s=0}^{s=\rho} - \int_0^\rho (f'(s)+sf''(s))\left(\int_{B_s} |\nabla u|^2 e^{-f}\right) \, ds}{\rho^{\frac{n-3}{2}} e^{-f(\rho)} D(\rho)} - 1 \right| \\
        &\leq \Bigg| -\frac{\int_0^\rho sf''(s)v(s)e^{-f(s)} \, ds}{v(\rho)e^{-f(\rho)}} - \frac{\int_0^\rho f'(s)v(s)e^{-f(s)} \, ds}{v(\rho)e^{-f(\rho)}} - 1 \Bigg|  + C\rho^{-\frac{1}{4}}, \label{eq:010103}
    \end{align}
    where $v(s) := s^{\frac{n-3}{2}} D(s) \geq 0$. By Corollary \ref{cor:I-almost-mono}, Proposition \ref{prop:u-lower-bounds} and Lemma \ref{lem:pointwise-prelim-bounds}, we have for all $\rho \geq s \geq 1$,
    \begin{align} \label{eq:v-bound}
        \left| \frac{v(s)}{v(\rho)} \right| &= \left( \frac{s}{\rho} \right)^{\frac{n-3}{2}} \frac{U(s)}{U(\rho)} \frac{I(s)}{I(\rho)} \leq C_\epsilon \left( \frac{s}{\rho} \right)^{\frac{n-3}{2}} s^\epsilon
    \end{align}
    for any $\epsilon > 0$.
    Also, Lemma \ref{lem:D'I'} and the definitions of $G,Q$ yield
    \begin{align}
        \frac{D'(s)}{D(s)} = \frac{3-n}{2s} + f'(s) + \frac{G(s)+Q(s)}{U(s)}.
    \end{align}
    Hence, by Proposition \ref{prop:u-lower-bounds}, Lemma \ref{lem:pointwise-prelim-bounds} and Lemma \ref{lem:UQcloseness-new},
    \begin{align} \label{eq:logv'-bound}
        \frac{v'(s)}{v(s)} &= \frac{n-3}{2s} + \frac{D'(s)}{D(s)} = f'(s) + \frac{G(s)+Q(s)}{U(s)} = \O(s^{-\frac{1}{3}}).
    \end{align}
    Using \eqref{eq:v-bound}, Assumption \ref{assump:f} and \eqref{eq:lhopital}, we estimate
    \begin{align}
        \left| \frac{\int_0^\rho sf''(s)v(s)e^{-f(s)} \, ds}{v(\rho)e^{-f(\rho)}}\right| &\leq \frac{C_\epsilon \int_0^\rho s^{-\frac{1}{2}+\epsilon+\frac{n-3}{2}} e^{-f(s)} \, ds}{\rho^{\frac{n-3}{2}} e^{-f(\rho)}} \leq C_\epsilon \rho^{-\frac{1}{2}+\epsilon}. \label{eq:est1-RHS}
    \end{align}
    Integrating by parts and using that $v(0) = 0$, we also compute
    \begin{align}
        -\frac{\int_0^\rho f'(s)v(s)e^{-f(s)} \, ds}{v(\rho)e^{-f(\rho)}} - 1 &= \frac{\int_0^\rho v(s) \frac{d}{ds}(e^{-f(s)}) \, ds}{v(\rho)e^{-f(\rho)}} - 1 = - \frac{\int_0^\rho v'(s)e^{-f(s)} \, ds}{v(\rho)e^{-f(\rho)}}. \label{eq:Z2}
    \end{align}
    Then using \eqref{eq:v-bound}, \eqref{eq:logv'-bound} and \eqref{eq:lhopital}, we can further estimate
    \begin{align}
        \left| -\frac{\int_0^\rho f'(s)v(s)e^{-f(s)} \, ds}{v(\rho)e^{-f(\rho)}} - 1 \right| &\leq \frac{C_\epsilon \int_0^\rho s^{-\frac{1}{3}+\epsilon+\frac{n-3}{2}} e^{-f(s)} \, ds}{\rho^{\frac{n-3}{2}} e^{-f(\rho)}} \leq C_\epsilon \rho^{-\frac{1}{3}+\epsilon}. \label{eq:est2-RHS}
    \end{align}
    Putting the estimates \eqref{eq:est1-RHS} and \eqref{eq:est2-RHS} back into \eqref{eq:010103}, and selecting $\epsilon = 1/12$, we get $|Z_1(\rho)| \leq C\rho^{-\frac{1}{4}}$. In view of \eqref{eq:010102}, this proves \eqref{eq:target-big-bound}.
\end{proof}

The almost-monotonicity of $U_u$ from Lemma \ref{lem:almost-mono-U-new} implies an upper bound which complements the lower bound from Proposition \ref{prop:u-lower-bounds}:

\begin{corollary} \label{cor:U-upper-bd}
    For each $u \in \H_{\lambda_{\ell+1}}^+$ outside the span of $\B_{\lambda_\ell}$, there exists $C>0$ such that
    \begin{align}
        U_u(\rho) \leq \lambda_{\ell+1} + C\rho^{-\sigma} \quad \text{for all } \rho > 0,
    \end{align}
    where $\sigma = \min\{\frac{\mu}{2}, \frac{1}{4}, 2\lambda_{\ell+1}\} > 0$.
\end{corollary}
\begin{proof}
    Let $u$ be as such.
    By Lemma \ref{lem:almost-mono-U-new}, there exists $C>0$ such that for each $s > \rho \geq 1$,
    \begin{align}
        \log\left( \frac{U_u(s)}{U_u(\rho)} \right) &\geq \int_{\rho}^{s} \min\{(\log U_u)'(t), 0\} \, dt \geq
        \int_{\rho}^\infty \min\{(\log U_u)'(t), 0\} \, dt > -C\rho^{-\sigma}
    \end{align}
    and so
    \begin{align} \label{eq:UU}
        \frac{U_u(s)}{U_u(\rho)} > e^{-C\rho^{-\sigma}} \geq 1-C\rho^{-\sigma} \quad \text{for all } s > \rho \geq 1.
    \end{align}
    For a contradiction, suppose there is a sequence $\rho_N \to \infty$ such that $U_u(\rho_N) > \lambda_{\ell+1} + N\rho_N^{-\sigma}$. By \eqref{eq:UU}, we have for each $N$,
    \begin{align}
        U_u(s) &> (\lambda_{\ell+1} + N\rho_N^{-\sigma}) (1-C\rho_N^{-\sigma}) \geq \lambda_{\ell+1} + N\rho_N^{-\sigma} - CN\rho_N^{-2\sigma} \quad \text{for all } s > \rho_N.
    \end{align}
    Choose $N$ large so that $CN \rho_N^{-2\sigma} < (N-1)\rho_N^{-\sigma}$. Then the above becomes
    \begin{align}
        U_u(s) > \lambda_{\ell+1} + \rho_N^{-\sigma} \quad \text{for all } s > \rho_N.
    \end{align}
    However, since $u \in \H_{\lambda_{\ell+1}}^+$, Corollary \ref{cor:liminf-upper} gives $\liminf_{\rho\to\infty} U_u(\rho) \leq \lambda_{\ell+1}$ which is a contradiction.
\end{proof}

\begin{proof}[Proof of Theorem \ref{thm:asymp-ctrl}]
    Let $u \in \H_{\lambda_{\ell+1}}^+$ be linearly independent from $\B_{\lambda_\ell}$. By Proposition \ref{prop:u-lower-bounds}, there exist $C,\tau > 0$ such that for each $v \in \B_{\lambda_\ell}$, the functions $u$ and $v$ are $(C,\tau)$-asymptotically orthogonal for some $C,\tau$. This proves part (b) of the theorem. To prove (a) we must show that $u \in \mathring{\mathcal{S}}_{\lambda_{\ell+1}}(C,\tau)$ for some $C,\tau>0$.
    
    By Proposition \ref{prop:u-lower-bounds}, Corollary \ref{cor:U-upper-bd} and Lemma \ref{lem:UQcloseness-new}, there exist $C,\tau > 0$ such that for all $\rho > 0$,
    \begin{align}
        \lambda_{\ell+1} - C\rho^{-\tau} &\leq U(\rho) \leq \lambda_{\ell+1} + C\rho^{-\tau}, \label{eq:Uu-pinched} \\
        \lambda_{\ell+1} - C\rho^{-\tau} &\leq Q(\rho) \leq \lambda_{\ell+1} + C\rho^{-\tau}, \label{eq:Qu-pinched} \\
        \frac{\left| \inner{u}{v}_\rho \right|}{\norm{u}_\rho \norm{v}_\rho} &\leq C\rho^{-\tau} \quad \text{for all } v \in \B_{\lambda_\ell}. \label{eq:u-orthog-bl}
    \end{align}
    By \eqref{eq:Qu-pinched} and \eqref{eq:u-orthog-bl}, $u$ satisfies the hypotheses of Proposition \ref{prop:B-almost-eigenfunc}, so
    \begin{align}
        \frac{\norm{\P_{\rho,\ell+1}u}_\rho'}{\norm{u}_\rho'} \geq 1-C\rho^{-\tau}. \label{eq:Proju}
    \end{align}
    Now let $s > \rho \geq 1$. Using \eqref{eq:Uu-pinched} and the Taylor series for $\log(1+x)$, we have
    \begin{align}
        \int_{\rho}^s \frac{U'}{U} \, d\rho &= \log\left( \frac{U(s)}{U(\rho)} \right) \leq \log\left( \frac{ \lambda_{\ell+1}+Cs^{-\tau}}{\lambda_{\ell+1}-C\rho^{-\tau}} \right) \leq \log\left( \frac{\lambda_{\ell+1}+C\rho^{-\tau}}{\lambda_{\ell+1}-C\rho^{-\tau}} \right) \leq C\rho^{-\tau}, \label{eq:ing1}
    \end{align}
    where $C$ is independent of $\rho$ and $s$; this will remain as such.
    By Lemma \ref{lem:almost-mono-U-new}, we also have
    \begin{align}
        \int_\rho^s \left| -\frac{U'}{U} + \frac{2G}{U} - \frac{2U}{\rho} \right| \, dt < C\rho^{-\tau}. \label{eq:ing2}
    \end{align}
    From the uniform upper bound for $U$ from Corollary \ref{cor:U-upper-bd}, as well as Corollary \ref{cor:I-almost-mono}, we have
    \begin{align} \label{eq:maxD}
        \max_{[\rho,s]} D \leq \left( \max_{[\rho,s]} U \right) \left( \max_{[\rho,s]} I \right) \leq CI(s).
    \end{align}
    Then by Lemma \ref{lem:almost-sep-formula}, \eqref{eq:ing1}, \eqref{eq:ing2}, and \eqref{eq:maxD},
    \begin{align}
        \int_{\{\rho \leq r \leq s\}} r^{-\frac{n+1}{2}} \left(r\inner{\nabla u}{\nu} - U u |\nabla r|\right)^2 &= \int_{\rho}^s \underbrace{\left( \frac{G}{U} - \frac{U}{t} \right)}_{\geq 0 \text{ by } \eqref{eq:simple-CS}} D \, dt \\
        &\leq \left(\max_{[\rho,s]} D \right) \left[ \frac{1}{2} \int_{\rho}^{s} \frac{U'}{U} \, dt + \frac{1}{2} \int_{\rho}^{s} \left( -\frac{U'}{U} + \frac{2G}{U} - \frac{2U}{t} \right) \, dt \right] \\
        &\leq C\rho^{-\tau} I(s).
    \end{align}
    Together with \eqref{eq:Uu-pinched}, \eqref{eq:Qu-pinched} and \eqref{eq:Proju}, it follows that $u \in \mathring{\mathcal{S}}_{\lambda_{\ell+1}}(C,\tau)$.
\end{proof}

\section{Construction of drift-harmonic functions: proof of Theorem \ref{thm:construction}} \label{sec:construction}

In this section, we will prove Theorem \ref{thm:construction} by constructing drift-harmonic functions on AP manifolds, generalizing the constructions of harmonic functions in \cites{ding,huangdimensions,xu}. The scaling arguments used there must be carefully modified to work for AP manifolds and the drift-harmonic equation $\L_f u = 0$.

\subsection{Outline for this section} \label{subsec:outline-construction}

In \S\ref{subsec:model-3circ}, we classify solutions to a model parabolic equation and use this to prove a three circles theorem. This is used in \S\ref{subsec:3circ} to establish a three circles theorem for drift-harmonic functions on $M$. In \S\ref{subsec:seq-superharmonic}, we exhibit a sequence of nonnegative $\L_f$-superharmonic functions defined on domains exhausting $M$.

In \S\ref{subsec:construction-from-ef}, we solve a sequence of Dirichlet problems on domains exhausting $M$; the tools from earlier will provide uniform bounds for the solutions, enabling us to take limits and find a global drift-harmonic function on $M$. In \S\ref{subsec:proof-Al}, we prove Theorem \ref{thm:construction} by repeatedly performing the construction in \S\ref{subsec:construction-from-ef}, with refinements to make the resulting drift-harmonic functions linearly independent and asymptotically orthogonal.

\subsection{The model parabolic equation} \label{subsec:model-3circ}

In this subsection, we prove a three circles property for solutions $w: \Sigma \times (0,\frac{7}{8}] \to \R$ of the parabolic equation $(\del_t - \Delta_{(1-t)g_X})w = (\del_t - \frac{1}{1-t}\Delta_{g_X})w = 0$. By Remark \ref{rmk:c1a-metric}, $\Delta_{g_X}$ exists classically with $C^{0,\alpha}$ coefficients, and its eigenfunctions are $C^{2,\alpha}$ by elliptic regularity.

Let $\lambda_k$ be an eigenvalue of $-\Delta_{g_X}$, with $L^2(g_X)$-orthonormal eigenfunctions $\Theta_k^{(1)},\ldots,\Theta_k^{(m_k)} \in C^{2,\alpha}(\Sigma)$. Then the functions
\begin{align}
    F_k^{(i)}(\theta,t) = (1-t)^{\lambda_k} \Theta_k^{(i)}(\theta) \quad \text{on } \Sigma \times [0,\tfrac{7}{8}]
\end{align}
have regularity $C^{2,1}$ and satisfy
\begin{align}
    (\del_t - \Delta_{(1-t)g_X}) F_k^{(i)} = 0 \quad \text{on } \Sigma \times (0,\tfrac{7}{8}].
\end{align}
The $F_k^{(i)}$ in fact account for all classical solutions:

\begin{lemma} \label{lem:classification-of-sols}
    Let $w \in C^{2,1}(\Sigma \times (0,\frac{7}{8}])$ be a classical solution to $(\del_t - \Delta_{(1-t)g_X})w = 0$ on $\Sigma \times (0,\frac{7}{8}]$. Then $w$ is an $L^2$-convergent sum of the $F_k^{(i)}$. In particular, $w$ extends continuously to $\Sigma \times [0,\frac{7}{8}]$.
\end{lemma}
\begin{proof}
    For any $\tau \in (0,\frac{7}{8}]$, we can $L^2$-orthogonally decompose $w(\cdot,\tau)$ as
    \begin{align}
        w(\theta,\tau) &= \sum_{k=1}^\infty \sum_{i=1}^{m_k} a_k^{(i)}(\tau) \Theta_k^{(i)}(\theta)
    \end{align}
    for some numbers $a_k^{(i)}(\tau) \in \R$.
    The function sending $(\theta,t)$ to
    \begin{align} \label{eq:w-parab}
        \sum_{k=1}^\infty \sum_{i=1}^{m_k} a_k^{(i)}(\tau) \frac{(1-t)^{\lambda_k}}{(1-\tau)^{\lambda_k}} \Theta_k^{(i)}(\theta) = \sum_{k=1}^\infty \sum_{i=1}^{m_k} \frac{a_k^{(i)}(\tau)}{(1-\tau)^{\lambda_k}} F_k^{(i)}(\theta,t)
    \end{align}
    is also a classical solution to $(\del_t - \Delta_{(1-t)g_X})w = 0$ and agrees with $w$ on $\Sigma \times \{\tau\}$. By the maximum principle, $w$ equals \eqref{eq:w-parab} everywhere on $\Sigma \times (0,\tfrac{7}{8}]$. This implies the first claim. Taking $t \downarrow 0$ in \eqref{eq:w-parab} yields the second claim.
\end{proof}

Lemma \ref{lem:classification-of-sols} lends way to the following three circles theorem for solutions of $(\del_t - \Delta_{(1-t)g_X})w = 0$. The proof is essentially the same as \cite{xu}*{Lemma 3.1}; see also \cite{ding}*{Lemma 1.1}.

\begin{lemma} \label{lem:3circles-model}
    Let $w$ be a classical solution to $(\del_t - \Delta_{(1-t)g_X})w = 0$ on $\Sigma \times (0,\frac{7}{8}]$, and let $d > 0$. Then
    \begin{align} \label{eq:classif-cond1}
        \int_\Sigma w(\cdot, 0)^2 \, \dvol_{g_X} \leq 2^{2d} \int_\Sigma w(\cdot,1/2)^2 \, \dvol_{g_X}
    \end{align}
    implies
    \begin{align} \label{eq:classif-cond2}
        \int_\Sigma w(\cdot,1/2)^2 \, \dvol_{g_X} \leq 2^{2d} \int_\Sigma w(\cdot,3/4)^2 \, \dvol_{g_X}.
    \end{align}
    Equality in \eqref{eq:classif-cond2} is achieved if and only if either (i) $w \equiv 0$, or (ii) $d = \lambda_k$ for some $k$ and $w(\theta,t) = c(1-t)^{\lambda_k}\Theta_k(\theta)$ for some constant $c \in \R$ and some eigenfunction $-\Delta_{g_X}\Theta_k = \lambda_k\Theta_k$.
\end{lemma}
\begin{proof}
    By Lemma \ref{lem:classification-of-sols}, we can write
    \begin{align}
        w(\theta,t) &= \sum_{k=1}^\infty \sum_{i=1}^{m_k} a_k^{(i)} F_k^{(i)}(\theta,t)
    \end{align}
    for some fixed constants $a_k^{(i)} \in \R$. Since the $\Theta_k^{(i)}$ are $L^2(g_X)$-orthonormal, $\lambda_1 = 0$, and $m_1 = 1$, the first condition \eqref{eq:classif-cond1} is equivalent to having
    \begin{align} \label{eq:cond1}
        \sum_{k=2}^\infty \sum_{i=1}^{m_k} (a_k^{(i)})^2 (1-2^{2d-2\lambda_k}) \leq (a_1^{(1)})^2 (2^{2d} - 1).
    \end{align}
    Meanwhile, the second condition \eqref{eq:classif-cond2}
    is equivalent to having
    \begin{align} \label{eq:cond2}
        \sum_{k=2}^\infty \sum_{i=1}^{m_k} 2^{-2\lambda_k} (a_k^{(i)})^2 (1-2^{2d-2\lambda_k}) \leq (a_1^{(1)})^2 (2^{2d} - 1).
    \end{align}
    For each $k \geq 1$, regardless of whether $d \geq \lambda_k$ or $d < \lambda_k$, the following inequality holds:
    \begin{align} \label{eq:00001}
        2^{-2\lambda_k} (1-2^{2d-2\lambda_k}) \leq 2^{-2d} (1-2^{2d-2\lambda_k}).
    \end{align}
    Hence
    \begin{align}
        \sum_{k=2}^\infty \sum_{i=1}^{m_k} 2^{-2\lambda_k} (a_k^{(i)})^2 (1-2^{2d-2\lambda_k}) &\leq 2^{-2d} \sum_{k=2}^\infty \sum_{i=1}^{m_k} (a_k^{(i)})^2 (1-2^{2d-2\lambda_k}), \label{eq:00002}
    \end{align}
    with equality if and only if for each $k \geq 2$ with $\lambda_k \neq d$, we have $a_k^{(i)} \neq 0$.
    We now assume that \eqref{eq:cond1} holds and split into two cases, showing that \eqref{eq:cond2} holds in each.
    \begin{itemize}
        \item If $\sum_{k=2}^\infty \sum_{i=1}^{m_k} (a_k^{(i)})^2 (1-2^{2d-2\lambda_k}) \leq 0$, then \eqref{eq:00002} is further bounded by $\leq 0 \leq (a_1^{(1)})^2 (2^{2d}-1)$, so \eqref{eq:cond2} holds.
        Equality in \eqref{eq:cond2} is therefore satisfied if and only if equality in \eqref{eq:00002} holds, \emph{and} $a_1^{(1)} = 0$.
        \item If $\sum_{k=2}^\infty \sum_{i=1}^{m_k} (a_k^{(i)})^2 (1-2^{2d-2\lambda_k}) \geq 0$, then we use \eqref{eq:cond1} to get the following upper bound for \eqref{eq:00002}
        \begin{align} \label{eq:00003}
            \eqref{eq:00002} \leq 2^{-2d} (a_1^{(1)})^2 (2^{2d}-1) \leq (a_1^{(1)})^2 (2^{2d}-1).
        \end{align}
        Thus \eqref{eq:cond2} holds as well. Equality in \eqref{eq:cond2} is satisfied if and only if equality in \eqref{eq:00002} holds (in which case the first inequality in \eqref{eq:00003} is an equality), \emph{and} the last inequality is an equality i.e. $a_1^{(1)} = 0$.
    \end{itemize}
\end{proof}

\subsection{A three circles theorem for drift-harmonic functions} \label{subsec:3circ}

We now combine Lemma \ref{lem:3circles-model} with a blowdown argument to obtain a three circles theorem for drift-harmonic functions on the AP manifold $(M^n,g,r)$. The reader may wish to revisit \S\ref{subsec:pointwise-estimates} before proceeding, as the notation and results there will be used here as well as in \S\ref{subsec:construction-from-ef}.

\begin{theorem} \label{thm:3circles}
    Let $d > 0$, $d \neq \lambda_k$ for any $k$. Then there exists $R_d > 0$ such that if $\rho \geq R_d$ and $u: \overline{B}_\rho \to \R$ satisfies $\L_f u = 0$ on $B_\rho$, then
    \begin{align}
        I_u(\rho) \leq 2^{2d} I_u\left( \frac{\rho}{2} \right) \quad \text{implies} \quad I_u\left( \frac{\rho}{2} \right) \leq 2^{2d} I_u\left( \frac{\rho}{4} \right).
    \end{align}
\end{theorem}
\begin{proof}
    Otherwise, there exist a sequence $\rho_i \to \infty$ and functions $v_i: \overline{B}_{\rho_i} \to \R$ such that $\L_f v_i = 0$ on $B_{\rho_i}$ and
    \begin{align}
        I_{v_i}(\rho_i) \leq 2^{2d} I_{v_i}\left( \frac{\rho_i}{2} \right) \quad \text{but} \quad I_{v_i}\left( \frac{\rho_i}{2} \right) > 2^{2d} I_{v_i}\left( \frac{\rho_i}{4} \right).
    \end{align}
    The strict inequality on the right allows us to define
    \begin{align}
        u_i := \frac{v_i}{\sqrt{I_{v_i}(\rho_i/2)}}: \overline{B}_{\rho_i} \to \R,
    \end{align}
    which satisfy $\L_f u_i = 0$ on $B_{\rho_i}$ and
    \begin{align}
        I_{u_i}(\rho_i) &\leq 2^{2d}, \label{eq:3circ1} \\
        I_{u_i}\left( \frac{\rho_i}{2} \right) &= 1, \label{eq:3circ2} \\
        I_{u_i}\left( \frac{\rho_i}{4} \right) &< 2^{-2d}. \label{eq:3circ3}
    \end{align}
    Define $w_i := \Psi_{\rho_i}^* \hat{u}_i^{(\rho_i)}$.
    Let $\tau \in (0,\frac{1}{2})$ and $\alpha \in (0,1)$. By Theorem \ref{thm:parabolicSchauder}, Lemma \ref{lem:phi-bounds}, and the maximum principle, there exist $C = C(\tau,\alpha)$ and $i_0 = i_0(\tau)$ such that for all $i \geq i_0$,
    \begin{align}
        \norm{w_i}_{C^{2+\alpha,1+\frac{\alpha}{2}}(\overline{\Omega}^{\rho_0}_{\tau} \times [\tau,\frac{7}{8}]; \Psi_{\rho_0}^*\hat{g}^{(\rho_0)}(0))}^2 &\leq C \norm{w_i}_{L^\infty(\overline{\Omega}^{\rho_0}_{\tau/2} \times [\frac{\tau}{2},\frac{7}{8}])}^2 \\
        &= C \sup \left\{ |u_i(\Phi_{\rho_i t}(\Psi_{\rho_i}(x)))| : (x,t) \in \overline{\Omega}^{\rho_0}_{\tau/2} \times [\tfrac{\tau}{2}, \tfrac{7}{8}] \right\}^2 \\
        &= C \sup \left\{ |u_i(\Phi_{\rho_i t}(y))| : (y,t) \in \overline{\Omega}^{\rho_i}_{\tau/2} \times [\tfrac{\tau}{2}, \tfrac{7}{8}] \right\}^2 \\
        &\leq C \sup_{\{r \leq (1-\tau/4)\rho_i\}} |u_i|^2 = C \sup_{\{r = (1-\tau/4)\rho_i\}} |u_i|^2. \label{eq:first-holder-bd}
    \end{align}
    Applying Theorem \ref{thm:meanValue}, Corollary \ref{cor:I-almost-mono}, and \eqref{eq:3circ1}, there exists $C = C(\tau,\alpha,d)$ such that for all $i \geq i_0(\tau)$,
    \begin{align}
        \norm{w_i}_{C^{2+\alpha,1+\frac{\alpha}{2}}(\overline{\Omega}^{\rho_0}_{\tau} \times [\tau,\frac{7}{8}]; \Psi_{\rho_0}^*\hat{g}^{(\rho_0)}(0))}^2 \leq C \rho_i^{-\frac{n+1}{2}} \int_{\frac{1}{32}\rho_i}^{\rho_i} s^{\frac{n-1}{2}} I_{u_i}(s) \, ds \leq C.
    \end{align}
    By Theorem \ref{thm:pHolder-cpt-embedding} and taking $\tau \to 0$, some subsequence of $w_i$ (which we continue to denote by $w_i$) converges in $C^{2,1}$ on compact subsets of $\Omega^{\rho_0} \times (0,\frac{7}{8}]$ to a limiting function $w_\infty$.
    
    By Lemma \ref{lem:unif-ctrl}, the metrics $\Psi_{\rho_i}^*\hat{g}^{(\rho_i)}(t)$ are uniformly $C^2$-controlled in space over $\overline{\Omega}^{\rho_0} \times [0,\frac{7}{8}]$ and converge uniformly to $g_\infty(t) := \rho_0^{-1} dr^2 + (1-t)g_X$ on the same domain. Passing to a further subsequence, we have $\Psi_{\rho_i}^*\hat{g}^{(\rho_i)}(t) \to g_\infty(t)$ in spatial $C^1$, uniformly on $\overline{\Omega}^{\rho_0} \times [0,\frac{7}{8}]$. Together with the local $C^{2,1}$ convergence $w_i \to w_\infty$ and the fact that $(\del_t - \Delta_{\Psi_{\rho_i}^*\hat{g}^{(\rho_i)}(t)}) w_i = 0$ (Lemma \ref{lem:conversion-to-parabolic}), it follows that
    \begin{align} \label{eq:winfty-heat}
        (\del_t - \Delta_{g_\infty(t)}) w_\infty = 0 \quad \text{on } \Omega^{\rho_0} \times (0,\tfrac{7}{8}].
    \end{align}
    We also make the following claims:
    
    \begin{lemma} \label{lem:limit-claims}
        \begin{enumerate}[label=(\alph*)]
            \item $w_\infty$ is $r$-invariant. That is, if $t \in (0,\frac{7}{8}]$, and $x,y \in \Omega^{\rho_0}$ have the same $\theta$-coordinate (where we are using $(r,\theta)$ coordinates on $\{r > 0\}$; see \S\ref{subsec:conventions}), then $w_\infty(x,t) = w_\infty(y,t)$. Hence there is a well-defined function $\omega_\infty: \Sigma \times (0,\frac{7}{8}] \to \R$ defined by
            \begin{align} \label{eq:omega-infty-def}
                \omega_\infty(\theta,t) = w_\infty(r,\theta,t) \quad \text{for any } r \in (\rho_0-\sqrt{\rho_0}, \rho_0).
            \end{align}
            \item The function $\omega_\infty: \Sigma \times (0,\frac{7}{8}] \to \R$ satisfies $(\del_t - \Delta_{(1-t)g_X}) \omega_\infty = 0$ and extends continuously to $\Sigma \times [0,\frac{7}{8}]$. Moreover, for each $t \in (0,\frac{7}{8}]$ we have
            \begin{align} \label{eq:I-limit-01010}
                \lim_{i\to\infty} I_{u_i}((1-t)\rho_i) = \int_\Sigma \omega_\infty(\cdot,t)^2 \, \dvol_{g_X}.
            \end{align}
        \end{enumerate}
    \end{lemma}
    \begin{proof}
        Let $\tau \in (0,\frac{1}{2})$. Let $t \in [\tau,\frac{7}{8}]$, and suppose $x,y \in \overline{\Omega}^{\rho_0}_{\tau}$ have the same $\theta$-coordinate.
        Let $\epsilon > 0$. We will show that there exists $i_0 = i_0(\epsilon,\tau)$ such that
        \begin{align} \label{eq:to-prove-b}
            |w_i(x,t) - w_i(y,t)| < \epsilon \quad \text{for all } i \geq i_0.
        \end{align}
        This will prove (a) in view of the uniform convergence $w_i \to w_\infty$ on $\overline{\Omega}^{\rho_0}_{\tau} \times [\tau,\frac{7}{8}]$.

        Recall that $\Psi_{\rho_i}: \overline{\Omega}^{\rho_0}_{\tau} \to \overline{\Omega}^{\rho_i}_{\tau}$ is a diffeomorphism, and $\Phi_t$ acts via $\Phi_t(r,\theta) = (\phi_t(r),\theta)$. From this and Lemma \ref{lem:phi-bounds}, it holds for all large $i$ (depending on $\tau$ but not $t \in [0,\frac{7}{8}]$),
        \begin{align}
            \Phi_{\rho_i t}(\Psi_{\rho_i}(x)), \Phi_{\rho_i t}(\Psi_{\rho_i}(y)) \in \Phi_{\rho_i t}(\overline{\Omega}^{\rho_i}_{\tau}) \subset \{ (1-t)\rho_i - \sqrt{\rho_i} - C \leq r \leq (1-t)\rho_i + C \},
        \end{align}
        where $C$ is independent of $i$.
        So $\Phi_{\rho_i t}(\Psi_{\rho_i}(x))$ and $\Phi_{\rho_i t}(\Psi_{\rho_i}(y))$ have the same $\theta$-coordinate, and their $r$-coordinates differ by at most $\sqrt{\rho_i}+2C \leq 2\sqrt{\rho_i}$. As $|\nabla f|$ is bounded, there must exist $s$ within $\frac{C}{\sqrt{\rho_i}}$ of $t$ such that $\Phi_{\rho_i s}(\Psi_{\rho_i}(x)) = \Phi_{\rho_i t}(\Psi_{\rho_i}(y))$. As $t \in [\tau,\frac{7}{8}]$, we can enlarge $i$ sufficiently (depending on $\tau$) so that $s \in [\frac{\tau}{2}, \frac{7}{8}]$. For such $i$, we therefore have
        \begin{align}
            w_i(x,s) = \Psi_{\rho_i}^* \hat{v}_i^{(\rho_i)}(x,s) = u_i(\Phi_{\rho_i s}(\Psi_{\rho_i}(x))) = u_i(\Phi_{\rho_i t}(\Psi_{\rho_i}(y))) = \Psi_{\rho_i}^*\hat{v}_i^{(\rho_i)}(y,t) = w_i(y,t)
        \end{align}
        and so
        \begin{align}
            |w_i(x,t) - w_i(y,t)| = |w_i(x,t) - w_i(x,s)|.
        \end{align}
        Using that $(x,t), (x,s) \in \overline{\Omega}^{\rho_0}_\tau \times [\frac{\tau}{2}, \frac{7}{8}]$, and that $w_i \to w_\infty$ uniformly there, the right-hand side is bounded by $\epsilon$ for all $i$ larger than $i_0 = i_0(\epsilon,\tau)$. This proves \eqref{eq:to-prove-b} and hence (a).

        Let $\omega_\infty: \Sigma \times (0,\frac{7}{8}] \to \R$ be given by \eqref{eq:omega-infty-def}. Using \eqref{eq:winfty-heat}, one has $(\del_t - \Delta_{(1-t)g_X}) \omega_\infty = 0$ on $\Sigma \times (0,\frac{7}{8}]$. By Lemma \ref{lem:classification-of-sols}, $\omega_\infty$ extends to a continuous function on $\Sigma \times [0,\frac{7}{8}]$. Now let $t \in (0,\frac{7}{8}]$, and for each $i \in \N$ let
        \begin{align}
            \tilde{s}_i := \phi_{-\left(t-\frac{1}{2\sqrt{\rho_i}} \right)\rho_i}((1-t)\rho_i)
        \end{align}
        and $s_i := \psi_{\rho_i}^{-1}(\tilde{s}_i)$.
        Using the definitions of $I_{u_i}$, $s_i$ and $w_i$, we therefore have
        \begin{align}
            I_{u_i}((1-t)\rho_i) &= \int_\Sigma u_i((1-t)\rho_i,\theta)^2 \, \dvol_{g_X((1-t)\rho_i)}(\theta) \\
            &= \int_\Sigma u_i(\phi_{\left(t - \frac{1}{2\sqrt{\rho_i}}\right)\rho_i}(\psi_{\rho_i}(s_i)),\theta)^2 \, \dvol_{g_X((1-t)\rho_i)}(\theta) \\
            &= \int_\Sigma w_i(s_i,\theta,t - \tfrac{1}{2\sqrt{\rho_i}})^2 \, \dvol_{g_X((1-t)\rho_i)}(\theta). \label{eq:conv-ui-1}
        \end{align}
        By Lemma \ref{lem:phi-bounds}, for all large $i$ we have $\phi_{(t-\frac{1}{2\sqrt{\rho_i}})\rho_i}(\rho_i - \frac{3}{4}\sqrt{\rho_i}) \leq (1-t)\rho_i \leq \phi_{(t-\frac{1}{2\sqrt{\rho_i}})\rho_i}(\rho_i - \frac{1}{4}\sqrt{\rho_i})$. Thus $\tilde{s}_i \in [\rho_i - \frac{3}{4}\sqrt{\rho_i}, \rho_i - \frac{1}{4}\sqrt{\rho_i}]$, and $s_i \in [\rho_0 - \frac{3}{4}\sqrt{\rho_0}, \rho_0 - \frac{1}{4}\sqrt{\rho_0}]$. The equality \eqref{eq:I-limit-01010} now holds from taking $i\to\infty$ in \eqref{eq:conv-ui-1}, and using the uniform convergence $w_i \to w_\infty$ on $\{\rho_0 - \frac{3}{4}\sqrt{\rho_0} \leq r \leq \rho_0 - \frac{1}{4}\sqrt{\rho_0}\} \times [\frac{t}{2}, \frac{7}{8}]$ as well as the convergence of metrics from Theorem \ref{thm:asymp-link}.
    \end{proof}

    We now finish off the proof of Theorem \ref{thm:3circles}. By Lemma \ref{lem:limit-claims}(b) and \eqref{eq:3circ2}, \eqref{eq:3circ3}, we have
    \begin{align}
        \int_\Sigma \omega_\infty(\cdot,1/2)^2 \, \dvol_{g_X} &= 1, \label{eq:3circ-31} \\
        \int_\Sigma \omega_\infty(\cdot,3/4)^2 \, \dvol_{g_X} &\leq 2^{-2d}. \label{eq:3circ-32}
    \end{align}
    Meanwhile, as $\omega_\infty$ extends continuously to $\Sigma \times [0,\frac{7}{8}]$, the following limit exists:
    \begin{align} \label{eq:tzero-limit}
        \int_\Sigma \omega_\infty(\cdot,0)^2 \, \dvol_{g_X} = \lim_{t \downarrow 0} \int_\Sigma \omega_\infty(\cdot,t)^2 \, \dvol_{g_X} = \lim_{t \downarrow 0} \lim_{i\to\infty} I_{u_i}((1-t)\rho_i).
    \end{align}
    By Corollary \ref{cor:I-almost-mono} and \eqref{eq:3circ1}, there exists $C>0$ such that for all $i \in \N$ and $t \in (0,\frac{7}{8}]$, we have $I_{u_i}((1-t)\rho_i) \leq e^{C(\rho_i/8)^{-\mu}} I_{u_i}(\rho_i) \leq e^{C(\rho_i/8)^{-\mu}} 2^{2d}$. Using this in \eqref{eq:tzero-limit}, we get
    \begin{align} \label{eq:tzero-limit-2}
        \int_\Sigma \omega_\infty(\cdot,0)^2 \, \dvol_{g_X} \leq 2^{2d}.
    \end{align}
    In conclusion, $\omega_\infty: \Sigma \times [0,\frac{7}{8}] \to \R$ satisfies $(\del_t - \Delta_{(1-t)g_X}) \omega_\infty = 0$, and by \eqref{eq:3circ-31}, \eqref{eq:3circ-32}, \eqref{eq:tzero-limit-2},
    \begin{align}
        \int_\Sigma \omega_\infty(\cdot,0)^2 \, \dvol_{g_X} \leq 2^{2d} \int_\Sigma \omega_\infty(\cdot,1/2)^2 \, \dvol_{g_X} \quad \text{and} \quad \int_\Sigma \omega_\infty(\cdot,1/2)^2 \, \dvol_{g_X} \geq 2^{2d} \int_\Sigma \omega_\infty(\cdot,3/4)^2 \, \dvol_{g_X}.
    \end{align}
    Since $d \neq \lambda_k$ for all $k$, the rigidity part of Lemma \ref{lem:3circles-model} gives $\omega_\infty \equiv 0$. This contradicts \eqref{eq:3circ-31}.
\end{proof}

\subsection{A sequence of $\L_f$-superharmonic functions} \label{subsec:seq-superharmonic}

For each $\tau > 0$, let $\Pi_\tau: \{r > 0\} \to \{r > 0\}$ be the diffeomorphism given by $\Pi_\tau(r,\theta) = (\tau r,\theta)$. Recall from Corollary \ref{cor:weakly-parab} that the rescaled metrics
\begin{align}
    g_\tau := dr^2 + \tau^{-1} g_{\Sigma_{\tau r}} = dr^2 + r g_X(\tau r)
\end{align}
satisfy
\begin{align} \label{eq:gtau-convergence-2}
    \lim_{\tau \to \infty} g_\tau = g_P := dr^2 + r g_X \quad \text{in } C^0(\{\tfrac{1}{2} \leq r \leq \tfrac{3}{2}\}).
\end{align}
This also implies uniformity of distance functions, a version of which we state next:
\begin{lemma} \label{lem:dist-lemma}
    There exists $C>0$ such that for all $\tau \geq 2$, all $z_1 \in \{r=\frac{3}{2}\}$ and all $z_2 \in \{\frac{1}{2} \leq r \leq 1\}$,
    \begin{align}
        C^{-1} \leq d_{g_\tau}(z_1,z_2) \leq C.
    \end{align}
\end{lemma}

\begin{lemma} \label{lem:laplace-warped-distance}
    There exists $C>0$ such that for all $\tau \geq 2$ and $q \in \{r=\frac{3}{2}\tau\}$,
    \begin{align}
        \sup_{\{\frac{1}{2} \tau \leq r \leq \tau\}} \left| \Delta_g \left[ d_{g_\tau}(\Pi_\tau^{-1}(q), \Pi_\tau^{-1}(\cdot)) \right] \right| \leq \frac{C}{\tau}.
    \end{align}
\end{lemma}
\begin{proof}
    Let $h_\tau(x) := d_{g_\tau}(\Pi_\tau^{-1}(q), \Pi_\tau^{-1}(x))$. We use $(r,\theta)$ coordinates, with Greek indices $(\alpha,\beta,\ldots)$ running over only the $\theta$ factor. Then
    \begin{align} \label{eq:laplace-h}
        \Delta_g h_\tau &= g^{rr} \del_r \del_r h_\tau + g^{\alpha\beta} \del_\alpha \del_\beta h_\tau - g^{rr} \Gamma_{rr}^k \del_k h_\tau - g^{\alpha\beta} \Gamma_{\alpha\beta}^k \del_k h_\tau.
    \end{align}
    Using that $D\Pi_\tau^{-1}|_x(\del_r) = \frac{1}{\tau}\del_r$ and $D\Pi_\tau^{-1}|_x(\del_\alpha) = \del_\alpha$, we compute
    \begin{align}
        \del_r h_\tau(x) &= \frac{1}{\tau} \inner{\nabla^{g_\tau} d_{g_\tau}(\Pi_\tau^{-1}(q), \cdot)}{\del_r}_{g_{\tau}(\Pi_\tau^{-1}(x))} =: \psi_r(\Pi_\tau^{-1}(x)), \label{eq:drht} \\
        \del_\alpha h_\tau(x) &= \inner{\nabla^{g_\tau} d_{g_\tau}(\Pi_\tau^{-1}(q_i), \cdot)}{\del_\alpha}_{g_\tau(\Pi_\tau^{-1}(x))} =: \psi_\alpha(\Pi_\tau^{-1}(x)). \label{eq:drha}
    \end{align}
    Now let $x \in \{\frac{1}{2}\tau \leq r \leq \tau\}$. Then $|\del_r h_\tau(x)| \leq \frac{C}{\tau}$ and $|\del_\alpha h_\tau(x)| \leq C$. Plugging these into \eqref{eq:laplace-h} and using Lemma \ref{lem:first-0-1-control}, we see that at $x$,
    \begin{align} \label{eq:laplace-h-2}
        \Delta_g h_\tau(x) &= \O(1)\del_r \del_r h_\tau + \O(\tau^{-1}) \del_\alpha \del_\beta h_\tau + \O(\tau^{-\mu-\frac{3}{2}}) + \O(\tau^{-1}).
    \end{align}
    From \eqref{eq:drht} and \eqref{eq:drha}, we compute
    \begin{align}
        \del_r \del_r h_\tau(x)
        &= \frac{1}{\tau^2} \Big[(\nabla^{g_\tau})^2 d_{g_\tau}(\Pi_\tau^{-1}(q), \cdot) \Big]\Big|_{\Pi_\tau^{-1}(x)} (\del_r,\del_r) + \frac{1}{\tau^2} (\Gamma^{g_\tau})_{rr}^k \inner{\nabla^{g_\tau} d_{g_\tau}(\Pi_\tau^{-1}(q), \cdot)}{\del_k}_{g_\tau(\Pi_\tau^{-1}(x))}, \\
        \del_\alpha \del_\beta h_\tau(x)
        &= \Big[(\nabla^{g_\tau})^2 d_{g_\tau}(\Pi_\tau^{-1}(q), \cdot) \Big] \Big|_{\Pi_\tau^{-1}(x)} (\del_\alpha, \del_\beta) + (\Gamma^{g_\tau})_{\alpha\beta}^k \inner{\nabla^{g_\tau} d_{g_\tau}(\Pi_\tau^{-1}(q), \cdot)}{\del_k}_{g_\tau(\Pi_\tau^{-1}(x))}.
    \end{align}
    So by Lemma \ref{lem:dist-lemma}, Corollary \ref{cor:gtau-components}, and the Hessian comparison theorem (which applies due to Corollary \ref{cor:gtau-components}),
    \begin{align}
        \sup_{\{\frac{1}{2}\tau \leq r \leq \tau\}} |\del_r \del_r h_\tau| &= \O(\tau^{-2}), \quad \sup_{\{\frac{1}{2}\tau \leq r \leq \tau\}} |\del_\alpha \del_\beta h_\tau| = \O(1).
    \end{align}
    Plugging this into \eqref{eq:laplace-h-2} proves the claim.
\end{proof}

\begin{corollary} \label{cor:barrier}
    There exists $K>0$ such that for all sufficiently large $\tau$ and all $y \in \{r=\tau\}$, the function
    \begin{align}
        b_{\tau,y}(x) &:= d_{g_\tau}(\Pi_\tau^{-1}(\Pi_{3/2}(y)), \Pi_\tau^{-1}(y))^{2-n} - d_{g_\tau}(\Pi_\tau^{-1}(\Pi_{3/2}(y)), \Pi_\tau^{-1}(x))^{2-n} + \frac{K}{\tau}(r(y) - r(x))
    \end{align}
    satisfies $\L_f b_{\tau,y}(x) \leq 0$ for all $x \in \{\frac{1}{2}\tau \leq r \leq \tau\}$.
\end{corollary}
\begin{proof}
    Let $y \in \{ r = \tau\}$ and write $q := \Pi_{3/2}(y)$.
    Using Lemmas \ref{lem:dist-lemma} and \ref{lem:laplace-warped-distance}, we have for all $x \in \{\frac{1}{2}\tau \leq r \leq \tau\}$,
    \begin{equation}
        \Delta_g \left[ d_{g_\tau}(\Pi_\tau^{-1}(q), \Pi_\tau^{-1}(x))^{2-n} \right] \geq (2-n) d_{g_\tau}(\Pi_\tau^{-1}(q), \Pi_\tau^{-1}(x))^{1-n} \Delta_g \left[ d_{g_\tau}(\Pi_\tau^{-1}(q), \Pi_\tau^{-1}(x)) \right] \geq -\frac{C}{\tau}. \label{eq:computation1}
    \end{equation}
    Also,
    \begin{align}
        \inner{\nabla f}{\nabla d_{g_\tau}(\Pi_\tau^{-1}(q), \Pi_\tau^{-1}(x))^{2-n}}
        &= \underbrace{(2-n) d_{g_\tau}(\Pi_\tau^{-1}(q), \Pi_\tau^{-1}(x))^{1-n} f'(r(x)) |\nabla r|^{-2}}_{=\O(1) \text{ as } \tau\to\infty} \inner{\del_r}{\nabla d_{g_\tau}(\Pi_\tau^{-1}(q), \Pi_\tau^{-1}(x))} \\
        &= \O(\tau^{-1}) Dd_{g_\tau}(\Pi_\tau^{-1}(q_i),\cdot)|_{\Pi_\tau^{-1}(x)} (\del_r)
        = \O(\tau^{-1}). \label{eq:computation2}
    \end{align}
    Meanwhile, on the domain $\{\frac{1}{2}\tau \leq r \leq \tau\}$ for large values of $\tau$, we have
    \begin{align}
        \L_f r &= \Delta r - \inner{\nabla f}{\nabla r} = \frac{1}{2r}\tr_g (g-dr^2+\eta) - |\nabla r|^2 f'(r) = 1 + \O(r^{-\mu}) \geq \frac{1}{2}. \label{eq:computation3}
    \end{align}
    By \eqref{eq:computation1}, \eqref{eq:computation2} and \eqref{eq:computation3}, it holds for all large $\tau$ and $x \in \{\frac{1}{2}\tau \leq r \leq \tau\}$ that
    \begin{align}
        \L_f b_{\tau,y}(x) &= -\L_f \left[ d_{g_\tau}(\Pi_\tau^{-1}(q_i), \Pi_\tau^{-1}(x))^{2-n} \right] - \frac{K}{\tau} \L_f r(x) \leq \frac{C}{\tau} - \frac{K}{2\tau}.
    \end{align}
    Choosing $K > 0$ sufficiently large, independently of $\tau$, the corollary follows.
\end{proof}

\begin{lemma} \label{lem:barrier-lemma}
    \begin{enumerate}[label=(\alph*)]
        \item For each $\delta > 0$, there exist $\sigma > 0$ and $\tau_0 \geq 1$ such that if $\tau \geq \tau_0$ and $x,y \in \{r=\tau\}$ have $d_{g_X}(x,y) \geq \delta$, then $b_{\tau,y}(x) \geq \sigma$.
        \item For each $\epsilon > 0$, there exists $\tau_1 \geq 1$ such that if $\tau \geq \tau_1$ and $x,y \in \{r=\tau\}$, then $b_{\tau,y}(x) \geq -\epsilon$. 
        \item There exists $A>0$ such that for all sufficiently large $\tau$, we have $b_{\tau,y}(x) \geq A$ for all $y \in \{r=\tau\}$ and $x \in \{\frac{1}{2}\tau \leq r \leq \frac{3}{4}\tau\}$.
    \end{enumerate}
\end{lemma}
\begin{proof}
    Let $x,y \in \{r=\tau\}$ be such that $d_{g_X}(x,y) \geq \delta$. Note that $d_{g_X}(x,y)$ is the orbital distance between the points $\Pi_\tau^{-1}(x), \Pi_\tau^{-1}(y) \in \{r=1\}$ with respect to $g_P$. Considering the geometry of $g_P$, there exists $\sigma_1 > 0$ depending on $\delta$ such that $d_{g_P}(\Pi_\tau^{-1}(x), \Pi_\tau^{-1}(y)) \geq \sigma_1$. Then there exists $\sigma_2 > 0$ depending on $\delta$ such that
    \begin{align}
        d_{g_P}(\Pi_\tau^{-1}(\Pi_{3/2}(y)), \Pi_\tau^{-1}(x)) \geq d_{g_P}(\Pi_\tau^{-1}(\Pi_{3/2}(y)), \Pi_\tau^{-1}(y)) + \sigma_2.
    \end{align}
    By \eqref{eq:gtau-convergence-2}, it follows that after increasing $\tau$ sufficiently,
    \begin{align}
        d_{g_\tau}(\Pi_\tau^{-1}(\Pi_{3/2}(y)), \Pi_\tau^{-1}(x)) \geq d_{g_\tau}(\Pi_\tau^{-1}(\Pi_{3/2}(y)), \Pi_\tau^{-1}(y)) + \frac{\sigma_2}{2}.
    \end{align}
    From the definition of $b_{\tau,y}$ and using that $r(x) = r(y)$, part (a) of the lemma follows.

    The proof of (b) is similar. Namely, the geometry of $g_P$ gives that for all $x,y \in \{r=\tau\}$,
    \begin{align}
        d_{g_P}(\Pi_\tau^{-1}(\Pi_{3/2}(y)), \Pi_\tau^{-1}(y)) \leq d_{g_P}(\Pi_\tau^{-1}(\Pi_{3/2}(y)), \Pi_\tau^{-1}(x)).
    \end{align}
    Then by \eqref{eq:gtau-convergence-2}, for each $\epsilon > 0$ it holds for sufficiently large $\tau$ that
    \begin{align}
        d_{g_\tau}(\Pi_\tau^{-1}(\Pi_{3/2}(y)), \Pi_\tau^{-1}(y)) \leq d_{g_\tau}(\Pi_\tau^{-1}(\Pi_{3/2}(y)), \Pi_\tau^{-1}(x)) + \epsilon.
    \end{align}
    Then (b) follows from the definition of $b_{\tau,y}$. Likewise, for all $y \in \{r=\tau\}$ and $x \in \{\frac{1}{2}\tau \leq r \leq \frac{3}{4}\tau\}$ we have
    \begin{align}
        d_{g_P}(\Pi_\tau^{-1}(\Pi_{3/2}(y)), \Pi_\tau^{-1}(y)) &= \frac{1}{2}, \quad d_{g_P}(\Pi_\tau^{-1}(\Pi_{3/2}(y)), \Pi_\tau^{-1}(x)) \geq \frac{3}{4}.
    \end{align}
    These hold up to errors for $d_{g_\tau}$, and inserting this into the definition of $b_{\tau,y}$ proves (c).
\end{proof}

With $A>0$ from Lemma \ref{lem:barrier-lemma}, we set for each $\tau \geq 1$ and each $y \in \{r = \tau\}$,
\begin{align}
    \hat{b}_{\tau,y}(x) &:= \begin{cases}
        b_{\tau,y}(x) &\text{if } x \in \{\tfrac{7}{8}\tau \leq r \leq \tau\}, \\
        \min\{ b_{\tau,y}(x), A \} &\text{if } x \in \{\tfrac{3}{4}\tau \leq r \leq \tfrac{7}{8}\tau\}, \\
        A &\text{if } x \in B_{3\tau/4}.
    \end{cases}
\end{align}
By Lemma \ref{lem:barrier-lemma}(c), $\hat{b}_{\tau,y}$ is continuous on $B_\tau$. Moreover, by Corollary \ref{cor:barrier}, $\L_f \hat{b}_{\tau,y} \leq 0$ on $B_\tau$ in the barrier sense. Since $\hat{b}_{\tau,y}$ and $b_{\tau,y}$ coincide on $\{\frac{7}{8}\tau \leq r \leq \tau\}$, parts (a) and (b) of Lemma \ref{lem:barrier-lemma} still hold for $\hat{b}_{\tau,y}$.

\subsection{Drift-harmonic functions from $g_X$-eigenfunctions} \label{subsec:construction-from-ef}

In this subsection, we give a recipe which turns a sequence of solutions of the $\L_f$-Dirichlet problem on increasing balls into a global drift-harmonic function. This is conveyed by the next two propositions:

\begin{proposition} \label{prop:eigenfn-to-bounds}
    Let $\rho_i \to \infty$ be a sequence of real numbers, and $\lambda > 0$.
    For each $i \in \N$, let $\Theta_i: \Sigma \to \R$ satisfy $-\Delta_{g_X}\Theta_i = \lambda \Theta_i$ and $\norm{\Theta_i}_{L^2(g_X)} = 1$, and solve the following $\L_f$-Dirichlet problem:
    \begin{align}
        \begin{cases}
            \L_f u_i = 0 & \text{in } B_{\rho_i}, \\
            u_i = \Theta_i & \text{on } \{r = \rho_i\}.
        \end{cases}
    \end{align}
    Fix $p_0 \in M$. Then there exists a subsequence of $v_i := u_i - u_i(p_0)$ with the following property. For each $\epsilon > 0$, there exists $i_\epsilon \in \N$ such that for all $i \geq i_\epsilon$ in the subsequence, we have
    \begin{align} \label{eq:initial-3circ}
        I_{v_i}(\rho_i) \leq 2^{2(\lambda+\epsilon)} I_{v_i}\left(\frac{\rho_i}{2}\right).
    \end{align}
\end{proposition}

\begin{proposition} \label{prop:I-to-convergence}
    Let $\rho_i \to \infty$, $\lambda > 0$, and let $v_i: \overline{B}_{\rho_i} \to \R$ be a sequence of nonzero functions such that
    \begin{enumerate}[label=(\roman*)]
        \item $\L_f v_i = 0$ in $B_{\rho_i}$.
        \item For each $\epsilon > 0$, there exists $i_\epsilon \in \N$ such that for all $i \geq i_\epsilon$, we have $I_{v_i}(\rho_i) \leq 2^{2(\lambda+\epsilon)} I_{v_i}(\rho_i/2)$.
    \end{enumerate}
    Then there exists $\bar{\rho} > 0$ such that a subsequence of the normalized functions $\hat{v}_i := \frac{v_i}{\sqrt{I_{v_i}(\bar{\rho})}}$ converge uniformly on compact sets of $M$ to a nonzero drift-harmonic function $v \in \H_\lambda^+$.
\end{proposition}

Although the conclusions of Proposition \ref{prop:eigenfn-to-bounds} are precisely the hypotheses of Proposition \ref{prop:I-to-convergence}, we keep the propositions separate. This is because when we apply them in \S\ref{subsec:proof-Al}, we will add intermediate steps to ensure that the drift-harmonic functions thus constructed are linearly independent and asymptotically orthogonal.

\begin{proof}[Proof of Proposition \ref{prop:I-to-convergence}]
    Assume $\rho_i = 2^i$ for simplicity.
    Fix a small $\tau > 0$ so that $\lambda+\tau \neq \lambda_k$ for all $k \in \N$. Then for all $i \geq i_\tau$ we have $I_{v_i}(2^i) \leq 2^{2(\lambda+\tau)} I_{v_i}(2^{i-1})$. Increasing $i_\tau$ if needed, Theorem \ref{thm:3circles} allows to iterate this estimate inward, down to the scale $2^{i_\tau}$. Thus, for all $i \geq i_\tau + 1$ and $j \in \{i_\tau+1, i_\tau+2, \cdots, i\}$, we have
    \begin{align} \label{eq:itau}
        I_{v_i}(2^j) \leq 2^{2(j-i_\tau)(\lambda+\tau)} I_{v_i}(2^{i_\tau}).
    \end{align}
    Define the normalized functions $\hat{v}_i$ by
    \begin{align}
        \hat{v}_i := \frac{v_i}{\sqrt{I_{v_i}(2^{i_\tau})}},
    \end{align}
    so that
    \begin{align} \label{eq:nonzero-conditions}
        I_{\hat{v}_i}(2^{i_\tau}) = 1.
    \end{align}
    Now let $\epsilon \in (0,\tau)$ be arbitrary. Reasoning as in \eqref{eq:itau}, we can increase $i_\epsilon$ so that for any $i \geq i_\epsilon + 1$ and $j \in \{i_\epsilon+1, i_\epsilon+2, \ldots, i\}$, we have
    \begin{align} \label{eq:iepsilon}
        I_{v_i}(2^j) \leq 2^{2(j-i_\epsilon)(\lambda+\epsilon)} I_{v_i}(2^{i_\epsilon}).
    \end{align}
    We may also assume $i_\epsilon \geq i_\tau$. Then for any $i \geq i_\epsilon + 1$ and $j \in \{i_\epsilon+1, i_\epsilon+2, \ldots, i\}$, \eqref{eq:itau} and \eqref{eq:iepsilon} give
    \begin{align}
        I_{\hat{v}_i}(2^j) &= \frac{I_{v_i}(2^j)}{I_{v_i}(2^{i_\tau})} \leq \frac{2^{2(j-i_\epsilon)(\lambda+\epsilon)} I_{v_i}(2^{i_\epsilon})}{I_{v_i}(2^{i_\tau})} \leq 2^{2(j-i_\epsilon)(\lambda+\epsilon)} 2^{2(i_\epsilon-i_\tau)(\lambda+\tau)} \leq C_\epsilon (2^j)^{2(\lambda+\epsilon)}.
    \end{align}
    Then Corollary \ref{cor:I-almost-mono} gives $I_{\hat{v}_i}(\rho) \leq C_\epsilon \rho^{2(\lambda+\epsilon)}$ for all $i \geq i_\epsilon + 1$ and $\rho \leq 2^{i-1}$.
    By the mean value inequality (Theorem \ref{thm:meanValue}) and maximum principle, it follows that
    \begin{align} \label{eq:vi-bound}
        |\hat{v}_i| \leq C_\epsilon (1+r^{\lambda+\epsilon}) \quad \text{on } \overline{B}_{2^{i-2}}, \text{ for all } i \geq i_\epsilon+1.
    \end{align}
    By Corollary \ref{cor:deriv-growthbounds} and the Arzel\`a--Ascoli theorem, a subsequence of $\hat{v}_i$ converges in $C^1$ on compact sets of $M$ to some $v \in C^\infty(M)$. Then $v$ is a weak solution of $\L_f v = 0$, hence a classical solution by elliptic regularity ($\L_f$ has smooth coefficients). By \eqref{eq:nonzero-conditions} and \eqref{eq:vi-bound}, $v$ is nonzero with $v \in \H_\lambda^+$.
\end{proof}

The proof of Proposition \ref{prop:eigenfn-to-bounds} is more delicate. The key is to get uniform estimates near the boundary for each $u_i$. This is done in Proposition \ref{prop:unif-equi-1}, using the functions $\hat{b}_{\tau,y}$ from \S\ref{subsec:seq-superharmonic} to construct barriers. However, we first need uniform estimates for the boundary data:

\begin{lemma} \label{lem:equi-c2alpha}
    For each $\ell \in \N$, there exists $C > 0$ such that every eigenfunction $\Theta: \Sigma \to \R$ with $-\Delta_{g_X} \Theta = \lambda_\ell \Theta$ and $\norm{\Theta}_{L^2(g_X)} = 1$ satisfies
    \begin{align}
        \norm{\Theta}_{C^{2,\alpha}(\Sigma;g_X)} \leq C.
    \end{align}
\end{lemma}
\begin{proof}
    By finite-dimensionality of eigenspaces, there exists $C$ (depending on $\lambda_\ell)$ such that $\norm{\Theta}_{L^\infty(\Sigma)} \leq C$ for all eigenfunctions $-\Delta_{g_X} \Theta = \lambda_\ell\Theta$ with $\norm{\Theta}_{L^2(g_X)} = 1$. Also, by Remark \ref{rmk:c1a-metric}, the equation  $-\Delta_{g_X} \Theta = \lambda_\ell \Theta$ has H\"older continuous coefficients, so Schauder estimates give $\norm{\Theta}_{C^{2,\alpha}(\Sigma;g_X)} \leq C \norm{\Theta}_{L^\infty(\Sigma)}$ for all eigenfunctions $-\Delta_{g_X} \Theta = \lambda_\ell \Theta$. Putting these facts together, the lemma follows.
\end{proof}

\begin{proposition} \label{prop:unif-equi-1}
    For each $\epsilon > 0$, there exist $\delta > 0$ and $i_0 \in \N$ such that for all $i \geq i_0$ and $x,y \in \{(1-\delta)\rho_i \leq r \leq \rho_i\}$,
    \begin{enumerate}[label=(\alph*)]
        \item If $x,y$ have the same $\theta$ coordinate, then
        $|u_i(x) - u_i(y)| < \epsilon$.
        \item If $x,y$ have the same $r$-coordinate and $d_{g_X}(x,y) < \delta$, then $|u_i(x) - u_i(y)| < \epsilon$.
    \end{enumerate}
\end{proposition}
\begin{proof}
    Let $\epsilon > 0$ be given. By Lemma \ref{lem:equi-c2alpha}, the boundary data $u_i = \Theta_i$ are uniformly equicontinuous with respect to $g_X$. Hence, there exists $\xi > 0$ so that for all $i \in \N$,
    \begin{align} \label{eq:ui-equicont}
        |u_i(x) - u_i(y)| < \epsilon \quad \text{whenever } x,y \in \{r=\rho_i\} \text{ and } d_{g_X}(x,y) < \xi.
    \end{align}
    By Lemma \ref{lem:barrier-lemma}, there exist $\sigma > 0$ and $i_0 \in \N$ such that if $i \geq i_0$ and $x,y \in \{r=\rho_i\}$ satisfy $d_{g_X}(x,y) \geq \xi$, then $\hat{b}_{\rho_i,y}(x) \geq \sigma$. Hence we can find $\kappa > 0$ such that
    \begin{align} \label{eq:kappa-bi}
        \kappa \hat{b}_{\rho_i,y}(x) \geq 2\mathcal{M} \quad \text{whenever } i \geq i_0, \; x,y \in \{r=\rho_i\} \text{ and } d_{g_X}(x,y) \geq \xi,
    \end{align}
    where $\mathcal{M} = \sup_i|\Theta_i|$ which is finite by Lemma \ref{lem:equi-c2alpha}. By Lemma \ref{lem:barrier-lemma}, we can further increase $i_0$ so that $\kappa \hat{b}_{\rho_i,y}(x) \geq -\epsilon$ for all $x,y \in \{r=\rho_i\}$. It follows from this, \eqref{eq:ui-equicont} and \eqref{eq:kappa-bi} that
    \begin{align}
        u_i(x) + 2\epsilon + \kappa \hat{b}_{\rho_i,y}(x) \geq u_i(y) \quad \text{for all } i \geq i_0 \text{ and } x, y \in \{r = \rho_i\}.
    \end{align}
    Corollary \ref{cor:barrier} gives
    \begin{align}
        \L_f \left( u_i(x) + 2\epsilon + \kappa \hat{b}_{\rho_i,y}(x) \right) \leq 0 \quad \text{for all } y \in \{r=\rho_i\} \text{ and } x \in B_{\rho_i}.
    \end{align}
    Hence
    \begin{align}
        u_i(x) + 2\epsilon + \kappa \hat{b}_{\rho_i,y}(x) \geq u_i(y) \quad \text{for all } i \geq i_0, \, x \in B_{\rho_i}, \, y \in \{r = \rho_i\}.
    \end{align}
    Repeating these arguments, we have
    \begin{align}
        u_i(x) - 2\epsilon - \kappa \hat{b}_{\rho_i,y}(x) \leq u_i(y) \quad \text{for all } i \geq i_0, \, x \in B_{\rho_i}, \, y \in \{r = \rho_i\}.
    \end{align}
    Altogether, this gives
    \begin{align} \label{eq:ui-diff-b}
        |u_i(x) - u_i(y)| \leq 2\epsilon + \kappa \hat{b}_{\rho_i,y}(x) \quad \text{for all } i \geq i_0, \, x \in B_{\rho_i}, \, y \in \{r = \rho_i\}.
    \end{align}
    Let $\delta \in (0,\frac{1}{8})$ to be chosen.
    Suppose $z \in \{r=\rho_i\}$ and $x \in \{(1-\delta)\rho_i \leq r \leq \rho_i\}$ share the same $\theta$-coordinate. Then $\Pi_{\rho_i}^{-1}(\Pi_{3/2}(z))$, $\Pi_{\rho_i}^{-1}(z)$ and $\Pi_{\rho_i}^{-1}(x)$ all have the same $\theta$-coordinate, with $r$-coordinates $\frac{3}{2}$, $1$, and $\frac{r(x)}{\rho_i} \in [1-\delta,1]$ respectively. Then by the definition of $\hat{b}_{\rho_i,z}$,
    \begin{align}
        \hat{b}_{\rho_i,z}(x) &= b_{\rho_i,z}(x) = \underbrace{d_{g_\tau}(\Pi_{\rho_i}^{-1}(\Pi_{3/2}(z)), \Pi_{\rho_i}^{-1}(z))^{2-n} - d_{g_\tau}(\Pi_{\rho_i}^{-1}(\Pi_{3/2}(z)), \Pi_{\rho_i}^{-1}(x))^{2-n}}_{\text{small if } \delta \text{ small}} + \underbrace{\frac{K}{\rho_i}(\rho_i - r(x))}_{\leq K\delta}.
    \end{align}
    For $\delta$ small enough, this is $\leq \epsilon/\kappa$. That is,
    \begin{align} \label{eq:06101086}
        \kappa \hat{b}_{\rho_i,z}(x) < \epsilon \quad \text{for all } z \in \{r=\rho_i\} \text{ and } x \in \{(1-\delta)\rho_i \leq r \leq \rho_i\} \text{ with the same } \theta\text{-coordinates}.
    \end{align}
    
    Combining \eqref{eq:ui-diff-b} and \eqref{eq:06101086}, the following holds: if $i \geq i_0$, and $x,y \in \{(1-\delta)\rho_i \leq r \leq \rho_i\}$ have the same $\theta$-coordinate, then
    \begin{align}
        |u_i(x) - u_i(y)| \leq |u_i(x) - u_i(z)| + |u_i(z) - u_i(y)| \leq 6\epsilon,
    \end{align}
    where $z$ is the point on $\{r=\rho_i\}$ with the same $\theta$-coordinate as both $x$ and $y$. This proves part (a).

    Now shrink $\delta$ so that $\delta < \xi$ from earlier. Suppose $x,y \in \{(1-\delta)\rho_i \leq r \leq \rho_i\}$ have the same $r$-coordinate, and $d_{g_X}(x,y) < \delta$. Let $x_i$ and $y_i$ be the points on $\{r=\rho_i\}$ with the same $\theta$ coordinates as $x$ and $y$ respectively. Then using \eqref{eq:ui-equicont}, \eqref{eq:ui-diff-b} and \eqref{eq:06101086},
    \begin{align}
        |u_i(x) - u_i(y)| &\leq |u_i(x) - u_i(x_i)| + |u_i(x_i) - u_i(y_i)| + |u_i(y_i) - u_i(y)| < 3\epsilon + \epsilon + 3\epsilon = 7\epsilon.
    \end{align}
    This proves (b).
\end{proof}

Let $\tilde{w}_i := \Psi_{\rho_i}^*\hat{u}_i^{(\rho_i)}: \overline{\Omega}^{\rho_0} \times [0,\frac{7}{8}] \to \R$, which satisfies $(\del_t - \Delta_{\Psi_{\rho_i}^*\hat{g}^{(\rho_i)}(t)}) \tilde{w}_i = 0$  (see \S\ref{subsec:pointwise-estimates} and Lemma \ref{lem:conversion-to-parabolic}). The $u_i$ estimates from Proposition \ref{prop:unif-equi-1} translate to uniform equicontinuity estimates for $\tilde{w}_i$:

\begin{corollary} \label{cor:unif-equi-w}
    Define the cylindrical metric $g_C: = dr^2 + g_X$ on $\overline{\Omega}^{\rho_0}$.
    For every $\epsilon > 0$, there exist $\delta_0 > 0$ and $i_0 \in \N$ such that for all $i \geq i_0$, $x,y \in \overline{\Omega}^{\rho_0}$, and $s,t \in [0,\frac{7}{8}]$,
    \begin{enumerate}[label=(\alph*)]
        \item If $d_{g_C}(x,y) < \delta_0$, then $|\tilde{w}_i(x,t) - \tilde{w}_i(y,t)| < \epsilon$.
        \item If $|s-t| < \delta_0$, then $|\tilde{w}_i(x,s) - \tilde{w}_i(x,t)| < \epsilon$.
    \end{enumerate}
    In particular, the functions $\tilde{w}_i: \overline{\Omega}^{\rho_0} \times [0,\frac{7}{8}] \to \R$ are uniformly equicontinuous.
\end{corollary}
\begin{proof}
    Given $\epsilon > 0$, let $\delta > 0$ and $i_0 \in \N$ be given by Proposition \ref{prop:unif-equi-1}. Let $x,y \in \overline{\Omega}^{\rho_0}$.
    To prove (a), we divide into two cases.
    \begin{itemize}
        \item \underline{Case 1: $t \in [0,\frac{\delta}{2}]$.} Then for sufficiently large $i$ (depending on $\delta$), Lemma \ref{lem:phi-bounds} gives $$\Phi_{\rho_i t}(\Psi_{\rho_i}(x)), \Phi_{\rho_i t}(\Psi_{\rho_i}(y)) \in \{(1-\delta)\rho_i \leq r \leq \rho_i\}.$$
        Let $z \in \{(1-\delta)\rho_i \leq r \leq \rho_i\}$ have the same $r$-coordinate as $\Phi_{\rho_i t}(\Psi_{\rho_i}(x))$ and the same $\theta$-coordinate as $\Phi_{\rho_i t}(\Psi_{\rho_i}(y))$. If $d_{g_C}(x,y) < \delta$, then $d_{g_X}(\Phi_{\rho_i t}(\Psi_{\rho_i}(x)), z) < \delta$ as well. So Proposition \ref{prop:unif-equi-1} gives
        \begin{align}
            |\tilde{w}_i(x,t) - \tilde{w}_i(y,t)| &= |u_i(\Phi_{\rho_i t}(\Psi_{\rho_i}(x))) - u_i(\Phi_{\rho_i t}(\Psi_{\rho_i}(y)))| \\
            &\leq |u_i(\Phi_{\rho_i t}(\Psi_{\rho_i}(x))) - u_i(z)| + |u_i(z) - u_i(\Phi_{\rho_i t}(\Psi_{\rho_i}(y)))| < 2\epsilon.
        \end{align}
        \item \underline{Case 2: $t \in [\frac{\delta}{2}, \frac{7}{8}]$.} Then for sufficiently large $i$ (depending on $\delta$), Lemma \ref{lem:phi-bounds} gives
        $$\Phi_{\rho_i t}(\Psi_{\rho_i}(x)), \Phi_{\rho_i t}(\Psi_{\rho_i}(y)) \in \left\{ \tfrac{1}{16} \rho_i \leq r \leq \left( 1 - \tfrac{\delta}{4} \right)\rho_i \right\}.$$
        As $u_i = \Theta_i$ on $\{r=\rho_i\}$, Lemma \ref{lem:equi-c2alpha} followed by Theorem \ref{thm:scaled-back-schauder} gives $C = C(\delta)$ such that
        \begin{align} \label{eq:ui-grad-bds}
            \sup_{\{\frac{1}{16}\rho_i \leq r \leq \left(1-\frac{\delta}{4}\right)\rho_i\}} \left( \sqrt{\rho_i} |\nabla u_i| + \rho_i \left| \inner{\nabla u_i}{\nabla r} \right| \right) \leq C \quad \text{for each } u_i.
        \end{align}
        Let $z \in \{\frac{1}{16}\rho_i \leq r \leq \left(1-\frac{\delta}{4}\right)\rho_i\}$ have the same $r$-coordinate as $\Phi_{\rho_i t}(\Psi_{\rho_i}(x))$ and the same $\theta$-coordinate as $\Phi_{\rho_i t}(\Psi_{\rho_i}(y))$. If $d_{g_C}(x,y) < \delta_0$, where $\delta_0 < \delta$ is to be chosen, then $d_g(\Phi_{\rho_i t}(\Psi_{\rho_i}(x)), z) < C\delta_0\sqrt{\rho_i}$ and $d_g(z, \Phi_{\rho_i t}(\Psi_{\rho_i}(y))) < C\delta_0\rho_i$. Together with \eqref{eq:ui-grad-bds}, these imply
        \begin{align}
            |\tilde{w}_i(x,t) - \tilde{w}_i(y,t)| &= |u_i(\Phi_{\rho_i t}(\Psi_{\rho_i}(x))) - u_i(z)| + |u_i(z) - u_i(\Phi_{\rho_i t}(\Psi_{\rho_i}(y)))| < C\delta_0.
        \end{align}
        Now choose $\delta_0 < \delta$ small so that the right-hand side is $< \epsilon$.
    \end{itemize}
    Part (a) follows from these cases.

    For part (b), let $x \in \overline{\Omega}^{\rho_0}$ and $s, t \in [0,\frac{7}{8}]$. If $s,t < \frac{\delta}{2}$ then $\Phi_{\rho_i s}(\Psi_{\rho_i}(x)), \Phi_{\rho_i t}(\Psi_{\rho_i}(x)) \in \{(1-\delta)\rho_i \leq r \leq \rho_i\}$, and these points have the same $\theta$-coordinate. By Proposition \ref{prop:unif-equi-1},
    \begin{align}
        |\tilde{w}_i(x,s) - \tilde{w}_i(x,t)| &= \left| u_i(\Phi_{\rho_i s}(\Psi_{\rho_i}(x))) - u_i(\Phi_{\rho_i t}(\Psi_{\rho_i}(x))) \right| < \epsilon.
    \end{align}
    If $s \geq \frac{\delta}{2}$ or $t \geq \frac{\delta}{2}$, then we may suppose $|s-t| < \delta_1 < \frac{\delta}{4}$ ($\delta_1 > 0$ to be chosen) so that both $s, t \geq \frac{\delta}{4}$. In that case $\Phi_{\rho_i s}(\Psi_{\rho_i}(x)), \Phi_{\rho_i t}(\Psi_{\rho_i}(x)) \in \{\frac{1}{16}\rho_i \leq r \leq (1-\frac{\delta}{8})\rho_i\}$. The two points also have the same $\theta$-coordinate, and are at most $g$-distance $C\delta_1 \rho_i$ apart, so by \eqref{eq:ui-grad-bds} we have
    \begin{align}
        |\tilde{w}_i(x,s) - \tilde{w}_i(x,t)| &= \left| u_i(\Phi_{\rho_i s}(\Psi_{\rho_i}(x))) - u_i(\Phi_{\rho_i t}(\Psi_{\rho_i}(x))) \right| \leq C\delta_1 \rho_i \cdot \sup |\inner{\nabla u_i}{\nabla r}| \leq C\delta_1,
    \end{align}
    which is $<\epsilon$ if $\delta_1$ is appropriately small.
\end{proof}

The uniform equicontinuity estimates for $\tilde{w}_i$ from Corollary \ref{cor:unif-equi-w} will activate a convergence argument, with the limit satisfying a parabolic version of \eqref{eq:initial-3circ}. Scaling back leads to Proposition \ref{prop:eigenfn-to-bounds}:

\begin{proof}[Proof of Proposition \ref{prop:eigenfn-to-bounds}]
    Fix any $p_0 \in M$ and define $w_i := \tilde{w}_i - u_i(p_0): \overline{\Omega}^{\rho_0} \times [0,\frac{7}{8}] \to \R$, which satisfies $(\del_t - \Delta_{\Psi_{\rho_i}^*\hat{g}^{(\rho_i)}(t)})w_i = 0$. By the maximum principle and Lemma \ref{lem:equi-c2alpha}, we have $|u_i(p_0)| \leq \sup_i|\Theta_i| \leq C < \infty$ and $|w_i| \leq \sup_{\overline{B}_{\rho_i}} |u_i| + |u_i(p_0)| \leq 2 \sup_i|\Theta_i| \leq C < \infty$. Then $u_i(p_0) \to c \in [-C,C]$ along some subsequence.

    By Theorem \ref{thm:parabolicSchauder} and the uniform bound on $w_i$, for each $\tau \in (0,\frac{1}{2})$ there exists $C = C(\tau)$ such that
    \begin{align}
        \norm{w_i}_{C^{2+\alpha,1+\frac{\alpha}{2}}(\overline{\Omega}^{\rho_0}_{\tau} \times [\tau,\frac{7}{8}]; \Psi_{\rho_0}^*\hat{g}^{(\rho_0)}(0))} &\leq C \norm{w_i}_{L^\infty(\overline{\Omega}^{\rho_0} \times [0,\frac{7}{8}])} \leq C.
    \end{align}
    By Corollary \ref{cor:unif-equi-w}, the $w_i$ are also uniformly equicontinuous on $\overline{\Omega}^{\rho_0} \times [0,\frac{7}{8}]$.
    By the Arzel\`a--Ascoli theorem, a subsequence of $w_i$'s converge uniformly in $\overline{\Omega}^{\rho_0} \times [0,\frac{7}{8}]$, and in $C^{2,1}$ on compact sets of $\Omega^{\rho_0} \times (0,\frac{7}{8}]$, to a limiting function $w_\infty$. Moreover, since $w_i + u_i(p_0)$ restricts to $\Theta_i$ on $\{r=\rho_0\} \times \{0\} \subset \overline{\Omega}^{\rho_0} \times [0,\frac{7}{8}]$, and these restrictions are uniformly $C^{2,\alpha}(\Sigma)$-bounded by Lemma \ref{lem:equi-c2alpha}, we may take a further subsequence so that $(w_i+u_i(p_0))|_{\{r=\rho_0\} \times \{0\}} \to \Theta$ in $C^2(\Sigma)$, where $-\Delta_{g_X}\Theta = \lambda\Theta$. Thus,
    \begin{align} \label{eq:bdry-wi-lim}
        w_i|_{\{r=\rho_0\} \times \{0\}} \to \Theta - c \quad \text{in } C^2(\Sigma).
    \end{align}
    As in Lemma \ref{lem:limit-claims}, there is a continuous function $\omega_\infty: \Sigma \times [0,\frac{7}{8}] \to \R$ given by $\omega_\infty(\theta,t) = w_\infty(r,\theta,t)$ for any $r \in (\rho_0-\sqrt{\rho_0}, \rho_0)$. Moreover, $\omega_\infty$ satisfies
    \begin{align}
        (\del_t - \Delta_{(1-t)g_X}) \omega_\infty = 0 \quad \text{on } \Sigma \times (0,\tfrac{7}{8}],
    \end{align}
    and
    \begin{align} \label{eq:3circ-initial-2}
        \lim_{i\to\infty} I_{u_i-u_i(p_0)}\left(\frac{\rho_i}{2}\right) = \int_\Sigma \omega_\infty(\cdot,1/2)^2 \, \dvol_{g_X}.
    \end{align}
    From \eqref{eq:bdry-wi-lim}, we have $\omega_\infty(\theta,0) = \Theta(\theta) - c$. By Lemma \ref{lem:classification-of-sols}, it follows that $\omega_\infty(\theta,t) = (1-t)^{\lambda} \Theta(\theta) - c$, and $\Theta$ and $c$ are $L^2(g_X)$-orthogonal so
    \begin{align} \label{eq:3circ-initial-1}
        \frac{\int_\Sigma \omega_\infty(\cdot,0)^2 \, \dvol_{g_X}}{\int_\Sigma \omega_\infty(\cdot,1/2)^2 \, \dvol_{g_X}} = \frac{1 + c^2}{(1/2)^{2\lambda} + c^2} \leq \frac{1}{(1/2)^{2\lambda}} = 2^{2\lambda}.
    \end{align}
    Also from \eqref{eq:bdry-wi-lim},
    \begin{align} \label{eq:3circ-initial-3}
        \int_\Sigma \omega_\infty(\cdot,0)^2 \, \dvol_{g_X} = \lim_{i\to\infty} \int_\Sigma (\Theta_i^2+c^2) \, \dvol_{g_X(\rho_i)} = \lim_{i\to\infty} I_{u_i-u_i(p_0)}(\rho_i).
    \end{align}
    Writing $v_i = u_i - u_i(p_0)$, it follows by \eqref{eq:3circ-initial-2}, \eqref{eq:3circ-initial-1}  and \eqref{eq:3circ-initial-3} that
    \begin{align} \label{eq:lim-i-ratio}
        \lim_{i\to\infty} \frac{I_{v_i}(\rho_i)}{I_{v_i}(\rho_i/2)} \leq 2^{2\lambda},
    \end{align}
    which yields the proposition.
\end{proof}

\subsection{Proof of Theorem \ref{thm:construction}} \label{subsec:proof-Al}

Theorem \ref{thm:construction} is proved by using Propositions \ref{prop:eigenfn-to-bounds} and \ref{prop:I-to-convergence} to construct an appropriate collection of drift-harmonic functions $\mathring{\B}_{\lambda_{\ell+1}}$, with intermediate steps to make sure the conditions of Definition \ref{def:E-condition} are met. One of these steps involves taking several drift-harmonic functions and renormalizing so that they are asymptotically orthogonal. This is facilitated by the next lemma.

\begin{lemma} \label{lem:orthog-at-infty}
    Let $u,v \in \mathring{\mathcal{S}}_{\lambda_{\ell+1}}(C,\tau)$ be linearly independent. Then there exists $L \in \R$ such that up to increasing $C$,
    \begin{align}
        \left| \frac{\inner{u}{v}_\rho}{\norm{v}_\rho^2} - L \right| \leq C\rho^{-\tau} \quad \text{for all } \rho > 0.
    \end{align}
\end{lemma}
\begin{proof}
    As $u \in \mathring{\mathcal{S}}_{\lambda_{\ell+1}}(C,\tau)$, it $(C,\tau)$-asymptotically separates variables by definition, and
    \begin{align}
        \lambda_{\ell+1} - C\rho^{-\tau} \leq U_u(\rho) \leq \lambda_{\ell+1} + C\rho^{-\tau}.
    \end{align}
    By Corollary \ref{cor:I-almost-mono}, it follows that
    \begin{gather}
        e^{-C\rho^{-\tau}}  \left( \frac{\rho_2}{\rho_1} \right)^{2\lambda_{\ell+1}} \leq \frac{I_u(\rho_2)}{I_u(\rho_1)} \leq e^{C\rho^{-\tau}} \left( \frac{\rho_2}{\rho_1} \right)^{2\lambda_{\ell+1}}, \quad \text{for all } \rho_2 > \rho_1 \geq 1, \\
        C^{-1} \rho^{2\lambda_{\ell+1}} \leq I_u(\rho) \leq C\rho^{2\lambda_{\ell+1}} \quad \text{for all } \rho \geq 1.
    \end{gather}
    The same statements hold for $v$. Thus
    \begin{gather}
        e^{-C\rho^{-\tau}} \leq \sqrt{\frac{I_u(\rho_1)}{I_u(\rho_2)}} \sqrt{\frac{I_v(\rho_2)}{I_v(\rho_1)}} \leq e^{C\rho^{-\tau}} \quad \text{for all } \rho_2 > \rho_1 \geq 1, \label{eq:66100} \\
        \frac{I_u(\rho)}{I_v(\rho)} \leq C \quad \text{for all } \rho \geq 1. \label{eq:66101}
    \end{gather}
    Write $d = \max_{\rho \geq 1} U_v(\rho)$.
    For each $\rho \geq 1$ and $s \in \N$, we use Corollary \ref{cor:pres-of-inner-prod} and \eqref{eq:66100} to get
    \begin{align}
        \left| \frac{\inner{u}{v}_{2^s\rho}}{\norm{u}_{2^s\rho} \norm{v}_{2^s\rho}} - \frac{\inner{u}{v}_{2^{s-1}\rho}}{\norm{u}_{2^{s-1}\rho} \norm{v}_{2^{s-1}\rho}} \right| &\leq \left| \frac{\inner{u}{v}_{2^s\rho}}{\norm{u}_{2^s\rho} \norm{v}_{2^s\rho}} - \sqrt{\frac{I_u(2^{s-1}\rho)}{I_u(2^s\rho)}} \sqrt{\frac{I_v(2^s\rho)}{I_v(2^{s-1}\rho)}} \frac{\inner{u}{v}_{2^{s-1}\rho}}{\norm{u}_{2^{s-1}\rho} \norm{v}_{2^{s-1}\rho}} \right| \\
        &\quad + \left| \sqrt{\frac{I_u(2^{s-1}\rho)}{I_u(2^s\rho)}} \sqrt{\frac{I_v(2^s\rho)}{I_v(2^{s-1}\rho)}} - 1 \right| \frac{\left|\inner{u}{v}_{2^{s-1}\rho} \right|}{\norm{u}_{2^{s-1}\rho} \norm{v}_{2^{s-1}\rho}} \\
        &\leq Ce^{C(2^{s-1}\rho)^{2\tau-\mu}} (2^{s-1}\rho)^{-\tau} 2^{4d+1} + C(2^{s-1}\rho)^{-\tau} \leq C(2^{s-1}\rho)^{-\tau},
    \end{align}
    where $C$ is independent of $s$. Hence for each $\rho \geq 1$ and $q \in \N$, we have
    \begin{equation}
        \left| \frac{\inner{u}{v}_{2^q\rho}}{\norm{u}_{2^q\rho} \norm{v}_{2^q\rho}} - \frac{\inner{u}{v}_\rho}{\norm{u}_\rho \norm{v}_\rho} \right| \leq \sum_{s=1}^\infty \left| \frac{\inner{u}{v}_{2^s\rho}}{\norm{u}_{2^s\rho} \norm{v}_{2^s\rho}} - \frac{\inner{u}{v}_{2^{s-1}\rho}}{\norm{u}_{2^{s-1}\rho} \norm{v}_{2^{s-1}\rho}} \right| \leq C\rho^{-\tau} \sum_{s=1}^\infty (2^{s-1})^{-\tau} \leq C\rho^{-\tau}. \label{eq:lim1}
    \end{equation}
    Also, by \eqref{eq:66100} and \eqref{eq:66101},
    \begin{align}
        \left| \frac{I_u(2^q\rho)}{I_v(2^q\rho)} - \frac{I_u(\rho)}{I_v(\rho)} \right| = \frac{I_u(\rho)}{I_v(\rho)} \left| \frac{I_u(2^q\rho)}{I_v(2^q\rho)} \frac{I_v(\rho)}{I_u(\rho)} - 1 \right| \leq C \rho^{-\tau}. \label{eq:lim2}
    \end{align}
    It follows from \eqref{eq:lim1} and \eqref{eq:lim2} that the following limits $L_1$ and $L_2$ exist:
    \begin{align}
        \left| \frac{\inner{u}{v}_\rho}{\norm{u}_\rho \norm{v}_\rho} - L_1 \right| &\leq C\rho^{-\tau}, \quad \left| \frac{\norm{u}_\rho}{\norm{v}_\rho} - L_2 \right| \leq C\rho^{-\tau}.
    \end{align}
    The lemma therefore follows with $L = L_1 L_2$.
\end{proof}

We are finally in the position to prove Theorem \ref{thm:construction}:
\begin{proof}[Proof of Theorem \ref{thm:construction}]
    Assume that $(E_\ell)$ holds. This gives collections $\mathring{\B}_{\lambda_1}, \ldots, \mathring{\B}_{\lambda_\ell} \subset \H$ so that for each $j \in \{1,2,\ldots,\ell\}$, items (a)--(d) of Definition \ref{def:E-condition} hold. Let $p_0 \in M$ be the point at which $v(p_0) = 0$ for all $v \in \mathring{\B}_{\lambda_j}$, $j \geq 2$. We set $\B_{\lambda_\ell} = \bigcup_{j=1}^\ell \mathring{\B}_{\lambda_j}$.
    In this proof, $C$ and $\tau$ denote arbitrary positive constants, with $C$ increasing and $\tau$ decreasing freely from expression to expression.

    \vspace*{1em}
    \noindent\textbf{Step 1: constructing the first drift-harmonic function.}

    Let $\Theta: \Sigma \to \R$ be an eigenfunction $-\Delta_{g_X} \Theta = \lambda_{\ell+1} \Theta$. For each $i \in \N$, let $u_i$ be the solution to
    \begin{align}
        \begin{cases}
            \L_f u_i = \Theta &\text{in } B_{2^i}, \\
            u_i = \Theta &\text{on } \{r=2^i\}.
        \end{cases}
    \end{align}
    Let $w_i = u_i - u_i(p_0)$. The functions $w_i$ satisfy the following:
    \begin{enumerate}
        \item $w_i$ is linearly independent from $\B_{\lambda_\ell} \setminus \{1\}$ for all large $i$. Indeed, since $w_i$ on $\{r=2^i\}$ is an $\lambda_{\ell+1}$-eigenfunction of $-\Delta_{g_X}$ plus a constant, we have
        \begin{align} \label{eq:proj1}
            \P_{2^i,k}w_i = 0 \quad \text{for all } k \notin \{1,\ell+1\}.
        \end{align}
        and so
        \begin{align} \label{eq:10680186}
            \norm{w_i}_{2^i}'^2 = \norm{\P_{2^i,1}w_i}_{2^i}'^2 + \norm{\P_{2^i,\ell+1}w_i}_{2^i}'^2.
        \end{align}
        By Corollary \ref{cor:B-lin-combs-proj}, if $\phi$ is nonzero and is in the span of $\B_{\lambda_\ell} \setminus \{1\}$, then
        \begin{align}
            \frac{\norm{\P_{2^i,1}\phi}_{2^i}'^2}{\norm{\phi}_{2^i}'^2} \leq C(2^i)^{-2\tau} \quad \text{and} \quad \frac{\norm{\P_{2^i,\ell+1}\phi}_{2^i}'^2}{\norm{\phi}_{2^i}'^2} \leq C(2^i)^{-2\tau}.
        \end{align}
        This and \eqref{eq:10680186} show that $\phi \neq w_i$ for all large $i$.
        Thus $w_i$ is linearly independent from $\B_{\lambda_\ell} \setminus \{1\}$.
        \item $w_i$ and each $v \in \B_{\lambda_\ell} \setminus \{1\}$ are $C(2^i)^{-\tau}$-almost orthogonal on $\{r=2^i\}$. Indeed, since $v \in \mathring{\B}_{\lambda_j} \subset \mathring{\mathcal{S}}_{\lambda_j}(C,\tau)$ for some $j \in \{2,3,\ldots,\ell\}$, we have $\frac{\norm{\P_{2^i,j}v}_{2^i}'}{\norm{v}_{2^i}'} \geq 1-C(2^i)^{-\tau}$ and hence
        \begin{align} \label{eq:proj2}
            \frac{\norm{\P_{2^i,k}v}_{2^i}'}{\norm{v}_{2^i}'} \leq C(2^i)^{-\tau} \quad \text{for all } k \in \{1,\ell+1\}.
        \end{align}
        Then by \eqref{eq:proj1}, \eqref{eq:proj2} and the Cauchy--Schwarz inequality,
        \begin{align}
            \frac{\left|\inner{w_i}{v}_{2^i}'\right|}{\norm{w_i}_{2^i}' \norm{v}_{2^i}'} \leq \sum_{k=0}^\infty \frac{\left|\inner{\P_{2^i,k}w_i}{\P_{2^i,k}v}_{2^i}'\right|}{\norm{w_i}_{2^i}' \norm{v}_{2^i}'} \leq \sum_{k \in \{1,\ell+1\}} \frac{\norm{\P_{2^i,k}v}_{2^i}'}{\norm{v}_{2^i}'} \leq C(2^i)^{-\tau}.
        \end{align}
        The claim follows from this and \eqref{eq:inners-close}.
        \item By Proposition \ref{prop:eigenfn-to-bounds}, some subsequence of $w_i$ (which we still call $w_i$) has the following property: for each $\epsilon > 0$, it holds for all sufficiently large $i$ in the subsequence that $I_{w_i}(2^i) \leq 2^{2(\lambda_{\ell+1}+\epsilon)} I_{w_i}(2^{i-1})$.
    \end{enumerate}
    Fix a large $\bar{\rho}$, and choose coefficients $a_{w_i,v} \in \R$ such that the function
    \begin{align} \label{eq:tildev-i}
        \tilde{w}_i := w_i - \sum_{v \in \B_{\lambda_\ell} \setminus \{1\}} a_{w_i,v} v,
    \end{align}
    defined on $\overline{B}_{2^i}$, is orthogonal to each $v \in \B_{\lambda_{\ell}} \setminus \{1\}$ on $\{r=\bar{\rho}\}$.
    Note that
    \begin{itemize}
        \item $\tilde{w}_i(p_0) = 0$ (as each $v \in \B_{\lambda_\ell} \setminus \{1\}$ has $v(p_0) = 0$).
        \item $\tilde{w}_i$ is nonzero, by property (1) above and unique continuation.
    \end{itemize}
    At this point, we also recall some properties of $\B_{\lambda_\ell} \setminus \{1\}$ which are due to $(E_\ell)$ being true:
    \begin{enumerate}[label=(\roman*)]
        \item $\B_{\lambda_\ell} \setminus \{1\}$ is linearly independent, and each distinct pair of functions in this set is $(C,\tau)$-asymptotically orthogonal.
        \item For each $v \in \B_{\lambda_\ell} \setminus \{1\}$, we have that $v$ $(C,\tau)$-asymptotically separates variables.
        \item Each $v \in \B_{\lambda_\ell} \setminus \{1\}$ has $U_v(\rho) \leq \lambda_\ell + C\rho^{-\tau}$, so Corollary \ref{cor:I-almost-mono} gives $\frac{I_v(2^i)}{I_v(2^{i-1})} \leq 2^{2\lambda_\ell+1}$ for large $i$.
        \item The number $\tilde{d} := \max_{v \in \B_{\lambda_\ell} \setminus \{1\}} \max_{\rho \geq 1} U_v(\rho)$ is finite.
    \end{enumerate}
    From (i) and (2) above, the functions in $(\B_{\lambda_\ell} \setminus \{1\}) \cup \{w_i\}$ are $C(2^i)^{-\tau}$-almost orthogonal on $\{r=2^i\}$. Then from \eqref{eq:tildev-i},
    \begin{align} \label{eq:I2i}
        I_{\tilde{w}_i}(2^i) \leq (1+C(2^i)^{-\tau}) \left(I_{w_i}(2^i) + \sum_{v \in \B_{\lambda_\ell} \setminus \{1\}} a_{w_i,v}^2 I_v(2^i) \right).
    \end{align}
    Thanks to (ii), we can apply preservation of almost orthogonality (Corollary \ref{cor:pres-of-inner-prod}), and insert (iii), (iv) and (3) to get that for all large $i$,
    \begin{align}
        \frac{\left| \inner{w_i}{v}_{2^{i-1}} \right|}{\norm{w_i}_{2^{i-1}} \norm{v}_{2^{i-1}}} &\leq \sqrt{\frac{I_{w_i}(2^i)}{I_{w_i}(2^{i-1})}} \sqrt{\frac{I_v(2^{i-1})}{I_v(2^i)}} \left( Ce^{C(2^{i-1})^{2\tau-\mu}} (2^{i-1})^{-\tau} 2^{4\tilde{d}+1} + \frac{\left| \inner{w_i}{v}_{2^i} \right|}{\norm{w_i}_{2^i} \norm{v}_{2^i}} \right) \\
        &\leq 2^{\lambda_{\ell+1}+1} 2^{-\lambda_\ell} C(2^{i-1})^{-\tau} = C(2^{i-1})^{-\tau}. \label{eq:almost-orthog-2i-1}
    \end{align}
    By \eqref{eq:almost-orthog-2i-1} and (i), the functions in $(\B_{\lambda_\ell} \setminus \{1\}) \cup \{w_i\}$ are $C(2^{i-1})^{-\tau}$-almost orthogonal on $\{r=2^{i-1}\}$. Reasoning similarly to \eqref{eq:I2i}, one has
    \begin{align} \label{eq:I2i-1}
        I_{\tilde{w}_i}(2^{i-1}) \geq (1-C(2^{i-1})^{-\tau}) \left( I_{w_i}(2^{i-1}) + \sum_{v \in \B_{\lambda_\ell} \setminus \{1\}} a_{w_i,v}^2 I_v(2^{i-1}) \right).
    \end{align}
    Choose $\epsilon > 0$ with $\lambda_{\ell+1}+\epsilon < \lambda_{\ell+2}$. From (3) above, it holds for large $i$ (depending on $\epsilon$) that $I_{w_i}(2^i) \leq 2^{2(\lambda_{\ell+1}+\frac{\epsilon}{2})} I_{w_i}(2^{i-1})$. Also, for each $v \in \B_{\lambda_\ell} \setminus \{1\}$, we have $I_v(2^i) \leq 2^{2(\lambda_{\ell+1}+\frac{\epsilon}{2})} I_v(2^{i-1})$ by property (iii) and Corollary \ref{cor:I-almost-mono}. Combining these facts with \eqref{eq:I2i} and \eqref{eq:I2i-1}, we see that for all large $i$ depending on $\epsilon$,
    \begin{align}
        I_{\tilde{w}_i}(2^i) &\leq 2^{2(\lambda_{\ell+1}+\frac{\epsilon}{2})} (1+C(2^i)^{-\tau}) \left( I_{w_i}(2^{i-1}) + \sum_{v \in \B_{\lambda_\ell} \setminus \{1\}} a_{w_i,v}^2 I_v(2^{i-1}) \right) \\
        &\leq 2^{2(\lambda_{\ell+1}+\frac{\epsilon}{2})} \frac{1+C(2^i)^{-\tau}}{1-C(2^{i-1})^{-\tau}} I_{\tilde{w}_i}(2^{i-1}) \leq 2^{2(\lambda_{\ell+1}+\epsilon)} I_{\tilde{w}_i}(2^{i-1}).
    \end{align}
    Then by Proposition \ref{prop:I-to-convergence}, there exists $\bar{\rho}_1 > 0$ such that the normalized functions
    \begin{align}
        \hat{w}_i := \frac{\tilde{w}_i}{\sqrt{I_{\tilde{w}_i}(\bar{\rho}_1)}}
    \end{align}
    converge uniformly on compact subsets of $M$ to a nonzero limit $w^{(1)} \in \H$.
    Note that:
    \begin{itemize}
        \item $w^{(1)} \in \H_{\lambda_{\ell+1}}^+$, $w^{(1)}(p_0) = 0$, and $w^{(1)}$ is a nonzero function. The first and third assertions are directly from Proposition \ref{prop:I-to-convergence} and the second is because $\hat{w}_i(p_0) = 0$ for each $i$.
        \item $w^{(1)}$ is not spanned by $\B_{\lambda_\ell}$. Indeed, suppose $w^{(1)}$ is a linear combination of $\B_{\lambda_\ell}$. Since each $v \in \B_{\lambda_\ell} \setminus \{1\}$ has $v(p_0) = 0$, and we have $w^{(1)}(p_0) = 0$, the coefficient of the constant function $1$ must be zero. So $w^{(1)}$ is a linear combination of $\B_{\lambda_\ell} \setminus \{1\}$. However, $\hat{w}_i$ is orthogonal to each $v \in \B_{\lambda_\ell} \setminus \{1\}$ on $\{r=\bar{\rho}\}$, so the same is true of $w^{(1)}$. As $w^{(1)}$ is nonzero on $\{r=\bar{\rho}\}$, it cannot be a linear combination of $\B_{\lambda_\ell} \setminus \{1\}$. Contradiction.
    \end{itemize}
    Set $u^{(1)} = w^{(1)}$.
    By Theorem \ref{thm:asymp-ctrl}, $u^{(1)} \in \mathring{\mathcal{S}}_{\lambda_{\ell+1}}$. In particular, there exist $C,\tau>0$ such that
    \begin{enumerate}[label=(\alph*)]
        \item $u^{(1)} \notin \operatorname{span}(\B_{\lambda_\ell})$.
        \item $u^{(1)}$ is $(C,\tau)$-asymptotically orthogonal to each $v \in \B_{\lambda_\ell}$.
        \item $u^{(1)}$ $(C,\tau)$-asymptotically separates variables.
        \item $U_{u^{(1)}}(\rho) \leq \lambda_{\ell+1} + C\rho^{-\tau}$ and so by Corollary \ref{cor:I-almost-mono}, $\frac{I_{u^{(1)}}(2^i)}{I_{u^{(1)}}(2^{i-1})} \leq 2^{2\lambda_{\ell+1}+1}$ for all large $i$.
        \item $\frac{\norm{\P_{\rho,\ell+1}u^{(1)}}_\rho'}{\norm{u^{(1)}}_\rho'} \geq 1-C\rho^{-\tau}$.
    \end{enumerate}

    \vspace*{1em}
    \noindent\textbf{Step 2: constructing the second drift-harmonic function.}

    Suppose $m_{\ell+1} \geq 2$, i.e. $\lambda_{\ell+1}$ is a repeated eigenvalue of $-\Delta_{g_X}$. Then for each $i \in \N$, take $\Theta_i: \Sigma \to \R$ to be an eigenfunction $-\Delta_{g_X} \Theta_i = \lambda_{\ell+1} \Theta_i$ such that $\norm{\Theta_i}_{L^2(g_X)} = 1$ and
    \begin{align} \label{eq:select-theta-i}
        \left|\inner{u^{(1)}|_{r=2^i}}{\Theta_i}'\right| \leq C(2^i)^{-\tau}.
    \end{align}
    This is possible in view of property (e) of $u^{(1)}$ above. Let $u_i$ be the solution to
    \begin{align}
        \begin{cases}
            \L_f u_i = 0 & \text{in } B_{2^i}, \\
            u_i = \Theta_i & \text{on } \{r=2^i\}.
        \end{cases}
    \end{align}
    Let $w_i = u_i - u_i(p_0)$. The functions $w_i$ satisfy the following:
    \begin{enumerate}
        \item $w_i$ is linearly independent from $(\B_{\lambda_\ell} \setminus \{1\}) \cup \{u^{(1)}\}$. Indeed, if some linear combination vanishes, then $aw_i + bu^{(1)} = \phi$ for some $a,b \in \R$ and $\phi \in \operatorname{span}(\B_{\lambda_\ell} \setminus \{1\})$. The left-hand side is an $\lambda_{\ell+1}$-eigenfunction of $-\Delta_{g_X}$ plus a constant, so we apply similar arguments as earlier to show that the linear combination must be trivial.
        \item $w_i$ and each $v \in (\B_{\lambda_\ell} \setminus \{1\}) \cup \{u^{(1)}\}$ are $C(2^i)^{-\tau}$-almost orthogonal on $\{r=2^i\}$. For $v \in \B_{\lambda_\ell} \setminus \{1\}$, this is justified as earlier. For $v = u^{(1)}$, this is \eqref{eq:select-theta-i}.
        \item By Proposition \ref{prop:eigenfn-to-bounds}, some subsequence of $w_i$ (which we still call $w_i$) has the following property: for each $\epsilon > 0$, it holds for all sufficiently large $i$ in the subsequence that $I_{w_i}(2^i) \leq 2^{2(\lambda_{\ell+1}+\epsilon)} I_{w_i}(2^{i-1})$.
    \end{enumerate}
    Fix a large $\bar{\rho}$, and choose coefficients $a_{w_i,v} \in \R$ such that the function
    \begin{align}
        \tilde{w}_i := w_i - \sum_{v \in (\B_{\lambda_\ell} \setminus \{1\}) \cup \{u^{(1)}\}} a_{w_i,v} v,
    \end{align}
    defined on $\overline{B}_{2^i}$, is orthogonal to each $v \in (\B_{\lambda_\ell} \setminus \{1\}) \cup \{u^{(1)}\}$ on $\{r=\bar{\rho}\}$. Note that
    \begin{itemize}
        \item $\tilde{w}_i(p_0) = 0$.
        \item $\tilde{w}_i$ is nonzero.
    \end{itemize}
    We also recall some properties of $(\B_{\lambda_\ell} \setminus \{1\}) \cup \{u^{(1)}\}$ which are due to $(E_\ell)$ being true and properties (a)--(d) for $u^{(1)}$ above:
    \begin{enumerate}[label=(\roman*)]
        \item $(\B_{\lambda_\ell} \setminus \{1\}) \cup \{u^{(1)}\}$ is linearly independent, and each pair of functions in this set is $(C,\tau)$-asymptotically orthogonal.
        \item For each $v \in (\B_{\lambda_\ell} \setminus \{1\}) \cup \{u^{(1)}\}$, we have that $v$ $(C,\tau)$-asymptotically separates variables.
        \item Each $v \in (\B_{\lambda_\ell} \setminus \{1\}) \cup \{u^{(1)}\}$ has $\frac{I_v(2^i)}{I_v(2^{i-1})} \leq 2^{2\lambda_{\ell+1}+1}$.
        \item The number $\tilde{d} := \max_{v \in (\B_{\lambda_\ell} \setminus \{1\}) \cup \{u^{(1)}\}} \max_{\rho \geq 1} U_v(\rho)$ is finite.
    \end{enumerate}
    Arguing as in Step 1, we see that for each $\epsilon > 0$, it holds for all large $i$ depending on $\epsilon$ that
    \begin{align}
        I_{\tilde{w}_i}(2^i) \leq 2^{2(\lambda_{\ell+1}+\epsilon)} I_{\tilde{w}_i}(2^{i-1}).
    \end{align}
    Then we invoke Proposition \ref{prop:I-to-convergence} to obtain a limit $w^{(2)} \in \H$ satisfying
    \begin{itemize}
        \item $w^{(2)} \in \H_{\lambda_{\ell+1}}^+$, $w^{(2)}(p_0) = 0$, and $w^{(2)}$ is a nonzero function.
        \item $w^{(2)}$ is not spanned by $\B_{\lambda_\ell} \cup \{u^{(1)}\}$.
    \end{itemize}

    \vspace*{1em}
    \noindent\textbf{Step 2.5: asymptotic orthogonalization.}

    By Lemma \ref{lem:orthog-at-infty}, there exists $L \in \R$ such that
    \begin{align}
        \left| \frac{\inner{w^{(2)}}{u^{(1)}}_\rho}{\norm{u^{(1)}}_\rho^2} - L \right| \leq C\rho^{-\tau} \quad \text{for all } \rho > 0.
    \end{align}
    Set $u^{(2)} := w^{(2)} - Lu^{(1)}$. The two properties listed for $w^{(2)}$ above are easily seen to also apply for $u^{(2)}$. By Theorem \ref{thm:asymp-ctrl}, $u^{(2)} \in \mathring{\mathcal{S}}_{\lambda_{\ell+1}}$. Then $\frac{\norm{u^{(1)}}_\rho}{\norm{u^{(2)}}_\rho}$ is bounded (as in the proof of Lemma \ref{lem:orthog-at-infty}), so for all $\rho > 0$,
    \begin{align}
        \left|\frac{\inner{u^{(2)}}{u^{(1)}}_\rho}{\norm{u^{(2)}}_\rho \norm{u^{(1)}}_\rho} \right| &= \left| \frac{\inner{w^{(2)}}{u^{(1)}}_\rho}{\norm{u^{(2)}}_\rho \norm{u^{(1)}}_\rho} - L \frac{\norm{u^{(1)}}_\rho}{\norm{u^{(2)}}_\rho} \right| = \left| \frac{\inner{w^{(2)}}{u^{(1)}}_\rho}{\norm{u^{(1)}}_\rho^2} - L \right| \frac{\norm{u^{(1)}}_\rho}{\norm{u^{(2)}}_\rho} \leq C\rho^{-\tau}.
    \end{align}
    That is,
    \begin{itemize}[label=($\star$)]
        \item $u^{(2)}$ and $u^{(1)}$ are $(C,\tau)$-asymptotically orthogonal.
    \end{itemize}
    Moreover, by virtue of having $u^{(2)} \in \mathring{\mathcal{S}}_{\lambda_{\ell+1}}$, there exist $C,\tau > 0$ such that
    \begin{enumerate}[label=(\alph*)]
        \item $u^{(2)} \notin \operatorname{span}(\B_{\lambda_\ell} \cup \{u^{(1)}\})$.
        \item $u^{(2)}$ is $(C,\tau)$-asymptotically orthogonal to each $v \in \B_{\lambda_\ell} \cup \{u^{(1)}\}$.
        \item $u^{(2)}$ $(C,\tau)$-asymptotically separates variables.
        \item $U_{u^{(2)}}(\rho) \leq \lambda_{\ell+1} + C\rho^{-\tau}$ and so by Corollary \ref{cor:I-almost-mono}, $\frac{I_{u^{(2)}}(2^i)}{I_{u^{(2)}}(2^{i-1})} \leq 2^{2\lambda_{\ell+1}+1}$ for all large $i$.
        \item $\frac{\norm{\P_{\rho,\ell+1}u^{(2)}}_\rho'}{\norm{u^{(2)}}_\rho'} \geq 1-C\rho^{-\tau}$.
    \end{enumerate}

    \vspace*{1em}
    \noindent\textbf{Step 3: constructing the rest of the functions and concluding.}

    If $m_{\ell+1} \geq 3$, repeat Step 2. Namely, choose the boundary eigenfunctions $\Theta_i$ so that $\norm{\Theta_i}_{L^2(g_X)} = 1$ and
    \begin{align}
        \left| \inner{u^{(i)}|_{r=2^i}}{\Theta_i}' \right| \leq C(2^i)^{-\tau} \quad \text{for } i = 1,2.
    \end{align}
    Then Step 2 carries through with straightforward modifications, allowing us to construct $u^{(3)}$. Only two differences are worth noting:
    \begin{itemize}
        \item To justify property (i) for $(\B_{\lambda_\ell} \setminus \{1\}) \cup \{u^{(1)}, u^{(2)}\}$, we additionally use that $(\star)$ holds.
        \item In Step 2.5, we take $u^{(3)} := w^{(3)} - L_1 u^{(1)} - L_2 u^{(2)}$, where $L_1, L_2 \in \R$ are given by Lemma \ref{lem:orthog-at-infty} so that
        \begin{align}
            \left| \frac{\inner{w^{(3)}}{u^{(1)}}_\rho}{\norm{u^{(1)}}_\rho^2} - L_1 \right| \leq C\rho^{-\tau}, \quad \left| \frac{\inner{w^{(3)}}{u^{(2)}}_\rho}{\norm{u^{(2)}}_\rho^2} - L_2 \right| \leq C\rho^{-\tau}, \quad \text{for all } \rho > 0.
        \end{align}
    \end{itemize}

    Similarly, we construct $u^{(4)},\ldots,u^{(m_{\ell+1})}$. Note that the first part of Step 2 (selecting boundary eigenfunctions) fails after $m_{\ell+1}$ drift-harmonic functions have been constructed. At this point we have produced $m_{\ell+1}$ drift-harmonic functions $u^{(1)}, \ldots,u^{(m_{\ell+1})}$ such that for each $i \in \{1,\ldots,m_{\ell+1}\}$,
    \begin{itemize}
        \item $u^{(i)} \in \mathring{\mathcal{S}}_{\lambda_{\ell+1}}$.
        \item $u^{(i)}(p_0) = 0$.
        \item $\B_{\lambda_\ell} \cup \{u^{(1)},\ldots,u^{(m_{\ell+1})}\}$ is linearly independent, and each pair of functions in this set is $(C,\tau)$-asymptotically orthogonal.
    \end{itemize}
    So the set $\mathring{\B}_{\lambda_{\ell+1}} = \{u^{(1)},\ldots,u^{(m_{\ell+1})}\}$ satisfies the conditions in Definition \ref{def:E-condition}. Thus $(E_{\ell+1})$ holds.
\end{proof}

\appendix
\section{The model case of Theorem \ref{thm:main-expanded}} \label{sec:appA}

Let $n \geq 3$, and let $(\Sigma^{n-1},g_X)$ be a closed $(n-1)$-dimensional smooth Riemannian manifold. Let the distinct eigenvalues of $-\Delta_{g_X}$ be $0 = \lambda_1 < \lambda_2 < \cdots \to \infty$ with respective finite multiplicities $1 = m_1, m_2, \cdots$.
Let $\varphi: (0,\infty) \to (0,\infty)$ be a smooth function such that $\varphi(r) = r$ for all small $r$ and $\varphi(r) = \sqrt{r}$ for all large $r$. Then the manifold $(0,\infty) \times \Sigma$ with the Riemannian metric
\begin{align}
    g_P = dr^2 + \varphi(r)^2 g_X
\end{align}
closes up at the origin to give a complete, smooth Riemannian manifold $(P^n,g_P)$. Any point away from the origin can be written as $(r,\theta)$ for some $\theta \in \Sigma$.

Let $f: P \to \R$ be a smooth function such that $f$ is a function of $r$ only, $f$ is constant near the origin, and $f(r,\theta) = -r$ for all large $r$. For $d \in \R$, let
\begin{align}
    \H_d(P) &= \{u \in C^\infty(P) \mid \L_f u = 0 \text{ and } |u| \leq C(r^d+1) \text{ for some } C > 0 \}
\end{align}
be the space of drift-harmonic functions with polynomial growth of degree at most $d$.

The model case of Theorem \ref{thm:main-expanded} is given in the next lemma and proposition. It explicitly determines the spaces $\H_d(P)$, and its proof generalizes the standard classification of entire harmonic functions in $\R^n$.

\begin{lemma} \label{lem:model-ODE-lemma}
    For each $\lambda > 0$, there is a unique solution $R_\lambda: (0,\infty) \to \R$ of the ODE
    \begin{align} \label{eq:model-ODE}
        R''(r) + \left( \frac{(n-1)\varphi'(r)}{\varphi(r)} - f'(r) \right) R'(r) - \frac{\lambda}{\varphi(r)^2} R(r) = 0 \quad \text{for } r \in (0,\infty),
    \end{align}
    which extends continuously to $R_\lambda(0) = 0$ and satisfies $R_\lambda(r) \sim r^{\lambda}$ as $r \to \infty$.
\end{lemma}
\begin{proof}
    Let $A_1(r) = \frac{(n-1)\varphi'(r)}{\varphi(r)} - f'(r)$ and $A_0(r) = -\frac{\lambda}{\varphi(r)^2}$. Then
    \begin{align}
        \lim_{r \to 0^+} r A_1(r) = n-1, \quad \lim_{r \to 0^+} r^2 A_0(r) = -\lambda,
    \end{align}
    so the ODE \eqref{eq:model-ODE} has a regular singular point at $r=0$. By the method of Frobenius, there are two linearly independent solutions $R_0, R_\infty$ of \eqref{eq:model-ODE} such that
    \begin{align}
        R_0(r) \sim r^{\alpha(\lambda)}, \quad R_\infty(r) \sim r^{2-n-\alpha(\lambda)} \quad \text{as } r \to 0^+,
    \end{align}
    where $\alpha(\lambda)$ is the unique $\alpha > 0$ solving the indicial equation $\alpha(\alpha+n-2) - \lambda = 0$.
    
    Since $R_0(r) \sim r^{\alpha(\lambda)}$ as $r \to 0^+$, it follows that $R_0$ extends continuously to $R_0(0) = 0$, and $R_0(r) > 0$ for all small $r$. If $R_0'(r) \leq 0$ for all small $r$, then $R_0(r) \leq 0$ for all small $r$, which is a contradiction. Hence, there exists a sequence $r_i \to 0^+$ such that $R_0'(r_i) > 0$. Fixing any sufficiently large $i$, we have $R_0(r_i) > 0$. If there is a first point $r_* > r_i$ at which $R_0'(r_*) = 0$, then $R_0''(r_*) < 0$ and $R_0(r_*) > 0$. However, \eqref{eq:model-ODE} gives
    \begin{align}
        R_0''(r_*) = \frac{\lambda}{\varphi(r)^2} R_0(r_*) > 0,
    \end{align}
    which is a contradiction. Hence, this first point must not exist, so $R_0'(r) > 0$ for all $r \geq r_i$. As $i$ can be made arbitrarily large, and $r_i \to 0^+$, it follows that $R_0'(r) > 0$ for all $r > 0$.

    We will return to $R_0$ in a moment. In what follows, let $q = \frac{1}{4}A_1^2 + \frac{1}{2}A_1' - A_0$. Then
    \begin{align} \label{eq:ydefn}
        y(r) := \exp\left\{\frac{1}{2} \int_1^r A_1(s) \, ds \right\} R(r)
    \end{align}
    satisfies
    \begin{align} \label{eq:y-ode}
        y'' = qy \quad \text{on } (0,\infty).
    \end{align}
    Using the definitions of $A_1$ and $A_0$, and that $\varphi(r) = \sqrt{r}$ and $f(r) = -r$ for large $r$, we have
    \begin{align} \label{eq:q-asymptotics}
        q(r) = \frac{1}{4} \left( 1 + \frac{n-1+4\lambda}{r} + \O(r^{-2}) \right), \quad q'(r) = \O(r^{-1}), \quad q''(r) = \O(r^{-2}), \quad \text{as } r \to \infty,
    \end{align}
    and $\int_1^\infty q^{1/2}(s) \, ds = \infty$. By the Liouville--Green approximation (see \cite{olver}*{\S 6}, specifically Theorem 2.1, Theorem 3.1, and Section 4.2), there are linearly independent solutions $y_\pm: (0,\infty) \to \R$ of \eqref{eq:y-ode} satisfying
    \begin{align}
        y_\pm(r) \sim q^{-1/4}(r) \exp \left\{ \pm \int_1^r q^{1/2}(s) \, ds \right\} \quad \text{as } r \to \infty.
    \end{align}
    Undoing \eqref{eq:ydefn}, it follows that there are linearly independent solutions $R_\pm: (0,\infty) \to \R$ to \eqref{eq:model-ODE} satisfying
    \begin{align} \label{eq:Rpm-sols}
        R_\pm(r) \sim q^{-1/4}(r) \exp \left\{ \int_1^r \left( \pm q^{1/2}(s) - \frac{1}{2}A_1(s) \right) \, ds \right\} \quad \text{as } r \to \infty.
    \end{align}
    Using \eqref{eq:q-asymptotics} and the Taylor expansion for $\sqrt{1+x}$, it is easy to check that
    \begin{align}
        q^{-1/4}(r) &= \text{constant} + o(1) \quad \text{as } r \to \infty, \label{eq:q14-asymp} \\
        q^{1/2}(s) - \frac{1}{2}A_1(s) &= \frac{\lambda}{s} + \O(s^{-2}) \quad \text{as } s \to \infty, \label{eq:q12-asymp-p} \\
        -q^{1/2}(s) - \frac{1}{2}A_1(s) &=  -1 - \frac{n-1+2\lambda}{2s} + \O(s^{-2}) \quad \text{as } s \to \infty. \label{eq:q12-asymp-m}
    \end{align}
    Inserting \eqref{eq:q14-asymp}, \eqref{eq:q12-asymp-p} and \eqref{eq:q12-asymp-m} into \eqref{eq:Rpm-sols}, the linearly independent solutions $R_\pm$ satisfy (up to scaling)
    \begin{align}
        R_+(r) \sim r^\lambda, \quad R_-(r) \sim e^{-r} r^{-\frac{n-1}{2}-\lambda} \quad \text{as } r \to \infty.
    \end{align}
    Recall from above that $R_0'(r) > 0$ for all $r>0$. Hence $R_0$ cannot decay to zero, so $R_0 \neq R_-$. It follows that $R_0 = aR_+ + bR_-$ for some $a \neq 0$ and $b \in \R$; we have $a>0$ since $R_0(r) > 0$ for all $r > 0$. Letting $R_\lambda = \frac{1}{a} R_0$, it follows that $R_\lambda(0) = 0$ and  $R_\lambda(r) \sim r^{\lambda}$ as $r \to \infty$.
\end{proof}

\begin{proposition} \label{prop:model}
    For each $d \in \R$, there is a basis $\B_d(P)$ for $\H_d(P)$ consisting of separable functions:
    \begin{align}
        \B_d(P) = \bigcup_{\ell=0}^{\left \lfloor{d}\right \rfloor} \left\{ R_{\lambda_\ell}(r) \Theta_\ell^{(k)}(\theta) \in C^\infty(P) \mid k \in \{1,2,\ldots,m_\ell\} \right\},
    \end{align}
    where for each integer $\ell \geq 1$,
    \begin{itemize}
        \item If $\ell =1$ we set $R_{\lambda_1} = 1$. If $\ell \geq 2$, the function $R_{\lambda_\ell}: [0,\infty) \to \R$ is given by Lemma \ref{lem:model-ODE-lemma}, so $R_{\lambda_\ell}(0) = 0$ and $R_{\lambda_\ell}(r) \sim r^{\lambda_\ell}$ as $r \to \infty$.
        \item the set $\{\Theta_\ell^{(1)}, \ldots, \Theta_\ell^{(m_\ell)}\}$ is an $L^2(g_X)$-orthonormal basis for the $\lambda_\ell$-eigenspace of $-\Delta_{g_X}$.
    \end{itemize}
    In particular, the dimension of $\H_d(P)$ is finite with
    \begin{align}
        \dim \H_d(P) = \sum_{\ell=1}^{\left\lfloor{d}\right\rfloor} m_\ell,
    \end{align}
    and any distinct pair $u,v \in \B_d(P)$ satisfies $\int_{\{r=\rho\}} uv = 0$ for every $\rho > 0$.
\end{proposition}
\begin{proof}
    The operator $\L_f u$ separates variables with respect to $(r,\theta)$ coordinates as 
    \begin{align}
        \L_f u = \frac{\del^2 u}{\del r^2} + \left( \frac{(n-1)\varphi'(r)}{\varphi(r)} - f'(r) \right) \frac{\del u}{\del r} + \frac{1}{\varphi(r)^2} \Delta_{g_X} u.
    \end{align}
    Assuming $\L_f u = 0$ with a separation ansatz $u(r,\theta) = R(r)\Theta(\theta)$, we therefore have
    \begin{align}
        R''(r)\Theta(\theta) + \left( \frac{(n-1)\varphi'(r)}{\varphi(r)} - f'(r) \right) R'(r) \Theta(\theta) + \frac{1}{\varphi(r)^2} R(r) \Delta_{g_X}\Theta(\theta) = 0.
    \end{align}
    It follows that $\Theta$ is an eigenfunction of $-\Delta_{g_X}$, say $-\Delta_{g_X}\Theta = \lambda_\ell \Theta$. Then $R$ satisfies the ODE
    \begin{align} \label{eq:model-ODE-2}
        R''(r) + \left( \frac{(n-1)\varphi'(r)}{\varphi(r)} - f'(r) \right) R'(r) - \frac{\lambda_\ell}{\varphi(r)^2} R(r) = 0 \quad \text{for } r \in (0,\infty).
    \end{align}
    By Lemma \ref{lem:model-ODE-lemma}, there is a unique solution $R_{\lambda_\ell}: [0,\infty) \to \R$ to \eqref{eq:model-ODE-2} satisfying $R_{\lambda_\ell}(r) \sim r^{\lambda_\ell}$ as $r \to \infty$ and extending continuously to $R_{\lambda_\ell}(0) = 0$. Hence, the function $u(r,\theta) = R_{\lambda_\ell}(r) \Theta(\theta)$ is continuous on $P$ and drift-harmonic on $P \setminus \{0\}$. By a removable singularity theorem, e.g. \cite{miranda}*{Theorem 27, VI}, $u$ is $C^2$ on $P$ and therefore drift-harmonic on $P$. This shows that $\B_d(P)$, as defined in the proposition, is a subset of $\H_d(P)$.
    
    It is clear that $\B_d(P)$ is linearly independent. That it spans $\H_d(P)$ follows from a standard argument using the maximum principle (e.g. \cite{cm97a}*{Theorem 1.11}).
\end{proof}

\section{Second-order control of asymptotically paraboloidal metrics} \label{sec:appB}

Let $(M^n,g,r)$ be an AP manifold. As per \S\ref{subsec:conventions}, we use $(r,\theta)$ coordinates on $\{r > 0\} \cong (0,\infty) \times \Sigma$. Greek indices ($\alpha,\beta,\ldots$) will only run over the $\theta$ coordinates.

This appendix computes growth bounds for $g$ up to second order.
Since components of the form $g_{\alpha r}$ and their derivatives all vanish, they are omitted from the listings of Lemma \ref{lem:first-0-1-control} and Corollary \ref{cor:gtau-components}.

\begin{lemma} \label{lem:first-0-1-control}
    We have $g_{rr} = 1 + \O(r^{-\mu})$ as $r \to \infty$, and $C^{-1} r \leq g_{\Sigma_r} \leq Cr$ as bilinear forms for all $r>0$. As $r \to \infty$, we also have
    \begin{alignat*}{3}
        \del_r g_{rr} &= \O(r^{-\mu-1})  &\qquad  \Gamma_{rr}^r &= \O(r^{-\mu-1})  &\qquad  \del_r \del_r g_{rr} &= \O(r^{-2}) \\
        \del_\alpha g_{rr} &= \O(r^{-\mu-\frac{1}{2}})  &  \Gamma_{rr}^\alpha &= \O(r^{-\mu-\frac{3}{2}})  &  \del_r \del_\alpha g_{rr} &= \O(r^{-\frac{3}{2}}) \\
        \del_r g_{\alpha\beta} &= \O(1)  &  \Gamma_{r\alpha}^r &= \O(r^{-\mu-\frac{1}{2}})  &  \del_\alpha \del_\beta g_{rr} &= \O(r^{-\mu-\frac{1}{2}}) \\
        \del_\alpha g_{\beta\gamma} &= \O(r)  &  \Gamma_{r\alpha}^\beta &= \O(r^{-1})  &  \del_r \del_\alpha g_{\beta\gamma} &= \O(1) \\
        &  &  \Gamma_{\alpha\beta}^r &= \O(1)  &  \del_r \del_r g_{\alpha\beta} &= \O(r^{-1}) \\
        &  &  \Gamma_{\alpha\beta}^\gamma &= \O(1)  &  \del_\alpha \del_\beta g_{\gamma\delta} &= \O(r).
    \end{alignat*}
\end{lemma}
\begin{proof}
    The bound on $g_{rr} = |\nabla r|^{-2}$ follows from (i) in Definition \ref{def:asymp-parab}, and the bound on $g_{\Sigma_r}$ is from Theorem \ref{thm:asymp-link}. Next, using \eqref{eq:hess-r}, we compute
    \begin{align}
        \del_r g_{rr} &= \del_r (|\nabla r|^2)^{-1} = -2|\nabla r|^{-4} \nabla^2 r(\del_r,\nabla r) =
        \frac{1}{r|\nabla r|^4} (|\nabla r|^2 - 1) - \frac{1}{r|\nabla r|^2} \eta_{rr}, \label{eq:drgrr} \\
        \del_\alpha g_{rr} &= -2|\nabla r|^{-4} \nabla^2 r(\del_\alpha,\nabla r)
        = -\frac{1}{r|\nabla r|^2} \eta_{\alpha r}. \label{eq:dalpha-grr}
    \end{align}
    The bound on $\del_r g_{rr}$ now follows from Definition \ref{def:asymp-parab}, which gives $|\nabla r|^2 - 1 = \O(r^{-\mu})$ and $\eta_{rr} = \O(r^{-\mu})$.
    Similarly, we have $|\eta| = \O(r^{-\mu})$, $|\del_\alpha| = \sqrt{g_{\alpha\alpha}} = \O(\sqrt{r})$, and $|\nabla r| \leq C$, so $|\eta_{\alpha r}| \leq \O(r^{-\mu+\frac{1}{2}})$. The estimate for $\del_\alpha g_{rr}$ follows. The bound for $\del_r g_{\alpha\beta}$ follows from \eqref{eq:gL-convergence-3}, and the bound for $\del_\alpha g_{\beta\gamma}$ follows from (iii) in Definition \ref{def:asymp-parab}. This proves all bounds in the first column of the lemma. Using the explicit formula for Christoffel symbols in terms of the metric, the bounds in the second column follow.

    Using the fact that $|\nabla\eta| = \O(r^{-1})$ (which implies $|\nabla_r \eta_{rr}| = \O(r^{-1})$, $|\nabla_r \eta_{\alpha r}| = \O(r^{-\frac{1}{2}})$ for instance), and the bounds in the first two columns, we estimate
    \begin{align} 
        |\del_r \eta_{rr}| &= |\nabla_r \eta_{rr} + 2\Gamma_{rr}^i \eta_{ri}| \leq |\nabla_r \eta_{rr}| + 2|\Gamma_{rr}^r||\eta_{rr}| + 2|\Gamma_{rr}^\alpha||\eta_{r\alpha}| = \O(r^{-1}), \label{eq:dretarr} \\
        |\del_r \eta_{\alpha r}| &\leq |\nabla_r \eta_{\alpha r}| + |\Gamma_{r\alpha}^r||\eta_{rr}| + |\Gamma_{r\alpha}^{\beta}||\eta_{\beta r}| + |\Gamma_{rr}^r||\eta_{\alpha r}| + |\Gamma_{rr}^\beta| |\eta_{\alpha\beta}| = \O(r^{-\frac{1}{2}}). \label{eq:eta-derivs-2}
    \end{align}
    Similarly, we estimate
    \begin{align}
        |\del_r \eta_{\alpha\beta}| = \O(1), \quad |\del_\alpha \eta_{rr}| = \O(r^{-\frac{1}{2}}), \quad  |\del_\alpha \eta_{\beta r}| = \O(r^{-\mu+\frac{1}{2}}), \quad |\del_\alpha \eta_{\beta\gamma}| = \O(r^{-\mu+1}). \label{eq:eta-derivs-3}
    \end{align}
    Applying $\del_r$ to \eqref{eq:drgrr} gives
    \begin{equation}
        \del_r \del_r g_{rr} = -\frac{|\nabla r|^2 - 1}{r^2|\nabla r|^4} + \frac{\del_r |\nabla r|^{-4}}{r}(|\nabla r|^2 - 1) + \frac{\del_r |\nabla r|^2}{r|\nabla r|^4} + \frac{1}{r^2|\nabla r|^2} \eta_{rr} - \frac{\del_r |\nabla r|^{-2}}{r} \eta_{rr} - \frac{1}{r|\nabla r|^2} \del_r \eta_{rr}. \label{eq:drdrgrr}
    \end{equation}
    We have $\del_r |\nabla r|^{-2} = \del_r g_{rr} = \O(r^{-\mu-1})$, so $\del_r |\nabla r|^{-4} = 2|\nabla r|^{-2} \del_r |\nabla r|^{-2} = \O(r^{-\mu-1})$. Also $\del_r |\nabla r|^2 = \del_r(|\nabla r|^{-2})^{-1} = -|\nabla r|^4 \del_r|\nabla r|^{-2} = \O(r^{-\mu-1})$. Using these facts and \eqref{eq:dretarr} in \eqref{eq:drdrgrr}, we get $\del_r \del_r g_{rr} = \O(r^{-2})$. Next, we differentiate \eqref{eq:dalpha-grr} and use \eqref{eq:eta-derivs-2}, \eqref{eq:eta-derivs-3} to get
    \begin{align}
        \del_r \del_\alpha g_{rr} &= \frac{1}{r^2|\nabla r|^2} \eta_{\alpha r} - \frac{1}{r}(\del_r g_{rr}) \eta_{\alpha r} - \frac{1}{r|\nabla r|^2} \del_r \eta_{\alpha r} = \O(r^{-\frac{3}{2}}), \\
        \del_\alpha \del_\beta g_{rr} &= -\del_\alpha \left( \frac{1}{r|\nabla r|^2} \eta_{\beta r} \right) = -\frac{1}{r}(\del_\alpha g_{rr}) \eta_{\beta r} - \frac{1}{r|\nabla r|^2} \del_\alpha \eta_{\beta r} = \O(r^{-\mu-\frac{1}{2}}).
    \end{align}
    For the estimate on $\del_r \del_\alpha g_{\beta\gamma}$, we compute from the definition of $\eta$:
    \begin{align}
        \del_\alpha \eta_{\beta\gamma} &= \del_\alpha \inner{\nabla_\beta \nabla r^2}{\del_\gamma} - \del_\alpha g_{\beta\gamma} = -\del_\alpha \inner{\nabla r^2}{\nabla_\beta \del_\gamma} - \del_\alpha g_{\beta\gamma} \\
        &= -\del_\alpha (2r \inner{\nabla r}{\Gamma_{\beta\gamma}^\delta \del_\delta}) - \del_\alpha g_{\beta\gamma} = -\del_\alpha (2r \Gamma_{\beta\gamma}^r) - \del_\alpha g_{\beta\gamma} \\
        &= -2r \del_\alpha \Gamma_{\beta\gamma}^r - \del_\alpha g_{\beta\gamma} = r\del_\alpha(|\nabla r|^2 \del_r g_{\beta\gamma}) - \del_\alpha g_{\beta\gamma} \\
        &= r|\nabla r|^2 \del_r \del_\alpha g_{\beta\gamma} + r(\del_\alpha |\nabla r|^2) \del_r g_{\beta\gamma} - \del_\alpha g_{\beta\gamma}.
    \end{align}
    Rearranging and using that $\del_\alpha |\nabla r|^2 = -|\nabla r|^{-4} \del_\alpha |\nabla r|^{-2} = -|\nabla r|^{-4} \del_\alpha g_{rr}$, this becomes
    \begin{align}
        \del_r \del_\alpha g_{\beta\gamma} &= \frac{1}{r|\nabla r|^2} (\del_\alpha \eta_{\beta\gamma} + r|\nabla r|^{-4} (\del_\alpha g_{rr})(\del_r g_{\beta\gamma}) + \del_\alpha g_{\beta\gamma}) = \O(1).
    \end{align}
    The estimate for $\del_r \del_r g_{\alpha\beta}$ is obtained similarly to the one just proved. That is, we compute
    \begin{align}
        \del_r \eta_{\alpha\beta} &= -\del_r (2r\Gamma_{\alpha\beta}^r) - \del_r g_{\alpha\beta} = \del_r(r |\nabla r|^2 \del_r g_{\alpha\beta}) - \del_r g_{\alpha\beta} \\
        &= (|\nabla r|^2 - 1) \del_r g_{\alpha\beta} + r(\del_r |\nabla r|^2) \del_r g_{\alpha\beta} + r|\nabla r|^2 \del_r \del_r g_{\alpha\beta}
    \end{align}
    and rearrange to get
    \begin{align}
        \del_r \del_r g_{\alpha\beta} &= \frac{1}{r|\nabla r|^2} [\del_r \eta_{\alpha\beta} + (1-|\nabla r|^2)\del_r g_{\alpha\beta} + r|\nabla r|^{-4} (\del_r g_{rr}) \del_r g_{\alpha\beta}] = \O(r^{-1}).
    \end{align}
    Finally, the bound $\del_\alpha \del_\beta g_{\gamma\delta} = O(r)$ is from condition (iii) in Definition \ref{def:asymp-parab}.
\end{proof}

Lemma \ref{lem:first-0-1-control} provides local uniform control on the metrics $g_\tau := dr^2 + \tau^{-1} g_{\Sigma_{\tau r}}$ from Corollary \ref{cor:weakly-parab}:
\begin{corollary} \label{cor:gtau-components}
    There exists $C > 0$ such that for any $\tau > 1$ and at any point $y \in \{\frac{1}{2} \leq r \leq \frac{3}{2}\}$, we have
    \begin{gather}
        (g_{\tau})_{rr} = 1, \quad C^{-1} \leq (g_\tau)_{\alpha\beta} \leq C. \label{eq:gtau-ab}
    \end{gather}
    Moreover, for any indices $i,j,k,l$,
    \begin{align}
        \sup_{\tau \geq 1} \sup_{\{\frac{1}{2} \leq r \leq \frac{3}{2}\}} \left( |\del_i (g_\tau)_{jk}| + |\del_i \del_j (g_\tau)_{kl}| + |(\Gamma^{g_\tau})_{ij}^k| + |\Rm^{g_\tau}|_{g_\tau} \right) < \infty. \label{eq:gtau-rm}
    \end{align}
\end{corollary}
\begin{proof}
    For any $y \in \{\frac{1}{2} \leq r \leq \frac{3}{2}\}$, write $y = (r(y),\theta)$. Then $(g_\tau)_{rr}(y) = 1$ and $(g_\tau)_{\alpha\beta}(y) = \tau^{-1} g_{\alpha\beta}(\tau r(y), \theta)$. Using the two-sided estimate for $g_{\alpha\beta}$ in Lemma \ref{lem:first-0-1-control}, this implies \eqref{eq:gtau-ab}. Differentiating and using Lemma \ref{lem:first-0-1-control} again, we get
    \begin{align}
        (\del_r (g_{\tau})_{rr})(y) &= (\del_\alpha (g_{\tau})_{rr})(y) = 0, \\
        (\del_r (g_{\tau})_{\alpha\beta})(y) &= \tau^{-1} \del_r (g_{\alpha\beta}(\tau r(y),\theta)) = (\del_r g_{\alpha\beta})(\tau r(y),\theta) = \O(1), \\
        (\del_\alpha (g_{\tau})_{\beta\gamma})(y) &= \tau^{-1} (\del_\alpha g_{\beta\gamma})(\tau r(y),\theta) = \O(1)
    \end{align}
    as $\tau \to \infty$, where the bounds $\O(1)$ are independent of $y$. This proves the bounds on the first derivatives of $g_\tau$. The second derivatives are handled similarly. The bounds on Christoffel symbols and Riemann curvature follow from these.
\end{proof}

\section{Proofs of estimates for drift-harmonic functions} \label{sec:appC}

This section proves Theorems \ref{thm:parabolicSchauder}, \ref{thm:scaled-back-schauder} and \ref{thm:meanValue} for an AP manifold $(M^n,g,r)$ and $f \in C^\infty(M)$ satisfying Assumption \ref{assump:f}. We will use the conventions from \S\ref{subsec:conventions} and notation and setup from \S\ref{subsec:pointwise-estimates}.

\begin{lemma} \label{lem:phi-more-bounds}
    There exists $C>0$ such that for all $\rho > 0$, $r \in [\rho-\sqrt{\rho},\rho]$ and $t \in [0,\frac{7}{8}]$,
    \begin{align}
        \left| \frac{\del\phi_{\rho t}}{\del r}(r) + 1 \right| &\leq C\rho^{-1}, \quad \left| \frac{\del^2 \phi_{\rho t}}{\del r^2}(r) \right| \leq C\rho^{-\frac{3}{2}}, \quad \left| \frac{\del^3\phi_{\rho t}}{\del r^3}(r) \right| \leq C\rho^{-\frac{3}{2}}.
    \end{align}
\end{lemma}
\begin{proof}
    Let $h(r,t) = \frac{\del\phi_t}{\del r}(r)$. Then $h(r,0) = 1$ and
    \begin{align} \label{eq:h-ODE}
        \frac{\del}{\del t}h(r,t) = \frac{\del}{\del r} \frac{\del}{\del t} \phi_t(r) = \frac{\del}{\del r} f'(\phi_t(r)) = f''(\phi_t(r)) h(r,t).
    \end{align}
    Integrating this ODE, we see that
    \begin{align} \label{eq:dphir-formula}
        h(r,t) = \exp\left\{ \int_0^t f''(\phi_s(r)) \, ds \right\} = \exp\left\{ \int_r^{\phi_t(r)} \frac{f''(\xi)}{f'(\xi)} \, d\xi \right\} = \frac{f'(\phi_t(r))}{f'(r)}.
    \end{align}
    Now let $r \in [\rho-\sqrt{\rho},\rho]$ and $t \in [0,\frac{7}{8}]$. By Lemma \ref{lem:phi-bounds}, if $\rho$ is large, then $\phi_{\rho t}(r) \in [(1-t)\rho - 2\sqrt{\rho}, (1-t)\rho + \sqrt{\rho}] \subset [\frac{1}{16}\rho, 2\rho]$. Then Assumption \ref{assump:f} gives $|f'(\phi_{\rho t}(r)) + 1| \leq C\rho^{-1}$ and $|f'(r) + 1| \leq C\rho^{-1}$, where $C > 0$ is independent of $\rho, r$ and $t$. From this and \eqref{eq:dphir-formula}, we get
    \begin{align} \label{eq:hest}
        \left|h(r,\rho t) + 1 \right| = \left| \frac{f'(\phi_t(r))}{f'(r)} + 1 \right| \leq C\rho^{-1},
    \end{align}
    proving the first claimed estimate. Next, we differentiate \eqref{eq:dphir-formula} to get
    \begin{align} \label{eq:second-deriv-phi}
        \frac{\del^2\phi_t}{\del r^2}(r) = \frac{\del}{\del r} h(r,t) = \frac{f''(\phi_t(r))}{f'(r)} \frac{\del\phi_t}{\del r}(r) - \frac{f'(\phi_t(r))}{f'(r)^2} f''(r) = \left( \frac{f''(\phi_t(r)) - f''(r)}{f'(r)} \right) h(r,t).
    \end{align}
    If $\rho$ is large, $r \in [\rho-\sqrt{\rho}, \rho]$, and $t \in [0,\frac{7}{8}]$, then $\phi_{\rho t}(r) \in [\frac{1}{16}\rho,2\rho]$ as above. Assumption \ref{assump:f}, \eqref{eq:hest} and \eqref{eq:second-deriv-phi} give $\left| \frac{\del^2\phi_{\rho t}}{\del r^2}(r) \right| \leq C\rho^{-3/2}$. Differentiating \eqref{eq:second-deriv-phi} and estimating similarly, we get $\left| \frac{\del^3\phi_{\rho t}}{\del r^3}(r) \right| \leq C\rho^{-3/2}$.
\end{proof}

The next lemma uniformly controls the metrics $\Psi_\rho^*\hat{g}^{(\rho)}(t)$ on the spacetime domain $\overline{\Omega}^{\rho_0} \times [0,\frac{7}{8}]$.

\begin{lemma} \label{lem:unif-ctrl}
    There exists $C>0$ such that for all $\rho \geq 1$,
    \begin{align} \label{eq:C0-convergence}
        \sup_{\bar{\Omega}^{\rho_0} \times [0,\frac{7}{8}]} \left| \Psi_\rho^*\hat{g}^{(\rho)}(t) - (\rho_0^{-1} dr^2 + (1-t)g_X) \right| \leq C\rho^{-\min\{\mu,1\}}.
    \end{align}
    where the norm is taken using a fixed background metric on $\overline{\Omega}^{\rho_0} \times [0,\frac{7}{8}]$.
    In particular, the metrics
    \begin{align} \label{eq:unif-family}
        \{ \Psi_\rho^* \hat{g}^{(\rho)}(t) : \rho > 0, \, t \in [0,\tfrac{7}{8}] \}
    \end{align}
    are all uniformly equivalent on $\overline{\Omega}^{\rho_0}$.
    Moreover, for any indices $i,j,k,l$,
    \begin{align}
        \sup_{\bar{\Omega}^{\rho_0} \times [0,\frac{7}{8}]} \Big( |(\Psi_\rho^* \hat{g}^{(\rho)}(t))_{jk}| &+ |\del_i (\Psi_\rho^* \hat{g}^{(\rho)}(t))_{jk}| + |\del_i \del_j (\Psi_\rho^* \hat{g}^{(\rho)}(t))_{kl}| \\
        & + |\del_t (\Psi_\rho^* \hat{g}^{(\rho)}(t))_{jk}| + |\del_t \del_i(\Psi_\rho^*\hat{g}^{(\rho)}(t))_{jk}| \Big) = \O(1) \quad \text{as } \rho \to \infty. \label{eq:supbounds}
    \end{align}
\end{lemma}
\begin{proof}
    Let $(x,t) \in \overline{\Omega}^{\rho_0} \times [0,\frac{7}{8}]$. We begin by recording a few estimates. Write $x = (s,\theta)$, where $s \in [\rho_0 - \sqrt{\rho_0}, \rho_0]$ and $\theta \in \Sigma$. Then $\psi_\rho(s) \in [\rho-\sqrt{\rho},\rho]$, so Lemma \ref{lem:phi-bounds} gives that for all large $\rho$,
    \begin{align}
        (\Phi_{\rho t} \circ \Psi_\rho)(x) = (\phi_{\rho t}(\psi_\rho(s)), \theta) &\in \{(1-t)\rho - 2\sqrt{\rho} \leq r \leq (1-t)\rho + \sqrt{\rho}\} \subset \left\{ \tfrac{1}{16}\rho \leq r \leq 2\rho \right\}. \label{eq:phi-estimate-x}
    \end{align}
    At $(x,t)$, we have
    \begin{align} \label{eq:rescaled-gjks}
        (\Psi_\rho^*\hat{g}^{(\rho)}(t))_{jk} = \rho^{-1} ((\Phi_{\rho t} \circ \Psi_\rho)^*g)_{jk} = \rho^{-1} g(d(\Phi_{\rho t} \circ \Psi_\rho)|_x(\del_j), d(\Phi_{\rho t} \circ \Psi_\rho)|_x(\del_k)).
    \end{align}
    We also have
    \begin{align}
        \psi_\rho'(s) &= \sqrt{\rho/\rho_0}, \label{eq:deriv1} \\
        d(\Phi_{\rho t} \circ \Psi_\rho)|_x(\del_r) &= \sqrt{\frac{\rho}{\rho_0}} \frac{\del \phi_{\rho t}}{\del r}(\psi_\rho(s)) \del_r, \label{eq:deriv2} \\
        d(\Phi_{\rho t} \circ \Psi_\rho)|_x(\del_\alpha) &= \del_\alpha. \label{eq:deriv3}
    \end{align}
    Also, we have $\frac{\del}{\del t}(\Phi_{\rho t} \circ \Psi_\rho)(x) = \rho \frac{\del\phi_{\rho t}}{\del t}(\psi_\rho(s)) \del_r = \rho f'(\phi_{\rho t}(\psi_\rho(s))) \del_r$, so for any indices $i,j$,
    \begin{align}
        \frac{\del}{\del t} \left[ g_{ij}((\Phi_{\rho t} \circ \Psi_\rho)(x)) \right]
        &= \rho f'(\phi_{\rho t}(\psi_\rho(s))) \cdot (\del_r g_{ij})((\Phi_{\rho t} \circ \Psi_\rho)(x)). \label{eq:deriv4}
    \end{align}
    Moreover,
    \begin{align}
        \frac{\del}{\del t} \left(\frac{\del\phi_{\rho t}}{\del r}(\psi_\rho(s)) \right) &= \rho \frac{\del}{\del r}\Big|_{r=\psi_\rho(s)} \left(\frac{\del \phi_{\rho t}}{\del t}(r) \right) = \rho \frac{\del}{\del r}\Big|_{r=\psi_\rho(s)} (f'(\phi_{\rho t}(r))) \\
        &= \rho f''(\phi_{\rho t}(\psi_\rho(s))) \frac{\del \phi_{\rho t}}{\del r}(\psi_\rho(s)). \label{eq:deriv5}
    \end{align}
    \noeqref{eq:deriv4}

    Using \eqref{eq:rescaled-gjks}, \eqref{eq:deriv1}, \eqref{eq:deriv2} and \eqref{eq:deriv3}, we see that at $(x,t) \in \overline{\Omega}^{\rho_0} \times [0,\frac{7}{8}]$, with $x = (s,\theta)$,
    \begin{align}
        (\Psi_\rho^*\hat{g}^{(\rho)}(t))_{rr} &= \rho_0^{-1} \left[ \frac{\del\phi_{\rho t}}{\del r}(\psi_\rho(s)) \right]^2 g_{rr}((\Phi_{\rho t} \circ \Psi_\rho)(x)), \label{eq:grr-scaled} \\
        (\Psi_\rho^*\hat{g}^{(\rho)}(t))_{\alpha\beta} &= \rho^{-1} g_{\alpha\beta}((\Phi_{\rho t} \circ \Psi_\rho)(x)), \label{eq:gab-scaled} \\
        (\Psi_\rho^*\hat{g}^{(\rho)}(t))_{\alpha r} &= 0. \label{eq:gcross-scaled}
    \end{align}
    As $g_{\Sigma_\rho}$ is the restriction of $g$ to $\{r=\rho\} \cong \Sigma$, it follows that at $(x,t)$,
    \begin{align}
        \Psi_\rho^*\hat{g}^{(\rho)}(t) &= \rho_0^{-1} \left[ \frac{\del\phi_{\rho t}}{\del r}(\psi_\rho(s)) \right]^2 g_{rr}((\Phi_{\rho t} \circ \Psi_\rho)(x)) dr^2 + \rho^{-1} g_{\Sigma_{\phi_{\rho t}(\psi_\rho(s))}}. \label{eq:g-scaled-formula}
    \end{align}
    Using Lemma \ref{lem:phi-bounds}, Lemma \ref{lem:phi-more-bounds}, and the $C^0$-convergence $g_X(\rho) \to g_X$ from Theorem \ref{thm:asymp-link}, we estimate
    \begin{align}
        \left| \rho^{-1} g_{\Sigma_{\phi_{\rho t}(\psi_\rho(s))}} - (1-t)g_X \right| &= \left| \frac{\phi_{\rho t}(\psi_\rho(s))}{\rho} \cdot  g_X(\phi_{\rho t}(\psi_\rho(s))) - (1-t)g_X \right| \\
        &\leq \left| \frac{\phi_{\rho t}(\psi_\rho(s))}{\rho} - (1-t) \right| |g_X(\phi_{\rho t}(\psi_\rho(s)))| + |1-t| |g_X(\phi_{\rho t}(\psi_\rho(s))) - g_X| \\
        &\leq C\rho^{-1} + C\rho^{-\mu}, \label{eq:orbital-bd}
    \end{align}
    where $C$ is independent of $(x,t) \in \overline{\Omega}^{\rho_0} \times [0,\frac{7}{8}]$.
    Also, Lemma \ref{lem:first-0-1-control} and \eqref{eq:phi-estimate-x} yield
    \begin{align}
        |g_{rr}((\Phi_{\rho t} \circ \Psi_\rho)(x)) - 1| &\leq C\rho^{-\mu}.
    \end{align}
    as $\rho \to \infty$. Together with Lemma \ref{lem:phi-more-bounds}, this implies
    \begin{align} \label{eq:radial-bd}
        \left| \left[ \frac{\del\phi_{\rho t}}{\del r}(\psi_\rho(s)) \right]^2 g_{rr}((\Phi_{\rho t} \circ \Psi_\rho)(x)) - 1 \right| \leq C\rho^{-\mu}.
    \end{align}
    By \eqref{eq:g-scaled-formula}, \eqref{eq:orbital-bd} and \eqref{eq:radial-bd}, we have
    \begin{align}
        \sup_{\overline{\Omega}^{\rho_0} \times [0,\frac{7}{8}]} \left| \Psi_\rho^*\hat{g}^{(\rho)}(t) - (\rho_0^{-1} dr^2 + (1-t)g_X) \right| &\leq \sup_{\overline{\Omega}^{\rho_0} \times [0,\frac{7}{8}]} \left| \left( \left[ \frac{\del \phi_{\rho t}}{\del r}(\psi_\rho(r)) \right]^2 g_{rr}((\Phi_{\rho t} \circ \Psi_\rho)(x)) - 1 \right) \rho_0^{-1} dr^2 \right| \\
        &\quad + \sup_{\overline{\Omega}^{\rho_0} \times [0,\frac{7}{8}]} \left| \rho^{-1} g_{\Sigma_{\phi_{\rho t}(\psi_\rho(s))}} - (1-t)g_X \right| \\
        &\leq C\rho^{-\mu} + C\rho^{-1},
    \end{align}
    which proves \eqref{eq:C0-convergence}. The uniform equivalence of \eqref{eq:unif-family} also follows.

    To obtain \eqref{eq:supbounds}, we compute all relevant derivatives of $\Psi_\rho^*\hat{g}^{(\rho)}(t)$ by differentiating \eqref{eq:grr-scaled} and \eqref{eq:gab-scaled}, with the help of \eqref{eq:deriv1}--\eqref{eq:deriv5}. For example, we compute at $(x,t) \in \overline{\Omega}^{\rho_0} \times [0,\frac{7}{8}]$ with $x = (s,\theta)$:
    {\footnotesize  
    \setlength{\abovedisplayskip}{6pt}
    \setlength{\belowdisplayskip}{\abovedisplayskip}
    \setlength{\abovedisplayshortskip}{0pt}
    \setlength{\belowdisplayshortskip}{3pt}
    \begin{align}
        \del_r (\Psi_\rho^* \hat{g}^{(\rho)}(t))_{rr} &= 2\rho_0^{-1} \left[ \frac{\del\phi_{\rho t}}{\del r}(\psi_\rho(s)) \right] \left[ \frac{\del^2 \phi_{\rho t}}{\del r^2}(\psi_\rho(s)) \right] \sqrt{\frac{\rho}{\rho_0}} g_{rr}((\Phi_{\rho t} \circ \Psi_\rho)(x)) \\
        \del_r \del_r (\Psi_\rho^* \hat{g}^{(\rho)}(t))_{rr} &= 2\rho_0^{-1} \left[ \frac{\del^2\phi_{\rho t}}{\del r^2}(\psi_\rho(s)) \right]^2 \frac{\rho}{\rho_0} g_{rr}((\Phi_{\rho t} \circ \Psi_\rho)(x)) \\
        &\quad + 2\rho_0^{-1} \left[ \frac{\del\phi_{\rho t}}{\del r}(\psi_\rho(s)) \right] \left[ \frac{\del^3\phi_{\rho t}}{\del r^3}(\psi_\rho(s)) \right] \frac{\rho}{\rho_0} g_{rr}((\Phi_{\rho t} \circ \Psi_\rho)(x)) \\
        &\quad + 2\rho_0^{-1} \left[ \frac{\del\phi_{\rho t}}{\del r}(\psi_\rho(s)) \right]^2 \left[ \frac{\del^2 \phi_{\rho t}}{\del r^2}(\psi_\rho(s)) \right] \frac{\rho}{\rho_0} (\del_r g_{rr})((\Phi_{\rho t} \circ \Psi_\rho)(x)) \\
        &\quad + 3\rho_0^{-1} \left[ \frac{\del\phi_{\rho t}}{\del r}(\psi_\rho(s)) \right]^2 \left[ \frac{\del^2\phi_{\rho t}}{\del r^2}(\psi_\rho(s)) \right] \frac{\rho}{\rho_0} (\del_r g_{rr})((\Phi_{\rho t} \circ \Psi_\rho)(x)) \\
        &\quad + \rho_0^{-1} \left[ \frac{\del\phi_{\rho t}}{\del r}(\psi_\rho(s)) \right]^4 \frac{\rho}{\rho_0} (\del_r \del_r g_{rr})((\Phi_{\rho t} \circ \Psi_\rho)(x)), \\
        \del_t(\Psi_\rho^*\hat{g}^{(\rho)}(t))_{rr} &= 2\rho_0^{-1} \left[ \frac{\del\phi_{\rho t}}{\del r}(\psi_\rho(s)) \right]^2 \rho f''(\phi_{\rho t}(\psi_\rho(s))) \cdot g_{rr}((\Phi_{\rho t} \circ \Psi_\rho)(x)) \\
        &\quad + \rho_0^{-1} \left[ \frac{\del\phi_{\rho t}}{\del r}(\psi_\rho(s)) \right]^2 \rho f'(\phi_{\rho t}(\psi_\rho(s))) \cdot (\del_r g_{rr})((\Phi_{\rho t} \circ \Psi_\rho)(x)), \\
        \del_t \del_r (\Psi_\rho^* \hat{g}^{(\rho)}(t))_{rr} &= \del_r \del_t (\Psi_\rho^* \hat{g}^{(\rho)}(t))_{rr} \\
        &= 4\rho_0^{-1} \left[ \frac{\del\phi_{\rho t}}{\del r}(\psi_\rho(s)) \right] \left[ \frac{\del^2\phi_{\rho t}}{\del r^2}(\psi_\rho(s)) \right] \sqrt{\frac{\rho}{\rho_0}} \rho f''(\phi_{\rho t}(\psi_\rho(s))) \cdot g_{rr}((\Phi_{\rho t} \circ \Psi_\rho)(x)) \\
        &\quad + 2\rho_0^{-1} \left[ \frac{\del\phi_{\rho t}}{\del r}(\psi_\rho(s)) \right]^3 \rho f'''(\phi_{\rho t}(\psi_\rho(s))) \sqrt{\frac{\rho}{\rho_0}} \cdot g_{rr}((\Phi_{\rho t} \circ \Psi_\rho)(x)) \\
        &\quad + 2\rho_0^{-1} \left[ \frac{\del\phi_{\rho t}}{\del r}(\psi_\rho(s)) \right]^3 \rho f''(\phi_{\rho t}(\psi_\rho(s))) \cdot (\del_r g_{rr})((\Phi_{\rho t} \circ \Psi_\rho)(x)) \cdot \sqrt{\frac{\rho}{\rho_0}}\\
        &\quad + 2\rho_0^{-1} \left[ \frac{\del\phi_{\rho t}}{\del r}(\psi_\rho(s)) \right] \left[ \frac{\del^2\phi_{\rho t}}{\del r^2}(\psi_\rho(s)) \right] \sqrt{\frac{\rho}{\rho_0}} \rho f'(\phi_{\rho t}(\psi_\rho(s))) \cdot (\del_r g_{rr})((\Phi_{\rho t} \circ \Psi_\rho)(x)) \\
        &\quad + \rho_0^{-1} \left[ \frac{\del\phi_{\rho t}}{\del r}(\psi_\rho(s)) \right]^3 \rho f''(\phi_{\rho t}(\psi_\rho(s))) \sqrt{\frac{\rho}{\rho_0}} \cdot (\del_r g_{rr})((\Phi_{\rho t} \circ \Psi_\rho)(x)) \\
        &\quad + \rho_0^{-1} \left[ \frac{\del\phi_{\rho t}}{\del r}(\psi_\rho(s)) \right]^3 \rho f'(\phi_{\rho t}(\psi_\rho(s))) \cdot (\del_r \del_r g_{rr})((\Phi_{\rho t} \circ \Psi_\rho)(x)) \cdot \sqrt{\frac{\rho}{\rho_0}}.
    \end{align}
    }
    The other computations are similar and are omitted. Using Assumption \ref{assump:f}, Lemma \ref{lem:first-0-1-control}, Lemma \ref{lem:phi-more-bounds}, and \eqref{eq:phi-estimate-x}, one checks that each expression is uniformly bounded (with respect to $\rho$) on $\overline{\Omega}^{\rho_0} \times [0,\frac{7}{8}]$. The same is true for \eqref{eq:grr-scaled} and \eqref{eq:gab-scaled}. Then \eqref{eq:supbounds} follows.
\end{proof}

Using Lemma \ref{lem:unif-ctrl}, we proceed to prove Theorems \ref{thm:parabolicSchauder}, \ref{thm:scaled-back-schauder} and \ref{thm:meanValue}.

\begin{proof}[Proof of Theorem \ref{thm:parabolicSchauder}]
    By Lemma \ref{lem:conversion-to-parabolic}, the function $w = \Psi_\rho^* \hat{u}^{(\rho)}$ satisfies
    \begin{align} \label{eq:heat-v}
        \del_t w = \Delta_{\Psi_\rho^*\hat{g}^{(\rho)}(t)} w = \Psi_\rho^*\hat{g}^{(\rho)}(t)^{ij} \cdot \del_i \del_j w - \Psi_\rho^*\hat{g}^{(\rho)}(t)^{ij} \Gamma(\Psi_\rho^*\hat{g}^{(\rho)}(t))_{ij}^k \cdot \del_k w \quad \text{on } \overline{\Omega}^{\rho_0} \times [0,\tfrac{7}{8}].
    \end{align}
    By Lemma \ref{lem:unif-ctrl}, in particular the uniform equivalence of the metrics \eqref{eq:unif-family}, the equation \eqref{eq:heat-v} is uniformly parabolic with ellipticity constants bounded independently of $\rho$. Moreover, \eqref{eq:supbounds} implies that
    \begin{align} \label{eq:unif-pde-bds}
        \sup_{\rho > 0} \left( \norm{ (\Psi_\rho^*\hat{g}^{(\rho)}(t))_{jk} }_{C^{\alpha,\frac{\alpha}{2}}(\bar{\Omega}^{\rho_0} \times [0,\frac{7}{8}])} + \norm{ \del_i(\Psi_\rho^*\hat{g}^{(\rho)}(t))_{jk} }_{C^{\alpha,\frac{\alpha}{2}}(\bar{\Omega}^{\rho_0} \times [0,\frac{7}{8}])} \right) &< \infty,
    \end{align}
    so the equation \eqref{eq:heat-v} has uniformly H\"older-bounded (w.r.t. $\rho$) coefficients. The theorem now follows from parabolic interior Schauder estimates (e.g. \cite{krylov}*{Theorem 8.9.2}).
\end{proof}

\begin{proof}[Proof of Theorem \ref{thm:scaled-back-schauder}]
    Lemma \ref{lem:unif-ctrl} shows that the metrics $\Psi_\rho^*\hat{g}^{(\rho)}(0)$, for all $\rho > 0$, are uniformly equivalent on $\overline{\Omega}^{\rho_0}$. So there exists $C > 0$ such that for all $\rho > 0$, $\tau \in (0,\frac{1}{2})$ and functions $w$ on $\overline{\Omega}^{\rho_0} \times [0,\frac{7}{8}]$,
    \begin{align} \label{eq:est-110}
        \norm{w}_{C^{2+\alpha,1+\frac{\alpha}{2}}(\overline{\Omega}^{\rho_0}_{\tau} \times [\tau,\frac{7}{8}]; \Psi_{\rho}^*\hat{g}^{(\rho)}(0))} 
        \leq C \norm{w}_{C^{2+\alpha,1+\frac{\alpha}{2}}(\overline{\Omega}^{\rho_0}_{\tau} \times [\tau,\frac{7}{8}]; \Psi_{\rho_0}^*\hat{g}^{(\rho_0)}(0))}.
    \end{align}
    Take $w = \Psi_\rho^*\hat{u}^{(\rho)}$. Then Theorem \ref{thm:parabolicSchauder} estimates the right-hand side of \eqref{eq:est-110} by $C(\tau) \norm{w}_{L^\infty(\overline{\Omega}^{\rho_0} \times [0,\frac{7}{8}])}$. Also, Theorem \ref{thm:pHolder-cpt-embedding} lower-bounds the left-hand side by the $C^{2,1}$ norm, so overall we get
    \begin{align} \label{eq:schauder001}
        \norm{w}_{C^{2,1}(\overline{\Omega}^{\rho_0}_{\tau} \times [\tau,\frac{7}{8}]; \Psi_\rho^*\hat{g}^{(\rho)}(0))} \leq C \norm{w}_{L^\infty(\overline{\Omega}^{\rho_0} \times [0,\frac{7}{8}])}.
    \end{align}
    Unraveling the definition of $w$ and using the maximum principle, we have
    \begin{align} \label{eq:schauderRHS}
        \norm{w}_{L^\infty(\overline{\Omega}^{\rho_0} \times [0,\frac{7}{8}])} &= \norm{\Psi_\rho^*\hat{u}^{(\rho)}}_{L^\infty(\overline{\Omega}^{\rho_0} \times [0,\frac{7}{8}])} = \sup_{(y,t) \in \overline{\Omega}^{\rho} \times [0,\frac{7}{8}]} |u(\Phi_{\rho t}(y))| = \sup_{\{r=\rho\}} |u| = \sup_{\overline{B}_\rho} |u|.
    \end{align}
    Likewise,
    \begin{align}
        \norm{\del_t w}_{L^\infty(\overline{\Omega}^{\rho_0}_{\tau} \times [\tau,\frac{7}{8}])} &= \norm{\Psi_\rho^* \del_t \hat{u}^{(\rho)}}_{L^\infty(\overline{\Omega}^{\rho_0}_{\tau} \times [\tau,\frac{7}{8}])} = \sup_{(y,t) \in \overline{\Omega}^{\rho}_{\tau} \times [\tau,\frac{7}{8}]} |\del_t(u(\Phi_{\rho t}(y)))| \\
        &= \sup_{(y,t) \in \overline{\Omega}^{\rho}_{\tau} \times [\tau,\frac{7}{8}]} \rho \left| \inner{\nabla u(\Phi_{\rho t}(y))}{f'(\phi_{\rho t}(r(y))) \del_r} \right|, \label{eq:schauderLHS1}
    \end{align}
    then using Assumption \ref{assump:f} and Lemma \ref{lem:phi-bounds} we get
    \begin{equation}
        \norm{\del_t w}_{L^\infty(\overline{\Omega}^{\rho_0}_{\tau} \times [\tau,\frac{7}{8}])} \geq C^{-1} \rho \sup_{(y,t) \in \overline{\Omega}^{\rho}_{\tau} \times [\tau,\frac{7}{8}]} \left| \inner{\nabla u}{\nabla r}(\Phi_{\rho t}(y)) \right| \geq C^{-1} \rho \sup_{\{ \frac{1}{4}\rho \leq r \leq \left(1-2\tau\right)\rho\}} \left| \inner{\nabla u}{\nabla r} \right|. \label{eq:-235-}
    \end{equation}
    Similarly, one shows by carefully unraveling definitions that
    \begin{equation}
    \norm{w}_{C^2(\overline{\Omega}^{\rho_0}_{\tau} \times [\tau,\frac{7}{8}]; \Psi_\rho^* \hat{g}^{(\rho)}(0))} \geq C^{-1} \left( \sup_{\{\frac{1}{4}\rho \leq r \leq (1-2\tau)\rho\}} |u| + \sqrt{\rho} \sup_{\{\frac{1}{4}\rho \leq r \leq (1-2\tau)\rho\}} |\nabla u| + \rho \sup_{\{\frac{1}{4}\rho \leq r \leq (1-2\tau)\rho\}} |\nabla^2 u| \right). \label{eq:schauderLHS2}
    \end{equation}
    Adding \eqref{eq:schauderLHS1} and \eqref{eq:schauderLHS2}, we get
    \begin{align}
        \norm{w}_{C^{2,1}(\overline{\Omega}^{\rho_0}_{\tau} \times [\tau,\frac{7}{8}]; \Psi_\rho^* \hat{g}^{(\rho)}(0))} \geq C^{-1} \sup_{\{\frac{1}{4}\rho \leq r \leq (1-2\tau)\rho\}} \left( |u| + \sqrt{\rho} |\nabla u| + \rho \left| \inner{\nabla u}{\nabla r} \right| + \rho |\nabla^2 u| \right).
    \end{align}
    Substituting this and \eqref{eq:schauderRHS} into \eqref{eq:schauder001}, the theorem follows.
\end{proof}

\begin{proof}[Proof of Theorem \ref{thm:meanValue}]
    We observe that \eqref{eq:heat-v} and \eqref{eq:unif-pde-bds} hold for $w := \Psi_\rho^*\hat{u}^{(\rho)}$. By a local maximum principle, e.g. \cite{lieberman}*{Theorem 7.36}, for each $\tau \in (0,\frac{1}{2})$ there exists $C = C(\tau) > 0$ such that
    \begin{align}
        \norm{w}_{L^\infty(\overline{\Omega}^{\rho_0}_{\tau} \times [\tau,\frac{7}{8}])}^2 \leq C \norm{w}_{L^2(\overline{\Omega}^{\rho_0} \times [0,\frac{7}{8}]; \Psi_{\rho_0}^*\hat{g}^{(\rho_0)}(0))}^2. \label{eq:w-MVI}
    \end{align}
    By the definition of $w$, \eqref{eq:-235-}, and the maximum principle, we have
    \begin{align} \label{eq:MVI-LHS}
        \norm{w}_{L^\infty(\overline{\Omega}^{\rho_0}_{\tau} \times [\tau,\frac{7}{8}])} = \sup_{(y,t) \in \overline{\Omega}^{\rho_0}_{\tau} \times [\tau,\frac{7}{8}]} |u(\Phi_{\rho t}(y))| \geq \sup_{\{\frac{1}{4}\rho \leq r \leq (1-2\tau)\rho \}} |u| = \sup_{\{r=(1-2\tau)\rho\}} |u|.
    \end{align}
    Meanwhile, using the uniform equivalence of the metrics $\Psi_\rho^*\hat{g}^{(\rho)}(t)$ from Lemma \ref{lem:unif-ctrl}, it follows that
    \begin{align}
        \norm{w}_{L^2(\overline{\Omega}^{\rho_0} \times [0,\frac{7}{8}]; \Psi_{\rho_0}^*\hat{g}^{(\rho_0)}(0))}^2 &\leq C \int_0^{7/8} \int_{{\overline{\Omega}^{\rho_0}}} w(x,t)^2 \, \dvol_{\Psi_\rho^*\hat{g}^{(\rho)}(t)}(y) \, dt.
    \end{align}
    Substituting in the definition of $w$, then using the coarea formula and Fubini's theorem, we continue estimating
    \begin{align}
        \norm{w}_{L^2(\overline{\Omega}^{\rho_0} \times [0,\frac{7}{8}]; \Psi_{\rho_0}^*\hat{g}^{(\rho_0)}(0))}^2
        &\leq C\rho^{-\frac{n}{2}} \int_{0}^{7/8} \int_{\overline{\Omega}^{\rho}} u(\Phi_{\rho t}(x))^2 \, \dvol_{\Phi_{\rho t}^*g}(x) \, dt \\
        &= C\rho^{-\frac{n}{2}} \int_0^{7/8} \int_{\rho-\sqrt{\rho}}^{\rho} \left( \int_{\{r=s\}} \frac{u(\Phi_{\rho t}(x))}{|\nabla^{\Phi_{\rho t}^*g} r|_{\Phi_{\rho t}^*g}(x)} \, \dvol_{\Phi_{\rho t}^*g}(x) \right) \, ds \, dt \\
        &= C\rho^{-\frac{n}{2}} \int_{\rho-\sqrt{\rho}}^{\rho} \int_{0}^{7/8} \left( \int_{\{r=s\}} \frac{u(\Phi_{\rho t}(x))^2}{|\nabla(r \circ \Phi_{\rho t}^{-1})|(\Phi_{\rho t}(x))} \, \dvol_{\Phi_{\rho t}^*g}(x) \right) \, dt \, ds.
    \end{align}
    Since $\Phi_{\rho t}(\{r=s\}) = \{r = \phi_{\rho t}(s)\}$, it follows that
    \begin{align}
        \norm{w}_{L^2(\overline{\Omega}^{\rho_0} \times [0,\frac{7}{8}]; \Psi_{\rho_0}^*\hat{g}^{(\rho_0)}(0))}^2 &\leq C\rho^{-\frac{n}{2}} \int_{\rho-\sqrt{\rho}}^{\rho} \int_{0}^{7/8} \left( \int_{\{r=\phi_{\rho t}(s)\}} \frac{u(z)^2}{|\nabla(r \circ \Phi_{\rho t}^{-1})|(z)} \, \dvol_g(z) \right) \, dt \, ds \\
        &\stackrel{\tau = \rho t}{=} C\rho^{-\frac{n}{2}-1} \int_{\rho-\sqrt{\rho}}^\rho \int_0^{\frac{7}{8}\rho} \left( \int_{\{r = \phi_{\tau}(s)\}} \frac{u^2}{|\nabla(r \circ \Phi_{-\tau})|} \, \dvol_g \right) \, d\tau \, ds \\
        &\leq C\rho^{-\frac{n}{2}-1} \int_{\rho-\sqrt{\rho}}^{\rho} \int_0^{\frac{7}{8}\rho} \left( \int_{\{r = \phi_{\tau}(s)\}} u^2 \, \dvol_g \right) \, d\tau \, ds. \label{eq:10958}
    \end{align}
    The final line uses that for all $s \in [\rho-\sqrt{\rho}, \rho]$, $\tau \in [0,\frac{7}{8}\rho]$, and $x \in \{r=\phi_\tau(s)\}$, one has $(r \circ \Phi_{-\tau})(x) \in [\rho-\sqrt{\rho},\rho]$ and so $|\nabla (r \circ \Phi_{-\tau})|(x) \geq C^{-1} > 0$.
    Substituting $\zeta = \zeta(\tau) = \phi_{\tau}(s)$ in \eqref{eq:10958},
    then using Assumption \ref{assump:f}, the definition of $I_u$, and Lemma \ref{lem:phi-bounds} (which gives $\phi_{7\rho/8}(s) \geq \frac{1}{32}\rho$), we get
    \begin{equation}
        \norm{w}_{L^2(\Omega^{\rho_0}_8 \times (-8,0]; \psi_\rho^*\hat{g}^{(\rho)}(0))}^2
        \leq C\rho^{-\frac{n}{2}-1} \int_{\rho-\sqrt{\rho}}^{\rho} \int_{\phi_{7\rho/8}(s)}^s \zeta^{\frac{n-1}{2}} I_u(\zeta)  \, d\zeta \, ds \leq C\rho^{-\frac{n}{2}-\frac{1}{2}} \int_{\frac{1}{32}\rho}^\rho \zeta^{\frac{n-1}{2}} I_u(\zeta) \, d\zeta. \label{eq:MVI-RHS}
    \end{equation}
    The theorem follows by substituting \eqref{eq:MVI-LHS} and \eqref{eq:MVI-RHS} into \eqref{eq:w-MVI}.
\end{proof}

\bibliography{Refs}
\end{document}